\documentclass[10pt]{article}
\usepackage{amssymb}
\usepackage{graphicx}
\usepackage{xcolor} 
\usepackage{tensor}
\usepackage{fullpage} 
\usepackage{amsmath}
\usepackage{amsthm}
\usepackage{mathtools}
\usepackage{verbatim}
\usepackage{enumitem}
\usepackage{hyperref}
\setlist[enumerate]{leftmargin=1.5em}
\setlist[itemize]{leftmargin=1.5em}

\definecolor{green}{rgb}{0,0.8,0} 

\newtheorem{theorem}{Theorem}[section]

\newtheorem{lemma}[theorem]{Lemma}

\newtheorem{innercustomthm}{Theorem}
\newenvironment{customthm}[1]
{\renewcommand\theinnercustomthm{#1}\innercustomthm}
{\endinnercustomthm}
\theoremstyle{definition}
\newtheorem{definition}[theorem]{Definition}
\newtheorem{example}[theorem]{Example}
\theoremstyle{remark}
\newtheorem{remark}[theorem]{Remark}

\numberwithin{equation}{section}

\providecommand{\MR}{\relax\ifhmode\unskip\space\fi MR }

\providecommand{\href}[2]{#2}
\newcommand{\nrm}[1]{\Vert#1\Vert}
\newcommand{\abs}[1]{\vert#1\vert}

\newcommand{\tld}[1]{\widetilde{#1}}

\newcommand{\nnrm}[1]{{\vert\kern-0.25ex\vert\kern-0.25ex\vert #1 
    \vert\kern-0.25ex\vert\kern-0.25ex\vert}}

\newcommand{\supp}{{\mathrm{supp}}\,}

\renewcommand{\Re}{\mathrm{Re}}
\renewcommand{\Im}{\mathrm{Im}}

\newcommand{\rd}{\partial}
\newcommand{\nb}{\nabla}

\newcommand{\alp}{\alpha}
\newcommand{\bt}{\beta}

\newcommand{\Gmm}{\Gamma}
\newcommand{\dlt}{\delta}
\newcommand{\Dlt}{\Delta}
\newcommand{\eps}{\epsilon}

\newcommand{\lmb}{\lambda}

\newcommand{\tht}{\theta}

\newcommand{\omg}{\omega}
\newcommand{\Omg}{\Omega}

\newcommand{\zt}{\zeta}

\begin{document}
\title{Well-posedness Threshold of the 2D Euler Equations in H\"older Spaces on  Infinite Sectors}
\author{Jaewan Choi\thanks{Department of Mathematical Sciences, Seoul National University. E-mail: streetcroxx@snu.ac.kr} }
\date{\today}
\maketitle
\begin{abstract}
	This study aims to obtain existence and uniqueness results for the incompressible Euler equations on two-dimensional infinite sectors of acute angle $\beta$. Our main interest is the evolution of vorticity in H\"older-type spaces. The unboundedness and the corner of the domain create two difficulties in establishing well-posedness in $C^{k,\alpha}$. First, without a decay assumption at infinity, we cannot expect H\"older regularity for every $\beta$. Second, the corner causes the Schauder-type estimate to fail at sufficiently high regularity and limits the regularity of the velocity.

    To address these difficulties, we introduce weighted H\"older spaces, which serve as an alternative to the standard $C^{k,\alpha}$. We establish well-posedness in these weighted spaces for every acute angle $\bt$ and every H\"older exponent, without imposing any decay assumption at infinity. These spaces include scale-invariant functions, making the result applicable to the study of their regularity and dynamics. We present an application that describes the long-time dynamics of nondecaying scale-invariant data.

    We then apply the weighted Hölder well-posedness theory to the $C^{k,\alpha}$ spaces. Unlike weighted H\"older regularity, the condition $k+\alp = \pi/\bt - 1$ acts as the $C^{k,\alp}$ well-posedness threshold. Above this threshold, there exist initial data that exhibit instantaneous loss of regularity, and to the best of our knowledge, this case has not been previously studied. We further show that the problem is well-posed at the regularity threshold if and only if $\pi/\beta$ is not an integer. To establish well-posedness in $C^{k,\alp}$, we assume that the initial vorticity belongs to a weighted Hölder space, rather than imposing a decay condition at infinity. This class includes compactly supported initial vorticities.
\end{abstract}

\renewcommand{\thefootnote}{\fnsymbol{footnote}}
\renewcommand{\thefootnote}{\arabic{footnote}}

\tableofcontents

\section{Introduction}

We consider the incompressible Euler equations on an infinite sector domain $\Omg_{\bt} := \{(r\cos\tht, r\sin\tht)\in\mathbb{R}^2: r>0,\ 0<\tht<\bt\}$:
\begin{equation}\label{eq:Euler}
    \left\{
        \begin{aligned}
          \partial_t u+ u \cdot \nabla u + \nabla p &= 0,\\
            \nabla \cdot u &= 0,
        \end{aligned}
    \right.
\end{equation}
where $u(t,x)$ is the velocity field, $p(t,x)$ is the pressure function, and $\bt\in (0,2\pi)$ is the given angle. 

We impose the initial and boundary conditions
\begin{equation*}
    u(0,\cdot) = u_0(\cdot), \quad u\cdot n|_{\rd\Omg_{\bt}}=0,
\end{equation*}
where $n$ is the outer unit normal on $\rd\Omg_{\bt}$. It is often useful to consider the vorticity form of the Euler equations. Let $\omg := \nabla^{\perp}\cdot u=\rd_1 u_2 - \rd_2 u_1 $ be the vorticity. In two-dimensional domains, taking the curl of \eqref{eq:Euler} implies
\begin{equation}\label{eq:Euler_vorticity}
    \partial_t \omg + u\cdot \nb \omg = 0.
\end{equation}
In other words, the vorticity $\omg$ is transported by $u$ in $\Omg_{\bt}$.

\subsection{Vorticity without decay at infinity}\label{Bounded Vorticity}

In an unbounded domain such as $\mathbb{R}^2$, classical results often assume that the vorticity decays at infinity. One of the celebrated results is the well-posedness of the $L^1\cap L^{\infty}$ vorticity in $\mathbb{R}^2$ \cite{Y1}. If the vorticity does not have decay at infinity, then the integral in the Biot-Savart law
\begin{equation}\label{eq:BS_2D}
    u(x) = \frac{1}{2\pi}\int_{\mathbb{R}^2} \frac{(x-y)^{\perp}}{|x-y|^2} \omg(y) dA(y)
\end{equation}
need not be well-defined.

The spatial scaling symmetry of the Euler equations implies that if $\omg(t,x)$ is a solution of the Euler equations in $\mathbb{R}^2$, then for every $\lmb>0$, the function $\omg_{\lmb}$ defined by
\begin{equation*}
    \omg_{\lmb}(t,x) := \omg(t,\lmb x), \qquad t>0,\ x\in\mathbb{R}^2
\end{equation*} 
is also a solution of the Euler equations. Since the Euler equations are preserved by scaling in space, it is natural to consider the vorticity data with scale invariance.
\begin{definition}
    Suppose that a function $f$ defined on $\Omg_{\bt}$ satisfies $f(rx) = f(x)$ for all $r>0$ and $x\in\Omg_{\bt}$. Then the function $f$ is said to be \textit{scale-invariant} in $\Omg_{\bt}$.
\end{definition}

A scale-invariant vorticity function is defined on an unbounded domain, but it does not decay at infinity. These two features of scale-invariant data prevent us from applying classical well-posedness results, such as the Yudovich theory \cite{Y1}.

One way to overcome this far-field issue is to restrict $\beta$ to a smaller range $0<\bt<\bt_0$ for a suitable $\bt_0$. The Green's function on $\Omg_{\bt}$, introduced in Section \ref{Preliminary}, has improved far-field decay as $\bt$ decreases; see Lemma \ref{lem:ker_bound} for the details. Before addressing H\"older regularity, we first show that, when $\bt<\pi/2$, no decay assumption at infinity is required for the existence and uniqueness of an Euler flow with $L^{\infty}$ vorticity.
\begin{customthm}{A}[$L^{\infty}$ well-posedness]\label{thm:weak-wellposed}
    Suppose that $0<\bt<\pi/2$, and $\omg_0\in L^{\infty}(\Omg_{\bt})$. Then there exists a unique weak solution $\omg\in L^{\infty}([0,\infty); L^{\infty}(\Omg_{\bt}))$ to the Euler equations with initial vorticity $\omg_0$.
\end{customthm}
\begin{remark}
    In Section \ref{weak-wellposed}, we adopt the definition of weak solutions to the Euler equations given in \cite{MP}. This is motivated by the transport property in \eqref{eq:Euler_vorticity}.
\end{remark}
We briefly outline the proof. For the uniqueness proof, we fix $\dlt \in (0,\pi/\bt-2)$, and let $\Phi$ and $\tld{\Phi}$ be the flow maps associated with two solutions. The key identity is
\begin{equation}\label{eq:flow_uniqueness}
    \int_{\Omg_{\bt}}  \frac{|\Phi(t,x) - \tld{\Phi}(t,x)|}{|x|} \frac{|\omg_0(x)|}{1+|x|^{2+\dlt}} dx = 0.
\end{equation}
The weight $1/(1+|x|^{2+\dlt})$ is chosen because it is integrable in $\mathbb{R}^2$, and when the time interval is fixed,
\begin{equation*}
\frac{|\Phi(t,x) - \tld{\Phi}(t,x)|}{|x|}
\end{equation*}
is uniformly bounded in $x$. For the existence proof, we replace $|\Phi(t,x) - \tld{\Phi}(t,x)|$ with $|\Phi^{(n)}(t,x) - \Phi^{(n-1)}(t,x)|$ in \eqref{eq:flow_uniqueness}, where $\{\Phi^{(n)}\}$ is the sequence of approximate flow maps. This is necessary for the convergence of the sequence of approximate velocities.

Throughout this paper, we assume that $\bt<\pi/2$ so that no decay assumptions of vorticity at infinity are required. If $\bt\ge \pi/2$, the condition $\omg\in L^{\infty}$ does not guarantee that the velocity is well-defined by the Biot-Savart law \eqref{eq:Biot-Savart}.

\subsubsection*{Remark on the scale-invariance}

As noted above, Theorem \ref{thm:weak-wellposed} motivates the study of scale-invariant data. Furthermore, one can think of discretely scale-invariant vorticity data. A function $f$ defined on $\Omg_{\bt}$ is said to be \textit{discretely scale-invariant} if there exists $\lambda>1$ such that
\begin{equation*}
    f(x) = f(\lambda x),\qquad x\in\Omg_{\bt}.
\end{equation*}
Such vorticity data do not decay at infinity in general. In the following subsections, we discuss some H\"older regularity results that can be applicable to these types of data; see Theorems \ref{thm:GWPHolder} and \ref{thm:LWPHolder2} for details.

The dynamics of such data can also be studied. For example, we investigate the following long-time behavior of nonnegative scale-invariant vorticity in Section \ref{Long-time}.
\begin{theorem}[Long-time behavior of scale-invariant data]\label{thm:Long-Scale}
    Suppose that $0<\bt<\pi/2$, $\omg_0\in L^{\infty}(\Omg_{\bt})$ is scale-invariant, nonnegative, and the limit
    \begin{equation*}
        \lim_{x\rightarrow (\cos\bt,\sin\bt)} \omg_0(x) =: \omg_{0,\bt}
    \end{equation*}
     exists. Then, in $I_{\bt} := \{(\cos\tht, \sin\tht)\in \mathbb{R}^2: 0<\tht<\bt\}$, the solution $\omg$ defined by Theorem \ref{thm:weak-wellposed} satisfies 
     \begin{equation*}
        \lim_{t\rightarrow\infty} \|\omg(t,\cdot) - \omg_{0,\bt} \|_{L^1(I_{\bt})}=0.
     \end{equation*}
     In particular, $\omg(t,\cdot)$ converges to a constant function a.e. as $t\rightarrow\infty$.
\end{theorem}

\subsection{Well-posedness in weighted H\"older spaces}

In this paper, we consider vorticity data with H\"older regularity defined on $\Omg_{\bt}$. We briefly recall the definition of the $C^{k,\alp}(\overline{\Omg_{\bt}})$ spaces; see Chapter 4 of \cite{GT} for the details.
\begin{definition}
    Let $k\in\mathbb{N}_0$, and $0<\alp\le 1$. A function $f$ is said to be \textit{uniformly H\"older continuous with exponent $\alp$ in $\Omg_{\bt}$} if the quantity
    \begin{equation*}
        [f]_{C^{0,\alp}(\overline{\Omg_{\bt}})} := \sup_{x,y\in \Omg_{\bt} (x\neq y)} \frac{|f(x) - f(y)|}{|x-y|^{\alp}}
    \end{equation*}
    is finite. The \textit{H\"older space} $C^{k,\alp}(\overline{\Omg_{\bt}})$ is the collection of functions $f\in C^k(\overline{\Omg_{\bt}})$ whose $k$-th partial derivatives are uniformly H\"older continuous with exponent $\alp$ in $\Omg_{\bt}$. These are Banach spaces with the norms
    \begin{equation*}
        \|f\|_{C^{k,\alp}(\overline{\Omg_{\bt}})} := \|f\|_{C^k(\overline{\Omg_{\bt}})} + [f]_{C^{k,\alp}(\overline{\Omg_{\bt}})} := \sup_{0\le|m|\le k} \sup_{x\in \Omg_{\bt}}  \left|\nb^m f(x)\right| + \sup_{|m|=k}[\nb^m f]_{C^{0,\alp}(\overline{\Omg_{\bt}})}.
    \end{equation*}
    We define $C^{k,0} := C^k$, and $C^{k,\alp}(\Omg_{\bt})$ be the collection of $f$ such that $f\in C^{k,\alp}(K)$ for all compact $K\subset \Omg_{\bt}$. 
\end{definition}

The classical well-posedness theory for $C^{k,\alp}$ solutions on $\mathbb{R}^2$ is well known. Let $k\in\mathbb{N}_0$ and $0<\alp\le 1$. Suppose that $\omg_0 = \nb^{\perp}\cdot u_0$ belongs to $C^{k,\alp}(\mathbb{R}^2)$ and has compact support. Then for every $T>0$, there exists a unique solution
\begin{equation*}
    \omg\in L^{\infty}([0,T]; C^{k,\alp}(\mathbb{R}^2))
\end{equation*}
to the Euler equations \cite{MB}. Thus, compactly supported initial vorticity with Hölder regularity gives rise to a unique global solution that preserves its Hölder regularity.

In contrast, the Hölder regularity is not preserved in the infinite sector $\Omg_{\bt}$. Let $R>0$ satisfy $\supp \omg \subset B_R(0)$. The existence of a corner at the origin causes the failure of the Schauder-type estimate
\begin{equation}\label{eq:Schauder}
    [\nb u]_{C^{k,\alpha}(\overline{\Omg_{\bt}}\cap B_R(0))} \lesssim_{k,\alp,\bt} \|\omg\|_{C^{k,\alp}(\overline{\Omg_{\bt}}\cap B_R(0))}
\end{equation}
when $k+\alpha$ is sufficiently large. The range of $k$ and $\alp$ is given in Section \ref{Holder-wellposed2}. The failure of \eqref{eq:Schauder} prevents us from expecting analogous higher H\"older regularity in $\Omega_\beta$. For example, there exists a singular harmonic function near the corner; see Example \ref{ex:singular_harmonic}.

As noted in the previous subsection, throughout this paper, we assume that $\bt<\pi/2$ so that no decay assumptions of vorticity at infinity are required. However, this assumption is not enough for the H\"older regularity in $\Omg_{\bt}$. In the whole plane $\mathbb{R}^2$, vorticity $\omg$ with compact support and H\"older continuity satisfies the type of inequality
\begin{equation}\label{eq:velocity_Lip}
    \|\nabla u\|_{L^{\infty}(\mathbb{R}^2)} \lesssim \|\omg\|_{L^{\infty}(\mathbb{R}^2)} \log \frac{10\|\omg\|_{C^{0,\alp}(\mathbb{R}^2)}}{\|\omg\|_{L^{\infty}(\mathbb{R}^2)}}
\end{equation}
so that the velocity is globally Lipschitz (\cite{KS}, \cite{MB}). The logarithmic type of inequality \eqref{eq:velocity_Lip} for $\|\nb u\|_{L^{\infty}}$ enables us to establish the global well-posedness \cite{BKM}; see \cite{MB} for the details. Without a decay assumption at infinity, \eqref{eq:velocity_Lip} does not hold for all $0<\alp\le 1$ and $0<\bt<\pi/2$. This motivates us to find function spaces that are related to H\"older continuity and yield an analogue of \eqref{eq:velocity_Lip}.

Moreover, we want these function spaces to include some sufficiently smooth scale-invariant data. However, it is trivial that if a scale-invariant function is not a constant, then it cannot belong to any H\"older spaces. Instead, we can make an alternative question:
\begin{center}
\fbox{\begin{minipage}{0.9\textwidth}
Suppose that $0<\alp<1$, $0<\bt<\pi/2$, and the initial vorticity $\omg_0$ defined on $\Omg_\bt$ is scale-invariant. \\
If $\omg_0\in C^{0,\alp}$ in $\overline{\Omg_\bt}\cap\{|x|=1\}$, then is it true that $\omg(t,\cdot)\in C^{0,\alp}$ in $\overline{\Omg_\bt}\cap\{|x|=1\}$ for all $t>0$?
\end{minipage}}
\end{center}
Based on the preceding discussion, we expect the function spaces on $\Omg_{\bt}$ to satisfy the following properties:
\begin{itemize}
    \item They are subspaces of $L^{\infty}(\Omg_{\bt})$.
    \item Unlike H\"older spaces, they induce an inequality which is an analogue of \eqref{eq:velocity_Lip}.
    \item No decay assumptions at infinity are required.
    \item They contain scale-invariant functions that are H\"older continuous on $\Omg_{\bt}\cap \{|x|=1\}$. Consequently, well-posedness in these function spaces also solves the above one-dimensional problem. 
\end{itemize}
We focus on the function spaces $\mathring{C}^{k,\alp}(\overline{\Omg_{\bt}})$ (\cite{EJB}, \cite{EJS}).
\begin{definition}
    Let $\mathring{C}^{0,\alp}(\overline{\Omg_{\bt}})$ be the collection of functions $f$ defined on $\Omg_{\bt}$ with
    \begin{align*}
        \|f\|_{\mathring{C}^{0,\alp}(\overline{\Omg_{\bt}})} := \|f\|_{L^{\infty}(\Omg_{\bt})} + [f]_{\mathring{C}^{0,\alp}(\overline{\Omg_{\bt}})}:= \|f\|_{L^{\infty}(\Omg_{\bt})} + \sup_{x,y\in\Omg_{\bt}}\frac{\left||x|^{\alp} f(x) - |y|^{\alp} f(y)\right|}{\left|x-y\right|^{\alp}} < \infty,\qquad \alp\in (0,1].
    \end{align*}
    In general, we define the norms of the function spaces $\mathring{C}^{k,\alp}(\overline{\Omg_{\bt}})$ for $k\in \mathbb{N}$ and $0<\alp\le 1$ by 
    \begin{equation*}
        \|f\|_{\mathring{C}^{k,\alp}(\overline{\Omg_{\bt}})} := \|f\|_{\mathring{C}^k(\overline{\Omg_{\bt}})} + \sup_{x,y\in\Omg_{\bt}}\frac{\left||x|^{k+\alp} \nb^k f(x) - |y|^{k+\alp} \nb^k f(y)\right|}{\left|x-y\right|^{\alp}},
    \end{equation*}
    where
    \begin{equation*}
        \|f\|_{\mathring{C}^k(\overline{\Omg_{\bt}})} := \sup_{0\le |m|\le k} \sup_{x\in \Omg_{\bt}} \left(|x|^m \left|\nb^m f(x)\right|\right).
    \end{equation*}
    We define $\mathring{C}^{k,0} := \mathring{C}^k$, and $\mathring{C}^{k,\alp}(\Omg_{\bt})$ be the collection of $f$ such that $f\in \mathring{C}^{k,\alp}(K)$ for all compact $K\subset \Omg_{\bt}$. 
\end{definition}

\begin{customthm}{B}[$\mathring{C}^{k,\alp}$ well-posedness]\label{thm:GWPHolder}
    Suppose that $0<\bt<\pi/2$. Let $k\in\mathbb{N}_0$, $0<\alp\le 1$, and $\omg_0\in \mathring{C}^{k,\alp}(\overline{\Omg_{\bt}})$. Then, for any $T>0$, the unique weak solution to the Euler equations in Theorem \ref{thm:weak-wellposed} satisfies
    \begin{align*}
        \omg \in L^{\infty}([0,T); \mathring{C}^{k,\alp}(\overline{\Omg_{\bt}}))
    \end{align*}
    with initial data $\omg_0$. The norm $\|\omg(t,\cdot)\|_{\mathring{C}^{k,\alp}}$ grows at most double exponentially in time.
\end{customthm}

\begin{remark}
    Let $k\in\mathbb{N}_0$ and $\alp\in (0,1]$. We note some properties of $\mathring{C}^{k,\alp}$ that are similar to the standard H\"older spaces.
    \begin{itemize}
        \item If $f\in \mathring{C}^{k+1}$, then $\|f\|_{\mathring{C}^{k,1}} \lesssim \|f\|_{\mathring{C}^{k+1}}$. 
        \item Let $m\in\mathbb{N}_0$ with $m\le k$. If $g\in \mathring{C}^{k,\alp}$, then $\|g\|_{\mathring{C}^{m,\alp}} \lesssim \|g\|_{\mathring{C}^{k,\alp}}$.
    \end{itemize}
\end{remark}

We note that in Theorem \ref{thm:GWPHolder}, no decay assumption of $\omg_0$ at infinity is needed. Suppose that $0<\alp<1$, $0<\bt<\pi/2$, and $\omg \in \mathring{C}^{0,\alp}(\overline{\Omg_{\bt}})$. Lemma \ref{lem:gradu_infty} states that an analogue of \eqref{eq:velocity_Lip}
\begin{equation}\label{eq:gradu_infty2}
    \|\nb u\|_{L^{\infty}(\Omg_{\bt})} \lesssim_{\alp,\bt} \|\omg\|_{L^{\infty}(\Omg_{\bt})}\log\frac{c\|\omg\|_{\mathring{C}^{0,\alp}(\overline{\Omg_{\bt}})}}{\|\omg\|_{L^{\infty}(\Omg_{\bt})}}
\end{equation}
holds for some constant $c>3$. In other words, the condition $\omg \in \mathring{C}^{0,\alp}$ yields a Lipschitz velocity, and we use \eqref{eq:gradu_infty2} to prove our well-posedness results.

Since the H\"older spaces $C^{k,\alp}$ do not satisfy \eqref{eq:Schauder}
\begin{equation*}
    [\nb u]_{C^{k,\alpha}(\overline{\Omg_{\bt}}\cap B_R(0))} \lesssim_{k,\alp,\bt} \|\omg\|_{C^{k,\alp}(\overline{\Omg_{\bt}}\cap B_R(0))}
\end{equation*}
when $k+\alp$ is large enough, we cannot guarantee sufficient H\"older regularity of $\nb u$. However, when working in the spaces $\mathring{C}^{k,\alp}$, we do not give restrictions on the range of $k$ and $\alp$. Suppose that $k\in\mathbb{N}_0$, $0<\alp<1$, $0<\bt<\pi/2$, and $\omg\in \mathring{C}^{k,\alp}(\overline{\Omg_{\bt}})$. Lemma \ref{lem:gradu_bound} states that $\nb u\in \mathring{C}^{k,\alp}$ and
\begin{equation}\label{eq:mathring_C_k_alp}
    [\nb u]_{\mathring{C}^{k,\alp}(\overline{\Omg_{\bt}})} \lesssim_{k,\alp,\bt} \|\omg\|_{\mathring{C}^{k,\alp}(\overline{\Omg_{\bt}})}
\end{equation}
without additional restriction on $k$ and $\alp$.

The main difficulty in the proof of \eqref{eq:mathring_C_k_alp} is that we cannot expect the mean zero property of the gradient of the Biot-Savart kernel. However, we can obtain an alternative cancellation property of the kernel. This enables us to deal with the singular integrals that appear in the differentiation of the Biot-Savart law \eqref{eq:Biot-Savart}; see Lemmas \ref{lem:HalfGradKer} and \ref{lem:gradu_pv2} for the details.

\begin{remark}
    Theorem \ref{thm:GWPHolder} motivates several directions for further research. For example, in the study of various hydrodynamic models, e.g., Euler, Boussinesq, and SQG equations, the function space $\mathring{C}^{k,\alp}(\overline{\Omg_{\bt}})$ may serve as an alternative to the H\"older space $C^{k,\alp}$.
\end{remark}

\subsection{Applications to H\"older regularity}

As we noted above, without a decay assumption at infinity, we cannot expect in general that the estimate \eqref{eq:velocity_Lip} holds. However, if we assume that the initial vorticity belongs to $\mathring{C}^{0,\tld{\alp}}(\overline{\Omg_{\bt}})$ for some $0<\tld{\alp}<1$, then the inequality \eqref{eq:gradu_infty2} holds. Consequently, we ensure the Lipschitz property of the velocity. We investigate H\"older regularity under the $\mathring{C}^{0,\tld{\alp}}$ assumption.

\subsubsection*{Well-posedness result}

The following main result shows that the $\mathring{C}^{0,\tld{\alp}}$ condition is sufficient to overcome the far-field difficulties in establishing well-posedness in $C^{k,\alp}$. However, since the estimate \eqref{eq:Schauder} fails when $k+\alp$ is sufficiently large, we restrict our well-posedness result for lower H\"older regularity.
\begin{customthm}{C}[$C^{k,\alp}$ well-posedness under $\mathring{C}^{0,\tilde{\alp}}$]\label{thm:LWPHolder2}
    Suppose that $0<\bt<\pi/2$, $k\in\mathbb{N}_0$, $0 < \alp \le 1$ with
    \begin{equation}\label{eq:condition_k,a}
    \left\{
        \begin{aligned}
            k+\alp &\le \pi/\bt - 1, \qquad \pi/\bt\notin\mathbb{N},\\
            k+\alp &< \pi/\bt - 1, \qquad \pi/\bt\in\mathbb{N},
        \end{aligned}
    \right.
    \end{equation}
    $0<\tld{\alp}<1$, and $\omg_0\in C^{k,\alp}\cap \mathring{C}^{0,\tld{\alp}}(\overline{\Omg_{\bt}})$. Then, for any $T>0$, there exists a unique solution 
    \begin{align*}
        \omg \in L^{\infty}([0,T); C^{k,\alp}(\overline{\Omg_{\bt}}))
    \end{align*}
    to the Euler equations with initial data $\omg_0$. The norm $\|\omg(t,\cdot)\|_{C^{k,\alp}}$ grows at most double exponentially in time.
\end{customthm} 

The details about the bound for $\|\omg(t,\cdot)\|_{C^{k,\alp}}$ are given in Section \ref{Holder-wellposed2}. To prove Theorem \ref{thm:LWPHolder2}, we first bound $[\nb u]_{C^{0,\alp}}$ under an additional condition $0<\alp<\pi/\bt - 2$. This condition comes from the singular property of the gradient of the Biot-Savart kernel near $0$; see Lemma \ref{lem:gradu_bound2} for the details. This method can be extended to the higher regularity case with the aid of Lemma \ref{lem:HalfGradKer}. The critical case $k+\alp = \pi/\bt - 1$ is handled by an interpolation argument. Although the operator $\omg \mapsto \nb u$ is not bounded on $C^{k-1,\alp}$, the stronger assumption $\omg\in C^{k,\alp}$ helps us to prove by interpolation that $\|\nb u\|_{C^{k-1,\alp}}$ is double-exponentially bounded in time.

\begin{remark}
    Let $\alp$, $\tld{\alp}\in (0,1]$, and $k\in\mathbb{N}_0$. We note some remarks about the relation between $C^{k,\alp}$ and $\mathring{C}^{k,\tld{\alp}}$ spaces.
    \begin{itemize}
        \item The spaces $C^{k,\alp}(\overline{\Omg_{\bt}})$ and $\mathring{C}^{k,\tld{\alp}}(\overline{\Omg_{\bt}})$ satisfy $C^{k,\alp}(\overline{\Omg_{\bt}})\nsubseteq \mathring{C}^{k,\tld{\alp}}(\overline{\Omg_{\bt}})$ and $C^{k,\alp}(\overline{\Omg_{\bt}})\nsupseteq \mathring{C}^{k,\tld{\alp}}(\overline{\Omg_{\bt}})$.
        \item If a $C^{k,\tld{\alp}}(\overline{\Omg_{\bt}})$ initial vorticity has compact support, then it belongs to $\mathring{C}^{0,\tld{\alp}}(\overline{\Omg_{\bt}})$. Therefore, assuming that the vorticity belongs to $\mathring{C}^{0,\tld{\alp}}$ is weaker than assuming compact support.
    \end{itemize}
\end{remark}

The condition \eqref{eq:condition_k,a} is essential for our well-posedness result. As we show in Section \ref{Holder-wellposed2}, the estimate \eqref{eq:Schauder} holds if and only if $k+\alp < \pi/\bt - 2$ and $\alp\neq 0$. The upper bound of $k+\alp$ is essential, and for example, there exists a harmonic function that is not infinitely differentiable at the corner.
\begin{example}
    Let $\bt = 2\pi/7$, and $\chi\in C_c^{\infty}(\overline{\Omg_{\bt}})$ be a radial cut-off function with $\chi\equiv 1$ in $\Omg_{\bt}\cap B_1(0)$. We define a function $\psi$ on $\Omg_{\bt}$ such that
    \begin{equation*}
        \psi(r\cos\tht, r\sin\tht) = r^{7/2} \sin\left(\frac{7\tht}{2}\right)\chi(r).
    \end{equation*}
    Then, $\omg := \Dlt \psi \equiv 0$ in a neighborhood of the origin. On the other hand, $u = \nb^{\perp} \psi$ is homogeneous of degree $5/2$ near the origin. Hence $\nabla u \notin C^{1,\alp}$ for every $\alp>1/2$. For the general theory of singular harmonic functions, we refer to \cite{G1}.
\end{example}
If $k+\alp < \pi/\bt - 1$ and the vorticity satisfies $\omg\in C^{k-1,\alp}\cap \mathring{C}^{0,\tld{\alp}}$, then Lemma \ref{lem:gradu_bound2} implies that $\nb u\in C^{k-1,\alp}$ in a neighborhood of the origin. Unlike \eqref{eq:mathring_C_k_alp}, since the function space $C^{k,\alp}$ is not scale-invariant, we cannot use Lemma \ref{lem:gradu_pv2} directly for the proof. Instead, we use the property that if the vorticity is a polynomial of order $n$ with $n + 2 < \pi/\bt$, then the velocity is also represented by a smooth polynomial (Lemma \ref{lem:Gbeta_nm}). This acts as a cancellation property of the singular integrals that appear in the proof of Lemma \ref{lem:gradu_bound2}.

The critical case $k+\alp = \pi/\bt - 1\notin\mathbb{N}$ is explained by an interpolation argument. For example, suppose that $\pi/3 < \bt < \pi/2$ so that $k=1$ and $\alp = \pi/\bt - 2$. Suppose further that $0<\tld{\alp}<1$, and $\omg\in C^{0,1}\cap \mathring{C}^{0,\tld{\alp}}(\overline{\Omg_{\bt}})$. Lemma \ref{lem:borderline} says that $\nb u\in C^{0,\pi/\bt - 2}$ and 
\begin{equation*}
    [\nb u]_{C^{0,\pi/\bt-2}} \lesssim_{\bt} \|\omg\|_{C^{0,\pi/\bt - 2}} + \|\omg\|_{C^{0,1}}^{\pi/\bt-2} \|\omg\|_{L^{\infty}}^{3-\pi/\bt}. 
\end{equation*}
The well-posedness in $C^{0,\pi/\bt-2}$ and $C^{0,1}$ implies the double-exponential bound of $[\nb u]_{C^{0,\pi/\bt-2}}$ in time, and this bound is used for the well-posedness in $C^{1,\pi/\bt - 2}$. This result follows from the decomposition of the Biot-Savart kernel into the most singular term at $0$ and the smooth remainder:
\begin{itemize}
    \item The most singular term prevents us from obtaining an estimate in terms of $\|\omg\|_{C^{0,\pi/\bt - 2}}$, because the relevant integral contains the divergent logarithmic factor $\int_2^{\infty} \rho^{-1} d\rho$. To avoid this logarithmic divergence, we need to impose a stronger regularity assumption, such as $\omg\in C^{0,1}$, to improve the decay in $\rho$.  
    \item Meanwhile, the smooth remainder yields a corresponding factor $\int_2^{\infty} \rho^{-m} d\rho$ ($m>1$) when we obtain an estimate in terms of $\|\omg\|_{C^{0,\pi/\bt - 2}}$.
\end{itemize}

\subsubsection*{Ill-posedness result}

It is natural to ask what happens when the condition \eqref{eq:condition_k,a} fails, and our final main result concerns this case. The presence of the corner at the origin can prevent $\nb u$ from having sufficient Hölder regularity, which in turn causes the solution to lose the desired regularity.
\begin{customthm}{D}[Instantaneous loss of $C^{k,\alp}$ regularity]\label{thm:ill-posed}
    Suppose that $0<\bt<\pi/2$, $k\in\mathbb{N}$, and $0 < \alp \le 1$ with
    \begin{equation}\label{eq:condition_k,a_2}
    \left\{
        \begin{aligned}
            k+\alp &> \pi/\bt - 1, \qquad \pi/\bt\notin\mathbb{N},\\
            k+\alp &\ge \pi/\bt - 1, \qquad \pi/\bt\in\mathbb{N}.
        \end{aligned}
    \right.
    \end{equation}
    There exists an initial vorticity $\omg_0\in C^{\infty}_c(\overline{\Omg_{\bt}})$ and $T>0$, $\dlt>0$ such that the solution $\omg$ to the Euler equations satisfies
    \begin{equation*}
        \|\omg(t,\cdot)\|_{C^{k,\alp}(\overline{\Omg_{\bt}}\cap B_{\dlt}(0))} =\infty
    \end{equation*}
    for all $0<t<T$.
\end{customthm}

The key to the proof is to decompose $\nb u$ into the most singular term at the origin and the smooth remainder. Since the formula of the Green's function on $\Omg_{\bt}$ is given, we calculate the singular term in Section \ref{Holder-wellposed2} explicitly.

We next decompose the flow map $\Phi$ and its inverse. For example, suppose that $\pi/\bt\notin\mathbb{N}$. Iterating the bootstrap argument yields
\begin{equation*}
    \Phi_t(x) := \Phi(t,x) := L_t(x) + H_t(x) + R_t'(x),
\end{equation*}
a decomposition analogous to that of $u$, where $L_t$ is smooth, $H_t$ is homogeneous of degree $\pi/\bt - 1$, and $|R_t'(x)| \lesssim |x|^{\pi/\bt-\alp}$. By iterating the bootstrap argument once more, we obtain a corresponding decomposition of the inverse flow map by
\begin{equation*}
    \Phi_t^{-1}(x) := L_t^{-1}(x) - M_t H_t(M_t x) + \tld{R}_t''(x),
\end{equation*}
where $M_t$ denotes the linear part of the map $L_t^{-1}$, and $|\tld{R}_t''(x)|\lesssim |x|^{\pi/\bt-\alp}$. If $ M_t H_t(M_t (\cdot)) \not\equiv 0$, this term gives the leading singular behavior near the origin. The identity $\omg(t,\cdot) = \omg_0\circ \Phi_t^{-1}$ explains how the regularity of $\omg(t,\cdot)$ is determined.

When $\pi/\bt$ is not an integer, the coefficient of the singular term is determined by the values of the vorticity throughout $\Omg_{\bt}$. In contrast, if $\pi/\bt$ is an integer, the coefficient is determined by $\nb^{\pi/\bt-2}\omg(0)$. Motivated by these properties, we conclude this section with examples illustrating Theorem \ref{thm:ill-posed}.

\begin{example}\label{ex:singular_harmonic}
    We note the examples of $\omg_0$ that prove Theorem \ref{thm:ill-posed} in Section \ref{Holder-wellposed2}.
    \begin{itemize}
        \item Suppose that $\bt = 2\pi/7$, and $\chi\in C_c^{\infty}(\overline{\Omg_{\bt}})$ is a radial cut-off function with $\chi(x)\equiv 1$ if $|x|<1$ and $\chi(x)\equiv 0$ if $|x|>2$. Then there exist $T>0$, a constant $A\in\mathbb{R}$, and a cut-off function $\eta$ satisfying $\eta\equiv 0$ in $B_1(0)$ such that
        \begin{equation*}
            \omg_0(x_1, x_2) = \chi(x_1, x_2) \left(x_1\cos\frac{\pi}{7} + x_2\sin\frac{\pi}{7}\right)+ A\eta(x_1,x_2)
        \end{equation*}
        satisfies $\omg_0\in C_c^{\infty}$, but $\|\omg(t,\cdot)\|_{C^{2,\alp}}=\infty$ for all $\alp>1/2$ and $0<t<T$.
        \item Suppose that $\bt = \pi/4$. Then there exists $T>0$ such that the initial data
        \begin{equation*}
            \omg_0(x_1,x_2) = \chi(x_1,x_2) \left((x_1 + x_1^2 - x_2^2)\cos\frac{\pi}{8} + (x_2 + 2x_1x_2) \sin\frac{\pi}{8}\right)
        \end{equation*}
        satisfies $\omg_0\in C_c^{\infty}$, but $\|\omg(t,\cdot)\|_{C^{2,1}}=\infty$ for all $0<t<T$.
    \end{itemize}
\end{example}

\subsection{Previous results}

\subsubsection*{2D Euler equations in domain with corners}

The $2D$ Euler equations in a corner domain are closely related to elliptic theory on domains with corner boundaries. Since $\nabla \cdot u=0$, there exists a stream function $\psi$ such that $u=\nabla^{\perp} \psi$. Since $\Dlt \psi = \omg$ and $\nabla u = \nabla\nabla^{\perp} \psi$, the relationship between $\omg$ and $\nb u$ can be studied through an elliptic problem with corner geometry.

The elliptic problem $\Dlt \psi = \omg$ on domains with corners has been studied extensively; see \cite{G1}, \cite{BK1}, \cite{KMR}. When the corner angle satisfies $\bt>\pi/2$, the velocity gradient $\nb u$ fails to belong to $L^p$ for some $p<\infty$. Consequently, the standard proof strategy for Yudovich's theorem cannot be applied directly. Uniform $L^p$ estimates for $\nb u$ for $p\in [1,\infty)$ play an essential role in the proof of Yudovich's theorem in $\mathbb{R}^d$ \cite{MB}. The works \cite{BPT}, \cite{co2} study $L^p$ estimates for $\nb u$ near corners. For $\pi/2<\bt<\pi$, only a few uniqueness results are known under additional restriction on $\omg_0$ \cite{AN2}.

As noted in the previous subsection, Schauder estimates for domains with corners are available only under certain restrictions. Due to this limitation, many previous studies of Euler equations in corner domains have focused primarily on uniqueness. Several papers address uniqueness for the $2D$ Euler equations in corner domains (\cite{co1}, \cite{co3}), and some also obtain results for domains with more severe singularities (\cite{HZ1}, \cite{HZ2}, \cite{AN1}).

Elgindi and Jeong \cite{EJB}, \cite{EJS} suggested the space $\mathring{C}^{0,\alp}$ to remove the restriction on $k$ and $\alp$. For the sector domain $\Omg_{\bt}$, Lemmas 3.5 and 3.6 in \cite{EJB} prove the $\mathring{C}^{0,\alp}$-estimate and $C^{0,\alp}$-estimate of $\nabla u$, respectively. Note that $|\nb u(x)| \lesssim \log(1+|x|) \|f\|_{L^{\infty}} + \|f\|_{C^{\alp}}$ in \cite{EJB} does not imply $\|\nb u\|_{L^{\infty}}<\infty$. However, the $\mathring{C}^{0,\alp}$ estimate avoids this limitation, making this function space a suitable framework for studying regularity without further restrictions on $\alp$. We exclude the case $\alp=0$ in view of the ill-posedness result of \cite{BL2}.

Itoh, Miura, and Yoneda (\cite{IMY1}, \cite{IMY2}) studied Lipschitz vorticity in domains with acute corners and in a square under an additional odd-symmetry assumption. Despite the presence of corners, the Lipschitz norm of the vorticity grows at most exponentially in time. This growth rate is sharp in light of the mechanism of Kiselev and \v{S}ver\'ak \cite{KS}.

\subsubsection*{Expansion of Yudovich theory}

Since we need the condition $\omg\in L^1(\mathbb{R}^2)$ for \eqref{eq:BS_2D} to define a bounded velocity, Yudovich theory cannot be applied directly to nondecaying vorticity data. However, several results relax the $L^1$ condition in Yudovich-type theories. Serfati \cite{S1} proved that, if $u_0$, $\omg_0\in L^{\infty}$, then there exists a unique global-in-time solution to the $2D$ Euler equations. Benedetto, Marchioro, and Pulvirenti \cite{BMP} established global well-posedness under the assumptions $u/(1+|\cdot|^\alp) \in L^{\infty}$ and $\omg\in L^p\cap L^{\infty}$, where $0<\alp<1$ and $p<2/\alp$. The paper \cite{CK} also studies the $2D$ Euler equations with velocities that grow at infinity.

There are also  $L^{\infty}$ well-posedness results that apply to scale-invariant initial vorticity data. Elgindi and Jeong \cite{EJS} proved that every $m$-fold symmetric initial vorticity  $\omg_0\in L^{\infty}(\mathbb{R}^2)$, with $m\ge 3$, generates a unique global solution $\omg \in L^{\infty} ([0,\infty); L^{\infty}(\mathbb{R}^2))$. In particular, their result applies when the initial vorticity is both scale invariant and $m$-fold symmetric. Under the appropriate reflection symmetry, when  $\bt = \pi/n$ for an integer $n\ge 3$, Theorem \ref{thm:weak-wellposed} recovers Theorem A of \cite{EJS}. Our result extends this argument to every $0<\bt<\pi/2$.

\subsubsection*{Scale-invariant data}

Since a nontrivial scale-invariant $\omg_0\in L^{\infty}$ does not decay at infinity and is naturally defined on an unbounded domain, additional assumptions are needed to extend the classical Yudovich theory to this setting. The work \cite{EJS} provides one such extension. It also investigates the long-time behavior of positive, scale-invariant, $m$-fold symmetric vorticities, where $m\ge 3$. In particular, it shows that the only stationary solutions in this class are constant.

The long-time behavior of scale-invariant data was also studied in \cite{EMS}. Although the scale-invariant vorticity is sign-changing, the solution converges to a piecewise constant function when the initial data have bounded variation. Their result also assumes $m$-fold symmetry. These assumptions motivate the search for analogous conditions in other domains. One possible direction is the study of discretely scale-invariant data, which we introduce in Subsection \ref{Bounded Vorticity} but, to the best of our knowledge, has not been investigated in the existing literature.

\subsection{Organization}

The organization of this work is as follows:
\begin{itemize}
    \item In Section \ref{Preliminary}, we introduce the Green's function $G_{\bt}$ on $\Omg_{\bt}$, which is used throughout this paper. We also derive several singular-integral formulas involving associated kernels. These formulas are used in Sections \ref{Holder-wellposed} and \ref{Holder-wellposed2}.
    \item In Section \ref{weak-wellposed}, we prove Theorem \ref{thm:weak-wellposed}.
    \item In Section \ref{Holder-wellposed}, we prove Theorem \ref{thm:GWPHolder}. The well-posedness result in $\mathring{C}^{0,\alp}$ serves as the foundation for the analysis in Subsections \ref{LWPHolder} and \ref{borderline}.
    \item In Subsection \ref{LWPHolder}, we first prove Theorem \ref{thm:LWPHolder2} under the condition $k+\alp < \pi/\bt - 1$.
    \item In Subsection \ref{borderline}, we next prove Theorem \ref{thm:LWPHolder2} under the conditions $k+\alp = \pi/\bt - 1$ and $\pi/\bt\notin \mathbb{N}$.
    \item In Subsection \ref{ill-posed2}, we prove Theorem \ref{thm:ill-posed}.
    \item In Section \ref{Long-time}, we prove Theorem \ref{thm:Long-Scale}.
\end{itemize}

\subsection{Notation}

We introduce the notation used throughout this paper.
\begin{itemize}
    \item We define $\Omg_{\bt} := \{(r\cos\tht, r\sin\tht)\in\mathbb{R}^2: r>0,\ 0<\tht<\bt\}$, and equivalently write $\Omg_{\bt} = \{ r e^{i\tht}\in\mathbb{C}: r>0, 0<\tht<\bt\}$. Thus, it is an infinite sector with opening angle $\bt$. For convenience, we set $p:=\pi/\bt$.
    \item The set $B_R(a)$ is an open ball of radius $R$ centered at $a$. We abbreviate $B_R := B_R(0)$ and $\mathbb{N}_0 := \mathbb{N}\cup \{0\}$.
    \item The Green's function on $\Omg_{\bt}$ and the Biot-Savart kernel are denoted by $G_{\bt}$ and $K_{\bt}$, respectively. We define $\Gmm$ on $\mathbb{C}$ by $2\pi \Gmm(z) := \log|z|$.
    \item In $\Omg_{\bt}$, the complex number $z^s$ is defined by $z^s := r^s e^{is\tht}$ for each $z = re^{i\tht}$ ($0<\tht<\bt$) and $s>0$. The logarithm is defined by the principal branch.
    \item We denote $\omg_t := \omg(t,\cdot)$, $u_t := u(t,\cdot)$, and $\Phi_t := \Phi(t,\cdot)$. We use $\rd_t \omg$ for the partial derivative of $\omg$.
\end{itemize}

\subsection*{Acknowledgements}
The author would like to thank Prof. In-Jee Jeong for his invaluable guidance and support throughout this work. The author has been supported by NRF grant No. 2022R1C1C1011051 and RS-2024-00406821.

\section{Green's function on an infinite sector}\label{Preliminary}

\subsection{Green's function}
Throughout this paper, the angle of the corner is denoted by $\bt$, and we define $p$ by $p:=\pi/\bt$. The \textit{Dirichlet Green's function} $G_{\bt}$ on the domain $\Omg_{\bt}$ is the real-valued solution of
\begin{equation}\label{eq:Green}
    \left\{
        \begin{aligned}
            \Dlt_x G_{\bt}(x,y) & = \dlt_0(x-y), \quad\ x,y\in\Omg_{\bt},\\ G_{\bt}(x,y) & = 0, \quad\quad\quad\quad\quad x\in \partial\Omg_{\bt},\ y\in \Omg_{\bt}, \\ \lim_{|x|\rightarrow \infty} G_{\bt}(x,y) & = 0, \quad\quad\quad\quad\quad  x,y\in\Omg_{\bt},
        \end{aligned}
    \right.
\end{equation}
where $\dlt_0$ is the Dirac delta. 
\begin{remark}
    If a solution of \eqref{eq:Green} exists, then the weak maximum principle implies the uniqueness of \eqref{eq:Green}.
\end{remark}

From now on, we identify $(x,y)\in \mathbb{R}^2$ with $x+iy \in \mathbb{C}$, and $\Omg_{\bt}\subset \mathbb{R}^2$ with $\{ r e^{i\tht}\in\mathbb{C}: r>0, 0<\tht<\bt\}$, a subset of $\mathbb{C}$. Lemma \ref{thm:G_beta} implies that the solution $G_{\bt}$ of \eqref{eq:Green} exists and its formula is
\begin{align*}
    G_{\bt}(z,w) = \frac{1}{2\pi}\log \frac{\left|z^{\pi/\bt} - w^{\pi/\bt}\right|}{\left|z^{\pi/\bt} - \bar{w}^{\pi/\bt}\right|} = \frac{1}{2\pi}\log \frac{\left|z^p - w^p\right|}{\left|z^p - \bar{w}^p\right|},
\end{align*}
where $z$, $w\in \Omg_{\bt}$ are complex numbers. Recall that in $\mathbb{R}^2$, the Biot-Savart kernel is defined by the perpendicular gradient of the Green's function \cite{MB}. 
Since $2\partial_{\bar{z}} = \partial_x + i\partial_y$, the \textit{Biot-Savart kernel} $K_{\bt}$ is naturally defined by
\begin{align*}
    K_{\bt}(z,w):= 2i\partial_{\bar{z}} G_{\bt}(z,w) = \frac{ip}{2\pi} \frac{\bar{z}^{p-1} (\bar{w}^{p}- w^p)}{(\bar{z}^p - \bar{w}^p)(\bar{z}^p - w^p)}.
\end{align*}

The Biot-Savart kernel $K_{\bt}$ is a complex-valued function with two variables, and it satisfies $\Re(K_{\bt}) = -i(\partial_{\bar{z}}G_{\bt} - \partial_z G_{\bt})$ and $\Im(K_{\bt}) = \partial_{\bar{z}}G_{\bt} + \partial_z G_{\bt}$. In Section \ref{Holder-wellposed}, the gradient $\nb u$ is determined in complex notation by the formulas of $\rd_{\bar{z}} \Re K_{\bt}$ and $\rd_{\bar{z}} \Im K_{\bt}$. By direct calculation,
\begin{equation}\label{eq:G_deri_deri}
    \begin{split}
        \rd_{\bar{z}} \Im (K_{\bt}) (z,w)= \rd_{\bar{z}}^2 G_{\bt}(z,w) & = -\frac{p\bar{z}^{p-2}}{4\pi}\left( \frac{\bar{z}^p + (p-1)\bar{w}^p}{(\bar{z}^p - \bar{w}^p)^2} - \frac{\bar{z}^p + (p-1)w^p}{(\bar{z}^p - w^p)^2}\right),
    \end{split}
\end{equation}
and similarly $\rd_{\bar{z}} \Re (K_{\bt}) = -i\rd_{\bar{z}}^2 G_{\bt}$. Although the formula of $\rd_{\bar{z}}^k G_{\bt} (z,w)$ is quite complicated for large $k\in\mathbb{N}$, Lemma \ref{lem:ker_bound} states that if $\bt\in (0,\pi)$,  $k\in\mathbb{N}$, and $z$, $w\in\Omg_{\bt}$, then
\begin{equation*}
    \left| \rd_{\bar{z}}^k G_{\bt}(z,w) \right| \lesssim_{k,\bt} \frac{1}{|z-w|^k},\qquad |w|\le 2|z|,
\end{equation*}
and
\begin{equation*}
    \left| \rd_{\bar{z}}^k G_{\bt}(z,w) \right| \lesssim_{k,\bt} \frac{|z|^{\pi/\bt - k}}{|w|^{\pi/\bt}}, \qquad |w|\ge 2|z|
\end{equation*}
hold. These two bounds of $\rd_{\bar{z}}^k G_{\bt}(z,w)$ are used throughout this paper.

\subsection{Singular integral formulas involving Green's function}

In this subsection, we prove some properties of $G_{\bt}$ involving singular integrals. We first prove some lemmas that can be applied for all $0<\bt<\pi/2$. These lemmas are useful in Section \ref{Holder-wellposed}.

We define a function $\Gmm$ on $\mathbb{C}$ by
\begin{equation*}
    \Gmm(z) := \frac{1}{2\pi}\log|z|,\qquad z\in\mathbb{C}.
\end{equation*}
The Green's function on $\mathbb{C}$ is $\Gmm(z-w)$. Let $H_{\bt}$ be defined on $(\Omg_{\bt})^2$ by
\begin{equation*}
    H_{\bt} (z,w) := G_{\bt}(z,w) - G_{\pi}(z,w) =  G_{\bt}(z,w) - \Gmm(z-w) + \Gmm(z-\bar{w}), \qquad z, w\in\Omg_{\bt}.
\end{equation*}
The kernel $H_{\bt}$ extends continuously at $z=w$, and it is harmonic in $\Omg_{\bt}$.

We briefly recall the definition of the Taylor polynomial. Let $k\in\mathbb{N}_0$, $a\in\overline{\Omg_{\bt}}$ and $f\in C^k(\overline{\Omg_{\bt}})$. Then, the $k$-\textit{th order Taylor polynomial of} $f$ \textit{at} $a$ is a polynomial defined by
\begin{equation*}
    T_{k,a} f(z) := \sum_{j+l\le k} \frac{\rd_z^j \rd_{\bar{z}}^l f(a)}{j!l!}(z-a)^j (\bar{z}-\bar{a})^l.
\end{equation*}
If $a=0$, it is natural to abbreviate $T_k f(z) := T_{k,0} f(z)$.

\begin{lemma}\label{lem:gradu_pv}
    Let $k\in\mathbb{N}_0$, and $z\in\Omg_{\bt}$. Suppose that $0<\bt<\pi/2$, $0<\alp\le 1$, and $\omg\in C^{k,\alp}(\Omg_{\bt})$. Then
    \begin{equation}\label{eq:psi_Cauchy}
        \rd_{\bar{z}}^{k+2} \psi(z) = p.v.\int_{\Omg_{\bt}} \rd_{\bar{z}}^{k+2} G_{\bt}(z,w) \omg(w) dA(w).
    \end{equation}
    In particular, the Cauchy principal value is well-defined.
\end{lemma}

\begin{proof}
    It is enough to show that the Cauchy principal value in \eqref{eq:psi_Cauchy} exists by the definition of $\psi$. Since $\bt < \pi/2$, the singularity occurs only at $w=z$ by Lemma \ref{lem:ker_bound}. The function $\Gmm$ satisfies
    \begin{equation*}
        \int_{B_R\setminus B_r} \rd_{\bar{z}}^k \Gmm(z) z^m\bar{z}^n dA(z)=0
    \end{equation*}
    for all $R>r>0$ and $k\in\mathbb{N}$, $m$, $n\in\mathbb{N}_0$ with $k\ge m+n+2$. Choose $\dlt>0$ so that $\overline{B_{\dlt}(z)}\subset \Omg_{\bt}$. Since $\omg\in C^{k,\alp}(\overline{B_{\dlt}(z)})$, Taylor's theorem implies that $|\omg(w) - T_{k,z} \omg(w)|\lesssim \|\omg\|_{C^{k,\alp}}|z-w|^{k+\alp}$ for all $w\in B_{\dlt}(z)$. Therefore, the Cauchy principal value
    \begin{equation*}
        p.v.\int_{B_{\dlt}(z)} \rd_{\bar{z}}^{k+2}\Gmm(z-w)\omg(w) dA(w) = \int_{B_{\dlt}(z)\setminus B_{\eps}(z)} \rd_{\bar{z}}^{k+2}\Gmm(z-w)\{\omg(w) - T_{k,z} \omg(w)\}dA(w)
    \end{equation*}
    is well-defined for all $0<\eps<\dlt$. Since $\rd_{\bar{z}}^k H_{\bt}(z,w)$ is bounded near $z$ for all $k\in\mathbb{N}_0$, the Lebesgue integrals
    \begin{equation*}
        \int_{B_{\dlt}(z)} \rd_{\bar{z}}^{k+2}\Gmm(z-\bar{w})\omg(w) dA(w),\quad \int_{B_{\dlt}(z)} \rd_{\bar{z}}^{k+2}H_{\bt}(z,w)\omg(w) dA(w)
    \end{equation*}
    are finite. Therefore, the Cauchy principal value
    \begin{equation*}
        \begin{split}
            & p.v.\int_{B_{\dlt}(z)} \rd_{\bar{z}}^{k+2}G_{\bt}(z,w)\omg(w) dA(w) \\ & = p.v.\int_{B_{\dlt}(z)} \rd_{\bar{z}}^{k+2}\Gmm(z-w)\omg(w) dA(w) + \int_{B_{\dlt}(z)} \left(\rd_{\bar{z}}^{k+2}\Gmm(z-\bar{w}) + \rd_{\bar{z}}^{k+2}H_{\bt}(z,w)\right)\omg(w) dA(w)
        \end{split}
    \end{equation*}
    and the right-hand side of \eqref{eq:psi_Cauchy} are also well-defined.
\end{proof}

\begin{lemma}\label{lem:G_pi_poly}
    Let $k\ge 0$ and $m\ge 2$ be integers. Suppose that $\Im(z)>0$. If $m\ge 3$, then
    \begin{equation}\label{eq:G_pi_poly}
        p.v.\int_{\{\Im(w)>0\}} \rd_{\bar{z}}^{k+m} G_{\pi}(z,w)(z-w)^k dA(w) = 0.
    \end{equation}
    If $m=2$, then there exists a constant $C=C(k)$ such that 
    \begin{equation}\label{eq:G_pi_poly2}
        p.v.\int_{\{\Im(w)>0\}} \rd_{\bar{z}}^{k+2} G_{\pi}(z,w)(z-w)^k dA(w) = C.
    \end{equation}
\end{lemma}

Throughout this paper, we denote $\mathbb{H} := \{z\in\mathbb{C}: \Im(z)>0\}$.

\begin{proof}
    To prove \eqref{eq:G_pi_poly} and \eqref{eq:G_pi_poly2}, without loss of generality, we assume that $z=i$. Let us denote $w=:x+iy$, and $\xi := w-z = x+i(y-1)$.

    We first suppose that $m\ge 3$. Since the function $\Gmm$ satisfies
    \begin{equation}\label{eq:Gmm_poly}
        \int_{B_{\dlt}(z)\setminus B_{\eps}(z)} \rd_{\bar{z}}^{k+m} \Gmm(z-w)(z-w)^k dA(w) = 0
    \end{equation}
    for all $\dlt>\eps>0$, the Cauchy principal value \eqref{eq:G_pi_poly} is well-defined. Our goal is to show that the integral
    \begin{equation*}
        \int_{\mathbb{R}} \frac{(z-w)^k}{(z-\bar{w})^{k+m}} dx = \int_{\mathbb{R}} \frac{((-x+i(y+1)) - 2iy )^k}{(-x+i(y+1))^{k+m}} dx
    \end{equation*}
    is equal to zero for all $y>0$. Since
    \begin{equation*}
        (-x+i(y+1))^{n-(k+m)} = \frac{d}{dx} \frac{(-x+i(y+1))^{1+n-(k+m)}}{-(1+n-(k+m))}, \quad \lim_{|x|\rightarrow\infty} (-x+i(y+1))^{1+n-(k+m)}=0
    \end{equation*}
    for all $0\le n\le k$, the binomial theorem yields that
    \begin{equation*}
        \int_{\mathbb{R}} \frac{(z-w)^k}{(z-\bar{w})^{k+m}} dx = 0.
    \end{equation*}
    Since $|\rd_{\bar{z}}^{k+m} G_{\pi}(z,w)| \lesssim |z-w|^{-(k+m)}$, Fubini's theorem implies that
    \begin{equation*}
        \int_{\mathbb{H}} \frac{(z-w)^k}{(z-\bar{w})^{k+m}} dA(w) = 0.
    \end{equation*}
    Let $\xi = re^{i\tht}$, and $\phi(r) := \sin^{-1}(1/r)$. We use the cancellation \eqref{eq:Gmm_poly} to obtain
    \begin{equation}\label{eq:orthogonal}
        \begin{split}
            p.v.\int_{\mathbb{H}} \frac{(z-w)^k}{(\bar{z}-\bar{w})^{k+m}} dA(w) & = \int_1^{\infty} r^{1-m} \int_{-\phi(r)}^{\pi+\phi(r)} e^{i(2k+m)\tht} d\tht dr \\ & = \int_0^{\pi/2} \sin^{m-3} u \cos u \int_{-u}^{\pi+u} e^{i(2k+m)\tht} d\tht du, \qquad u:= \phi(r).
        \end{split}
    \end{equation}
    The integration with respect to $\tht$ is
    \begin{equation*}
        \int_{-u}^{\pi+u} e^{i(2k+m)\tht} d\tht = \left\{
        \begin{aligned}
            \frac{2\sin((2k+m) u) }{(2k+m)},\qquad m\ \textrm{even},\\
            \frac{2i\cos((2k+m) u) }{(2k+m)},\qquad m\ \textrm{odd}.
        \end{aligned}
    \right.
    \end{equation*}
    The right-hand side of \eqref{eq:orthogonal} becomes zero by the orthogonality of the trigonometric system on $[0,\pi/2]$. Therefore, if $m\ge 3$, then the equation \eqref{eq:G_pi_poly} holds.

    We now assume that $m = 2$. Since $z=i$ and $w=i+\xi$,
    \begin{equation*}
        \rd_{\bar{z}}^{k+2} G_{\pi}(z,w) (z-w)^k = c_k \left(  \frac{\xi^k}{\bar{\xi}^{k+2}}- \frac{\xi^k}{(\xi + 2i)^{k+2}} \right)
    \end{equation*}
    for some constant $c_k$. Let $\xi = re^{i\tht}$, and $\phi(r) :=  \sin^{-1}(1/r)$. Since the averages of $e^{i(2k+2)\tht}$ and $e^{-2i\tht}$ on each half-circle centered at the origin are zero, the integration of $\rd_{\bar{z}}^{k+2} G_{\pi}(z,w) (z-w)^k$ on $\mathbb{H}\cap \{|w-i| = r\}$ satisfies
    \begin{equation*}
        \left|\int_{-\phi(r)}^{\pi+\phi(r)}\left( \frac{e^{i(2k+2)\tht}}{r^2} - \frac{e^{-2i\tht}}{r^2} + \frac{(\xi+2i)^{k+2} - \xi^{k+2}}{\xi^2 (\xi+2i)^{k+2}}\right) d\tht\right| \lesssim \frac{C_k}{r^3},\qquad r>2
    \end{equation*}
    for some constant $C_k>0$. Therefore, the Cauchy principal value of \eqref{eq:G_pi_poly2} exists.
\end{proof}

\begin{lemma}\label{lem:H_beta}
    Suppose that $0<\bt<\pi/2$. Let $k\in\mathbb{N}_0$, and $z\in\Omg_{\bt}$ with $\arg(z)\le \bt/2$. Then $H_{\bt}$ satisfies
    \begin{equation*}
        |\rd_{\bar{z}}^k H_{\bt}(z,w)| \le C_{k,\bt}|z|^{-k}
    \end{equation*}
    for all $w\in\Omg_{\bt}$ such that $|z-w| \le 2|z|$. 
\end{lemma}

\begin{proof}
    Without loss of generality, suppose that $|z|=1$. The functions
    \begin{equation*}
        w\mapsto \log\frac{|z^p - w^p|}{|z-w|},\quad w\mapsto \log\frac{|z^p - \bar{w}^p|}{|z-\bar{w}|}
    \end{equation*}
    and their derivatives are smooth and bounded in $\{w\in\Omg_{\bt}: |z-w|\le 2\}$. This implies the boundedness of $|\rd_{\bar{z}}^k H_{\bt}(z,w)| \le C_{k,\bt}$ for all $k\in\mathbb{N}_0$. 
\end{proof}

\begin{lemma}\label{lem:HalfGradKer}
    Let $k\ge 0$, $m\ge 2$ be integers, and $z\in\Omg_{\bt}$. Suppose that $0<\bt<\pi/2$. There exists a constant $C_1=C_1(k,m,\bt)$ such that
    \begin{equation}\label{eq:halfGradKer}
        \left|p.v.\int_{\{w\in\Omg_{\bt}: |z-w|\le 2|z|\}} \rd_{\bar{z}}^{k+m} G_{\bt}(z,w) (z-w)^k dA(w)\right|\le \frac{C_1}{|z|^{m-2}}.
    \end{equation}
    If $z'\in\Omg_{\bt}$ satisfies $|z-z'|\le \bt|z|/10$, then there exists a constant $C_2=C_2(k,m,\bt)$ such that
    \begin{equation}\label{eq:halfGradKer2}
        \left|p.v.\int_{\{w\in\Omg_{\bt}: |z-w|\le 2|z|\}} (\rd_{\bar{z}}^{k+m} G_{\bt}(z,w) (z-w)^k - \rd_{\bar{z}}^{k+m} G_{\bt}(z',w) (z'-w)^k) dA(w)\right|\le \frac{C_2}{|z|^{m-1}}|z-z'|.
    \end{equation}
\end{lemma}

\begin{proof}
    Without loss of generality, we assume that $|z|=1$ and $\arg(z)\le \bt/2$. If $m\ge 3$, Lemma \ref{lem:G_pi_poly} implies that
    \begin{equation}\label{eq:poly-k1}
        \left|p.v.\int_{\mathbb{H}\cap B_2(z)} \rd_{\bar{z}}^{k+m} G_{\pi}(z,w) (z-w)^k dA(w)\right| \le C_k \int_2^{\infty} \frac{dr}{r^{m-1}} \le C_{k,m}.
    \end{equation}
    If $m=2$, then the cancellation on half-circles yields that
    \begin{equation*}
        \left|p.v.\int_{\mathbb{H}\cap B_2(z)} \rd_{\bar{z}}^{k+2} \Gmm(z-w) (z-w)^k dA(w)\right| \lesssim \int_{\Im(z)}^2 \frac{1}{r} \sin^{-1}\frac{\Im(z)}{r} dr \le C_k.
    \end{equation*}
    Since $\mathbb{H}\cap B_2(z) \subset \mathbb{H} \cap B_{2 + 2\Im(z)}(\bar{z})$, the binomial theorem implies that
    \begin{equation*}
        \begin{split}
            \left|\int_{\mathbb{H}\cap B_2(z)} \rd_{\bar{z}}^{k+2} \Gmm(\bar{z}-w) (z-w)^k dA(w)\right| & \lesssim \sum_{n=0}^k (2\Im(z))^{k-n} \int_{\Im(z)}^{2\Im(z) + 2} \frac{1}{r^{k+1-n}} \sin^{-1}\frac{\Im(z)}{r} dr \\ & \lesssim \sum_{n=0}^k \Im(z)^{k-n} \frac{\Im(z)}{\Im(z)^{k+1-n}} \le C_{k}.
        \end{split}
    \end{equation*}
    This completes the proof of \eqref{eq:halfGradKer}.
    
    We now prove \eqref{eq:halfGradKer2}. Lemma \ref{lem:G_pi_poly} and the complex mean value inequality yield that for all $m\ge 2$,
    \begin{equation*}
        \left|p.v.\int_{\mathbb{H}\cap B_2(z)} (\rd_{\bar{z}}^{k+m} G_{\pi}(z,w) (z-w)^k - \rd_{\bar{z}}^{k+m} G_{\pi}(z',w) (z'-w)^k) dA(w)\right|\le C_k |z-z'| \int_2^{\infty} \frac{dr}{r^m},
    \end{equation*}
    which is bounded by a constant. Therefore, Lemma \ref{lem:H_beta} completes the proof of \eqref{eq:halfGradKer2}.
\end{proof}

A simple modification yields that the polynomial $(z-w)^k$ in Lemma \ref{lem:HalfGradKer} can be replaced by
\begin{equation*}
    (z-w)^m (\overline{z-w})^{k-m}
\end{equation*} 
for all $0\le m\le k$ in Lemmas \ref{lem:G_pi_poly} and \ref{lem:HalfGradKer}.

Lemmas \ref{lem:gradu_pv} and \ref{lem:HalfGradKer} with the complex mean value theorem yield the following lemma.

\begin{lemma}\label{lem:gradu_pv2}
    Let $k\in\mathbb{N}_0$, and fix $z\in\Omg_{\bt}$ with $|z|=1$. Suppose that $0<\bt<\pi/2$, $0<\alp\le 1$, and $\omg\in C^{k,\alp}(\Omg_{\bt})$. Let $A := \{w\in\Omg_{\bt}: |z-w|\le 2\}$, and $z'\in\Omg_{\bt}$ satisfy $|z-z'|\le \bt/10$. Then, for every $m\in\mathbb{N}_0$ with $m\le k$,
    \begin{equation*}
        R_{m}[\omg](z):= p.v.\int_{A} \rd_{\bar{z}}^{k+2} G_{\bt}(z,w) T_{m,z}\omg(w) dA(w) 
    \end{equation*}
    satisfies
    \begin{equation*}
        \left|R_m[\omg](z)\right| \lesssim C_{k,\bt} \|\omg\|_{C^m(\overline{A})}, \quad \left|R_m[\omg](z) - R_m[\omg](z')\right| \lesssim C_{k,\alp,\bt} \|\omg\|_{C^{m,\alp}(\overline{A})}|z-z'|^{\alp}.
    \end{equation*}
\end{lemma}

\medskip

Before concluding this subsection, we formulate several lemmas that are valid only for sufficiently small $\bt$. These lemmas will be used in Section \ref{Holder-wellposed2}.

\begin{lemma}\label{lem:Gbeta_nm}
    Let $n$, $m\in \mathbb{N}_0$. Suppose that $0<\bt<\pi/(n+m+2)$. Then, for all $z\in\Omg_{\bt}$,
    \begin{equation}\label{eq:Gbeta_nm}
        \int_{\Omg_{\bt}} G_{\bt} (z,w) w^m \bar{w}^n dA(w) = C_1 z^{m+1} \bar{z}^{n+1} + C_2 z^{m+n+2} + C_3 \bar{z}^{m+n+2},
    \end{equation}
    where $C_1 = C_1(m,n)$, $C_2=C_2(\bt,m,n)$, and $C_3 = C_3(\bt,m,n)$ are constants.
\end{lemma}

The condition $\bt < \pi/(m+n+2)$ is necessary for the well-definedness of the integral in \eqref{eq:Gbeta_nm}.

\begin{proof}
    Let $z\in\Omg_{\bt}$, and $\psi$ be the function defined on $\Omg_{\bt}$ by
    \begin{equation}\label{eq:psi_nm}
        \psi(z) := \int_{\Omg_{\bt}} G_{\bt}(z,w) w^m \bar{w}^n dA(w).
    \end{equation}
    By Lemma \ref{lem:ker_bound}, the function $\psi$ is well-defined and $\Dlt_z \psi (z)= \int_{\Omg_{\bt}} \Dlt_z G_{\bt}(z,w) w^m \bar{w}^n dA(w)$. Moreover, it is a solution of the Dirichlet problem
    \begin{equation*}
        \left\{
        \begin{aligned}
            \Dlt \psi(z) &= z^m\bar{z}^n,\qquad z\in\Omg_{\bt},\\
            \psi(z) &= 0,\qquad\qquad z\in\rd \Omg_{\bt}.
        \end{aligned}
    \right.
    \end{equation*} 
    Let us denote $z = re^{i\tht}$ in polar coordinates with $r>0$ and $0<\tht<\bt$. Then
    \begin{equation*}
        \psi(re^{i\tht}) = r^{m+n+2} \int_{\Omg_{\bt}} G_{\bt} (e^{i\tht}, w) w^m\bar{w}^n dA(w) =: r^{m+n+2} A(\tht).
    \end{equation*}
    The function $A$ defined on $[0,\bt]$ satisfies $(m+n+2)^2 A(\tht) + A''(\tht) = e^{i(m-n)\tht}$ for all $0<\tht<\bt$, and $A(0)=A(\bt)=0$. The unique solution of this ODE is given by
    \begin{equation*}
        A(\tht) = \frac{1}{4(m+1)(n+1)}\left( e^{i(m-n)\tht} - \cos(m+n+2)\tht + \frac{\cos(m+n+2)\bt - e^{i(m-n)\bt}}{\sin(m+n+2)\bt} \sin(m+n+2)\tht\right).
    \end{equation*}
    This implies that $\psi(re^{i\tht})$ has the form of the right-hand side of \eqref{eq:Gbeta_nm}.
\end{proof}

\begin{lemma}\label{lem:IntGradKer}
    Let $k\in\mathbb{N}_0$. Suppose that $0<\bt<\pi/(k+2)$. For any $z\in\Omg_{\bt}$, the Green's function $G_{\bt}$ satisfies
    \begin{align*}
        p.v.\int_{\Omg_{\bt}} \partial_{\bar{z}}^{k+3} G_{\bt}(z,w) P_k(w)dA(w) = 0, \quad p.v.\int_{\Omg_{\bt}} \partial_{\bar{z}}^{k+2} G_{\bt}(z,w) w^n\bar{w}^m dA(w)= C(\bt,n,m)
    \end{align*}
    for all polynomials $P_k$ of order at most $k$, and $m$, $n\in\mathbb{N}_0$ with $m+n=k$.
\end{lemma}

\begin{proof}
    The desired equality is a direct consequence of Lemmas \ref{lem:gradu_pv} and \ref{lem:Gbeta_nm}.
\end{proof}

\section{Well-posedness for bounded vorticity}\label{weak-wellposed}

In this section, we prove Theorem \ref{thm:weak-wellposed}. Before the proof, we recall the weak formulation of the Euler equations on $\Omg_{\bt}$. Since \eqref{eq:Euler_vorticity} implies that the vorticity $\omg$ is transported by $u$, if $\Phi$ is the flow map associated with $u$, then
\begin{equation}\label{eq:Flow}
   \left\{
       \begin{aligned}
            \frac{d}{dt} \Phi_t(z) & = u_t(\Phi_t(z)), \qquad \Phi_0(z) = z, \\
            \omg_t(\Phi_t(z)) & = \omg_0(z),
       \end{aligned}
   \right.
\end{equation}
where $\omg_0$ is the initial vorticity. The Biot-Savart law in $\Omg_{\bt}$ takes the form 
\begin{equation}\label{eq:Biot-Savart}
    u_t(z) = \int_{\Omg_{\bt}} K_{\bt}(z,w) \omg_t(w)dA(w),
\end{equation}
where $K_{\bt}$ is the Biot-Savart kernel defined in Section \ref{Preliminary}. If $(\omg, u, \Phi)$ satisfies both \eqref{eq:Flow} and \eqref{eq:Biot-Savart}, then we call $(\omg, u, \Phi)$ a \textit{solution triple} on $\Omg_{\bt}$. If there is no ambiguity, we sometimes use $(\omg_t, u_t, \Phi_t)$ in place of $(\omg, u, \Phi)$. The classical derivatives of $u$ need not exist. Thus, this notion of a solution triple provides a weak formulation of the Euler equations. Theorem \ref{thm:weak-wellposed} is equivalent to the following statement:
\begin{theorem} \label{thm:weak-wp}
    Suppose that $0<\bt<\pi/2$ and $\omg_0 \in L^{\infty}(\Omg_{\bt})$. Then there exists a unique global solution triple $(\omg, u, \Phi)$ to the Euler equations such that $\omg\in L^{\infty}([0,\infty); L^{\infty}(\Omg_{\bt}))$.
\end{theorem}

\subsection{Proof of uniqueness}

We first consider the uniqueness problem.

\begin{proof}[Proof of uniqueness in Theorem \ref{thm:weak-wellposed}]
Let $(\omg, u, \Phi)$, $(\tilde{\omg}, \tilde{u}, \tilde{\Phi})$ be two solution triples. Lemma \ref{1.1} implies that there exists a constant $T>0$, depending only on $\|\omg_0\|_{L^{\infty}}$, such that
\begin{equation*}
    \frac{|\Phi_t(z) - z|}{|z|}, \ \frac{|\tld{\Phi}_t(z) - z|}{|z|} < \frac{1}{100}, \qquad 0<t<T.
\end{equation*}
Since $p > 2$, there exists $\dlt>0$ such that $p>2+\dlt$. Let $f:\mathbb{R}_+ \rightarrow \mathbb{R}$ be a function defined by
\begin{align*}
    f(t):= \int_{\Omg_{\bt}} \frac{|\Phi_t(z) - \tilde{\Phi}_t(z)|}{|z|} \frac{|\omg_0(z)|}{1+|z|^{2+\dlt}}dA(z).
\end{align*}
Our goal is to establish the following inequality:
\begin{equation}\label{eq:f_ineq}
    f(t) \lesssim_{\bt,\|\omg_0\|_{L^{\infty}}} \int_0^t f(s)\log\left( 1+ \frac{1}{f(s)}\right) ds.
\end{equation}
Once the inequality \eqref{eq:f_ineq} is proved, we may apply Lemma \ref{lem:MP} to obtain the desired conclusion.
\begin{lemma} [{\cite[Lemma 2.3.2]{MP}}] \label{lem:MP}
    Suppose that $T>0$, and $g: [0,T) \rightarrow\mathbb{R}_+ \cup\{0\}$ is continuous with $g(0)=0$. If there exists a constant $C>0$ such that
    \begin{equation*}
        g(t) \le C\int_0^t \varphi\left(g(s)\right) ds, \quad \varphi(r) := \left\{
        \begin{aligned}
            r\log(1+1/r),\quad r>0,\\
            0,\quad r=0
        \end{aligned}
    \right.
    \end{equation*}
    for all $t\in (0,T)$, then $g\equiv 0$.
\end{lemma}
We now turn to the proof of \eqref{eq:f_ineq}. Let $t\in (0,T)$. Since $\Phi$, $\tld{\Phi}$ are the flow maps associated with $u$ and $\tld{u}$, respectively, we write
\begin{align*}
    \Phi_t(z) - \tilde{\Phi}_t(z) & = \int_0^t \int_{\Omg_{\bt}} \left[ K_{\bt} \left( \Phi_s(z), \Phi_s(w) \right) - K_{\bt} \left( \Phi_s(z), \tilde{\Phi}_s(w) \right) \right]\omg_0(w) dA(w) ds \\ &\quad\quad + \int_0^t  \left[\tilde{u}_s\left(\Phi_s(z)\right) - \tilde{u}_s (\tilde{\Phi}_s(z) ) \right] ds =: I_1 + I_2.
\end{align*}
It follows directly from the definition of $f$ that
\begin{equation*}
    f(t) \le \int_{\Omg_{\bt}} \frac{|I_1|}{|z|} \frac{|\omg_0(z)|}{1+|z|^{2+\dlt}}dA(z) + \int_{\Omg_{\bt}} \frac{|I_2|}{|z|} \frac{|\omg_0(z)|}{1+|z|^{2+\dlt}}dA(z).
\end{equation*}

We first estimate the contribution of $I_2$ to $f(t)$.
\begin{align*}
    \int_{\Omg_{\bt}} \frac{|I_2|}{|z|} \frac{|\omg_0(z)|}{1+|z|^{2+\dlt}} dA(z) & \lesssim_{\bt,\|\omg_0\|_{L^{\infty}}} \int_0^t \int_{\Omg_{\bt}} |\Phi_s(z) - \tilde{\Phi}_s(z)| \log\left( \frac{|\Phi_s(z)|}{|\Phi_s(z) - \tilde{\Phi}_s(z)|} \right) \frac{1}{|z|}\frac{|\omg_0(z)|}{1+|z|^{2+\dlt}}dA(z)ds \\ & \lesssim \int_0^t \int_{\Omg_{\bt}} \varphi\left(\frac{|\Phi_s(z) - \tilde{\Phi}_s(z)|}{|z|}\right) \frac{|\omg_0(z)|}{1+|z|^{2+\dlt}} dA(z)ds \\ &  \lesssim_{\|\omg_0\|_{L^{\infty}}} \int_0^t \varphi\left( \int_{\Omg_{\bt}}  \frac{|\Phi_s(z) - \tilde{\Phi}_s(z)|}{|z|} \frac{|\omg_0(z)|}{1+|z|^{2+\dlt}} dA(z) \right)  ds = \int_0^t \varphi\left(f(s) \right) ds.
\end{align*}
The first inequality follows from Lemma \ref{lem:u_diff}, while the last follows from Jensen's inequality. We next estimate the contribution of $I_1$ to $f(t)$.
\begin{align*}
    & \int_{\Omg_{\bt}} \frac{|I_1|}{|z|} \frac{|\omg_0(z)|}{1+|z|^{2+\dlt}} dA(z) \\ &\quad \le \int_0^t \int_{\Omg_{\bt}} \int_{\Omg_{\bt}} \left| K_{\bt} \left( \Phi_s(z), \Phi_s(w) \right) - K_{\bt} \left( \Phi_s(z), \tilde{\Phi}_s(w) \right) \right||\omg_0(w)| \frac{1}{|z|}\frac{|\omg_0(z)|}{1+|z|^{2+\dlt}}dA(z) dA(w)ds \\ &\quad \lesssim_{\bt,\|\omg_0\|_{L^{\infty}}} \int_0^t \int_{\Omg_{\bt}} |\Phi_s(w) - \tilde{\Phi}_s(w)|\log\left(\frac{|\Phi_s(w)|}{|\Phi_s(w)-\tilde{\Phi}_s(w)|}\right) \frac{1}{|w|}\frac{|\omg_0(w)|}{1+|w|^{2+\dlt}} dA(w)ds \\ &\quad\lesssim_{\|\omg_0\|_{L^{\infty}}} \int_0^t \varphi(f(s)) ds.
\end{align*}
The second inequality follows from Lemmas \ref{lem:K_beta1} and \ref{lem:K_beta2}. Combining the estimates for $I_1$ and $I_2$, we obtain
\begin{align*}
    f(t) & \le \int_{\Omg_{\bt}} \frac{|I_1| + |I_2|}{|z|} \frac{|\omg_0(z)|}{1+|z|^{2+\dlt}}dA(z)  \lesssim_{\bt,\|\omg_0\|_{L^{\infty}}} \int_0^t \varphi(f(s)) ds.
\end{align*}
By Lemma \ref{lem:MP}, $\Phi_s \equiv \tilde{\Phi}_s$ for $0\le s<T(\|\omg_0\|_{L^{\infty}})$. This proves local-in-time uniqueness.

Since the $L^{\infty}$ norm of the vorticity is preserved, the same argument can be iterated on the intervals $[T,2T]$, $[2T,3T]$ and so on. Hence, uniqueness extends globally in time, which proves the uniqueness assertion in Theorem \ref{thm:weak-wellposed}.
\end{proof}

We now prove Lemmas \ref{lem:K_beta1} and \ref{lem:K_beta2} to provide the details of the estimate for $I_1$. Fix $w\in\Omg_{\bt}$, and decompose $\Omg_{\bt}$ into the following four pieces:
\begin{align*}
    & A:=\{z\in\Omg_{\bt}: |z|\le \frac{|w|}{10}\}, \quad B:=\{z\in\Omg_{\bt}: |z|>10|w|\}, \\ & C:=\{ z\in\Omg_{\bt}: |\Phi_s(z)-\Phi_s(w)|\le 3|\Phi_s (w) - \tilde{\Phi}_s (w)|\}, \quad D:=\Omg_{\bt}\backslash(A\cup B\cup C).
\end{align*}

\begin{lemma} \label{lem:K_beta1}
    Suppose that $p>2$, $0<\dlt<p-2$, and $0<s<T$. In the region $A\cup B$,
    \begin{align*}
    \int_{A\cup B} \left|K_{\bt} (\Phi_s(z), \Phi_s(w)) - K_{\bt} (\Phi_s(z), \tilde{\Phi}_s(w))\right| \frac{|w|}{|z|}\frac{1+|w|^{2+\dlt}}{1+|z|^{2+\dlt}} dA(z) \lesssim_p |\Phi_s(w) - \tilde{\Phi}_s(w)|.
    \end{align*} 
\end{lemma}

\begin{proof}
    The kernel $K_{\bt}$ has the Wirtinger derivatives
    \begin{align*}
        \partial_w K_{\bt}(z,w) = -\frac{ip^2 \bar{z}^{p-1} w^{p-1}}{2\pi (\bar{z}^p - w^p)^2}, \qquad \partial_{\bar{w}} K_{\bt}(z,w) = \frac{ip^2 \bar{z}^{p-1} \bar{w}^{p-1}}{2\pi (\bar{z}^p - \bar{w}^p)^2}.
    \end{align*}
    Since $|\Phi_s(z)^p - \Phi_s(w)^p| \gtrsim |z^p - w^p|$ in $ A\cup B$, the complex mean value inequality implies that
    \begin{align*}
        |K_{\bt} (\Phi_s(z), \Phi_s(w)) - K_{\bt} (\Phi_s(z), \tilde{\Phi}_s(w))| \le C\frac{p^2|z|^{p-1}|w|^{p-1}}{|z^p - w^p|^2} |\Phi_s(w) - \tilde{\Phi}_s(w)|,
    \end{align*}
    where $C>0$ is a constant. Therefore, it is enough to show that
    \begin{align*}
        \int_{A\cup B} \frac{|z|^{p-1} |w|^{p-1}}{|z^p - w^p|^2} \frac{|w|}{|z|}\frac{1+|w|^{2+\dlt}}{1+|z|^{2+\dlt}} dA(z) \le C_p
    \end{align*}
    for some constant $C_p$.
    
    In the region $B$,
    \begin{align*}
        \int_{B} \frac{|z|^{p-1} |w|^{p-1}}{|z^p - w^p|^2} \frac{|w|}{|z|}\frac{1+|w|^{2+\dlt}}{1+|z|^{2+\dlt}} dA(z) & \lesssim \int_{B_{10}^c\cap\Omg_{\bt}} \frac{1}{|\zt|^{p+2}} \frac{1+|w|^{2+\dlt}}{1+|w|^{2+\dlt}|\zt|^{2+\dlt}} dA(\zt)\quad\quad (\zt:=z/|w|) \\ & \lesssim_p \int_{10}^{\infty} \frac{1}{r^{p+1}}\frac{1+|w|^{2+\dlt}}{1+|w|^{2+\dlt} r^{2+\dlt}} dr \le C_p.
    \end{align*}
    The last inequality follows from the dominated convergence theorem. Indeed, since the integral has finite limits as $|w|\rightarrow{}0$ and as $|w|\rightarrow{}\infty$, it is uniformly bounded in $|w|$.
    
    In the region $A$,
    \begin{align*}
        \int_{A} \frac{|z|^{p-1} |w|^{p-1}}{|z^p - w^p|^2} \frac{|w|}{|z|}\frac{1+|w|^{2+\dlt}}{1+|z|^{2+\dlt}} dA(z) & \lesssim \int_{B_{1/10}\cap\Omg_{\bt}} |\zt|^{p-2} \frac{1+|w|^{2+\dlt}}{1+|w|^{2+\dlt}|\zt|^{2+\dlt}} dA(\zt) \\ & \lesssim_p \int_0^{1/10} r^{p-1} \frac{1+|w|^{2+\dlt}}{1+|w|^{2+\dlt}r^{2+\dlt}}dr \le C_p.
    \end{align*}
    Similarly, since $p-3-\dlt > -1$, the above integral is uniformly bounded in $|w|$.
\end{proof}

\begin{lemma} \label{lem:K_beta2}
    Suppose that $p>1$, $\dlt>0$, and $0<s<T$. In the region $C\cup D$,
    \begin{align*}
        \int_{C\cup D} |K_{\bt} (\Phi_s(z), \Phi_s(w)) - K_{\bt} (\Phi_s(z), \tilde{\Phi}_s(w))| \frac{|w|}{|z|}\frac{1+|w|^{2+\dlt}}{1+|z|^{2+\dlt}} dA(z) \lesssim_p |\Phi_s(w) - \tilde{\Phi}_s(w)|\log\frac{|\Phi_s(w)|}{|\Phi_s(w)-\tilde{\Phi}_s(w)|}.
    \end{align*}
\end{lemma}

\begin{proof}
    Note that $|z|$, $|w|$, $|\Phi_s(z)|$, $|\Phi_s(w)|$ and $|\tilde{\Phi}_s(w)|$ are comparable in $C\cup D$. In the region $C$,
    \begin{align*}
        \int_C |K_{\bt} (\Phi_s(z), \Phi_s(w)) - K_{\bt} (\Phi_s(z), \tilde{\Phi}_s(w))| dA(z) & \lesssim_p \int_C\left( \frac{1}{|\Phi_s(z)-\Phi_s(w)|} + \frac{1}{|\Phi_s(z) - \tilde{\Phi}_s(w)|}\right) dA(z)\\ &  \le \int_{\{z\in\Omg_{\bt}: |z-w|<3|\Phi_s(w)-\tilde{\Phi}_s(w)|\}} \frac{2dA(z)}{|z-w|} \lesssim  |\Phi_s(w) - \tilde{\Phi}_s(w)|.
    \end{align*}
    The first inequality holds by Lemma \ref{lem:exp_p}, and the second inequality follows from the incompressibility. We now consider the remainder. Since $|\Phi_s(z) - \Phi_s(w)| \ge |z-w|/2$ in $D$,
    \begin{align*}
        \int_D |K_{\bt} (\Phi_s(z), \Phi_s(w)) - K_{\bt} (\Phi_s(z), \tilde{\Phi}_s(w))| dA(z) & \lesssim_p \int_D \frac{|\Phi_s(w)-\tilde{\Phi}_s(w)| dA(z)}{|\Phi_s(z)-\Phi_s(w)| |\Phi_s(z) - \tilde{\Phi}_s(w)|} \\ & \lesssim |\Phi_s(w) - \tilde{\Phi}_s(w)|\int_D \frac{dA(z)}{|z-w|^2} \\ &  \lesssim |\Phi_s(w) - \tilde{\Phi}_s(w)|\log\left(\frac{|\Phi_s(w)|}{|\Phi_s(w)-\tilde{\Phi}_s(w)|}\right).
    \end{align*}
\end{proof}

\subsection{Proof of existence}

Now we prove the existence part in Theorem \ref{thm:weak-wellposed}.
\begin{proof}[Proof of existence in Theorem \ref{thm:weak-wellposed}]
We construct a sequence $\{\omg^{(n)}\}$ such that $\omg^{(n)}$ converges to a solution $\omg$ as $n\rightarrow\infty$.

\medskip
\noindent\textit{Step 1.}
Let $\omg^{(0)}_t(z) := \omg^{(0)}(t,z) := \omg_0(z)$, and recursively define the sequence of triples $\left(\omg_t^{(n)}, u_t^{(n)}, \Phi_t^{(n)}\right)$ by
\begin{equation}\label{eq:triple_infty}
   \begin{split}
        u_t^{(n)}(z) & := \int_{\Omg_{\bt}} K_{\bt}(z,w) \omg_t^{(n-1)}(w) dA(w), \\ \frac{d}{dt} \Phi_t^{(n)} (z)  & =  u_t^{(n)}\left(\Phi_t^{(n)}(z)\right),\quad \Phi_0^{(n)}(z) := z, \\ \omg_t^{(n)} \left(\Phi_t^{(n)}(z)\right) & := \omg_0 \left(z \right)
   \end{split}
\end{equation}
for each $n\in\mathbb{N}$, $t\in\mathbb{R}$, and $z\in \Omg_{\bt}$. Let $\Phi_t^{(0)}(z) := z$. Note that the velocity $u_t^{(n)}$, defined by the Biot-Savart law, satisfies the log-Lipschitz property by Lemma \ref{lem:u_diff}. Moreover, the constant in Lemma \ref{lem:u_diff} depends only on $p$ and $\|\omg_t^{(n-1)}\|_{L^{\infty}}$. Hence, by Osgood's lemma, the ODE $\frac{d}{dt} \Phi_t^{(n)} (z) = u_t^{(n)}(\Phi_t^{(n)}(z))$ has a unique solution. By backward uniqueness for this ODE, $\Phi_t^{(n)}$ is a bijection. Therefore, $\omg_t^{(n)}$ is well-defined on $\Omg_{\bt}$. 

Let $\dlt\in (0,p-2)$, and define $\rho^{(n)}:\mathbb{R}_+\rightarrow\mathbb{R}$ by
\begin{align*}
    \rho^{(n)} (t) := \int_{\Omg_{\bt}} \frac{|\Phi_t^{(n)}(z) -\Phi_t^{(n-1)}(z)|}{|z|}\frac{|\omg_0(z)|}{1+|z|^{2+\dlt}} dA(z)
\end{align*}
for each $n\in\mathbb{N}$. Since
\begin{align*}
    \Phi_t^{(n)}(z) - \Phi_t^{(n-1)}(z)   = \int_0^t &\int_{\Omg_{\bt}}  \left[ K_{\bt} \left( \Phi_s^{(n-1)}(z), \Phi_s^{(n-1)}(w) \right) - K_{\bt} \left( \Phi_s^{(n-1)}(z), \Phi_s^{(n-2)}(w) \right) \right]\omg_0(w) dA(w)ds \\ & + \int_0^t \left[u_s^{(n)}\left(\Phi^{(n)}_s(z)\right) - u_s^{(n)} \left(\Phi_s^{(n-1)}(z) \right)\right] ds,\qquad n\ge 2,
\end{align*}
arguing as in the proof of uniqueness, we obtain the bound
\begin{align*}
    \rho^{(n)}(t) \lesssim \int_0^t \varphi(\rho^{(n-1)}(s))ds + \int_0^t \varphi(\rho^{(n)}(s)) ds.
\end{align*}
Let $\sigma^{(n)}:\mathbb{R}_+\rightarrow\mathbb{R}$ be defined by $\sigma^{(n)}(t) := \sup_{m\ge n}\rho^{(m)}(t)$. Then,
\begin{align*}
    \sigma^{(n)} (t) \lesssim \int_0^t \varphi\left( \sigma^{(n-1)} (s)\right) ds.
\end{align*}
The proof of Lemma 3.2 in \cite{MP} implies that $\lim_{n\rightarrow\infty} \sigma^{(n)}=0$ uniformly in some short time interval $[0,T']$, where $T'$ depends only on $p$ and $\|\omg_0\|_{L^{\infty}}$. We assume that $T'$ is less than $T$, which appears in the proof of uniqueness.

\medskip
\noindent\textit{Step 2.}
We now prove the existence of a solution triple $\left(\omg_t, u_t, \Phi_t \right)$. To this end, we establish the uniform boundedness and equicontinuity of $\{\Phi_t^{(n)}\}_{n\ge 0}$ and apply the Arzel\`a-Ascoli theorem. Let $K\subset\Omg_{\bt}$ be a compact subset, and $0<t<T'$. Since $K$ is bounded, there exists $R>0$ such that $K\subset B_R$. For $z\in K$ and $t\in [0,T']$, we have
\begin{equation*}
    |\Phi^{(n)}_t(z)| \le |\Phi^{(n)}(0,z)| + \int_0^t |u^{(n)}_s(\Phi^{(n)}_s(z))| ds \le R + C_p \|\omg_0\|_{L^{\infty}}\int_0^{T'} |\Phi^{(n)}_s(z)| ds.
\end{equation*}
Therefore, by Gr\"onwall's inequality,
\begin{equation*}
    |\Phi^{(n)}_t(z)| \le R\exp(C_p\|\omg_0\|_{L^{\infty}} t), \qquad z\in K,\ t\in[0,T'].
\end{equation*}
Hence, $\{\Phi^{(n)}\}$ is uniformly bounded in $K\times [0,T']$.

To obtain the equicontinuity of $\{\Phi_t^{(n)}\}$ on $K$, we need to bound $|\Phi^{(n)}_t(z) - \Phi^{(n)}_t(z')|$ for $z$, $z'\in K$ ($z\neq z'$). By Lemma \ref{lem:u_diff}, $u_t^{(n)}$ is log-Lipschitz with an estimate
\begin{equation*}
   |u^{(n)}_t(\Phi^{(n)}_t(z)) - u^{(n)}_t(\Phi^{(n)}_t(z'))| \lesssim_p \|\omg_0\|_{L^{\infty}} |\Phi^{(n)}_t(z) - \Phi^{(n)}_t(z')|\log\frac{6|\Phi^{(n)}_t(z)|}{|\Phi^{(n)}_t(z) - \Phi^{(n)}_t(z')|}.
\end{equation*}
Since $\Phi_t^{(n)}$ is a flow map associated with $u_t^{(n)}$,
\begin{equation*}
    \left|\frac{d}{dt}|\Phi_t^{(n)}(z) - \Phi_t^{(n)}(z')|\right| \le | u^{(n)}_t(\Phi^{(n)}_t(z)) - u^{(n)}_t(\Phi^{(n)}_t(z')) |.
\end{equation*}
Therefore, a direct calculation yields 
\begin{equation*}
        \left|\frac{d}{dt}\log\frac{6|\Phi^{(n)}_t(z)|}{|\Phi^{(n)}_t(z) - \Phi^{(n)}_t(z')|}\right| \le C_p\|\omg_0\|_{L^{\infty}} \left( 1+ \log\frac{6|\Phi^{(n)}_t(z)|}{|\Phi^{(n)}_t(z) - \Phi^{(n)}_t(z')|}\right).
\end{equation*}
By Gr\"onwall's inequality, there exists $C_p>0$ such that for all $n\in\mathbb{N}$,
\begin{equation}\label{eq:Phi_diff}
    e^{-C_p\|\omg_0\|_{L^{\infty}}t}\log \frac{6e|z|}{|z-z'|} \le \log\frac{6e|\Phi^{(n)}_t(z)|}{|\Phi^{(n)}_t(z) - \Phi^{(n)}_t(z')|}\le e^{C_p\|\omg_0\|_{L^{\infty}}t}\log \frac{6e|z|}{|z-z'|}.
\end{equation}
Let $K\subset \Omg_{\bt}$ be compact. Since $K$ is bounded, $\{\Phi^{(n)}\}$ is equicontinuous on $K\times [0,T']$ by the first inequality.

We now apply the Arzel\`a-Ascoli theorem. For each compact subset $K\subset \Omg_{\bt}$, there exists a subsequence of $\{\Phi^{(n)}\}$ that converges uniformly on $K\times [0,T']$. Since $\Omg_{\bt}$ is $\sigma$-compact, a diagonal argument implies that there exist a subsequence $\{\Phi^{(n_k)}\}$ and a map $\Phi$ such that
\begin{equation*}
    \Phi(t,z) := \Phi_t (z) := \lim_{k\rightarrow\infty} \Phi_t^{(n_k)}(z)
\end{equation*}
for every $z\in\Omg_{\bt}$ and $t\in [0,T']$. Since $\{\Phi_t^{(n_k)}\}$ is continuous and converges to $\Phi_t$ uniformly on compact subsets, $\Phi_t$ is also continuous on $\Omg_{\bt}$. For each $z$, $z'\in \Omg_{\bt}$ with $z\neq z'$, letting $k\rightarrow\infty$ in \eqref{eq:Phi_diff} yields
\begin{equation*}
    \frac{|\Phi_t(z) - \Phi_t(z')|}{ 6e|\Phi_t(z)|} > 0. 
\end{equation*}
Therefore, $\Phi_t$ is injective for each $t$. Since $\Phi_t^{(n)}$ are measure-preserving, $\Phi_t^{(n_k)}\rightarrow \Phi_t$ uniformly on every compact subset of $\Omg_{\bt}$, and $|\Phi_t^{(n)}(z) - z| \le |z|/100$, we conclude that the limit $\Phi_t$ is also measure-preserving.

For each $0\le t\le T'$, we define $\omg_t$ by 
\begin{equation*}
    \omg_t(\Phi_t(z)) := \omg_0(z),\qquad z\in\Omg_{\bt}
\end{equation*}
and set $\omg_t(z):=0$ for $z\notin \Phi_t(\Omg_{\bt})$. Since $\Phi_t$ is injective and $\omg_0\in L^{\infty}$, this defines a bounded vorticity $\omg_t\in L^{\infty}$. We then define $u_t$ through the Biot-Savart law \eqref{eq:Biot-Savart}:
\begin{equation*}
    u_t(z) := \int_{\Omg_{\bt}} K_{\bt}(z,w) \omg_t(w)dA(w), \qquad z\in\Omg_{\bt}.
\end{equation*}
This $u_t$ is log-Lipschitz by Lemma \ref{lem:u_diff}. As we have seen in the proof of uniqueness, the ODE
\begin{equation*}
    \frac{d}{dt}\phi_t (z) = u_t(\phi_t(z)), \quad \phi_0(z)=z,  \qquad z\in\Omg_{\bt}
\end{equation*}
has a unique solution $\phi_t$ $(0\le t\le T')$ by Lemma 3.2 of \cite{MP}.

\medskip
\noindent\textit{Step 3.}
Now it is enough to show that $\Phi_t = \phi_t$. By another application of Arzel\`a-Ascoli theorem to the subsequence $\{\Phi_t^{(n_k-1)}\}$, there exists a subsequence $\{\Phi_t^{(n_{k_l}-1)}\}$ such that $\lim_{l\rightarrow\infty} \Phi_t^{(n_{k_l}-1)} =: \Psi_t$ locally uniformly in $\Omg_{\bt}$. For simplicity, we relabel this subsequence as $\Phi_t^{(n_k-1)}$. Since $\lim_{n\rightarrow\infty} \sigma^{(n)}(t)=0$, Fatou's lemma implies
\begin{equation*}
    \int_{\Omg_{\bt}} \frac{|\Phi_t(z) -\Psi_t(z)|}{|z|}\frac{|\omg_0(z)|}{1+|z|^{2+\dlt}} dA(z)\le \liminf_{k\rightarrow\infty}\int_{\Omg_{\bt}} \frac{|\Phi_t^{(n_k)}(z) -\Phi_t^{(n_k-1)}(z)|}{|z|}\frac{|\omg_0(z)|}{1+|z|^{2+\dlt}} dA(z) =0.
\end{equation*}
Therefore, we obtain $\Phi_t = \Psi_t$ a.e. in $\supp(\omg_0)$.

Since $\Phi_t^{(n_k)} (z) = z+ \int_0^t u_s^{(n_k)} (\Phi_s^{(n_k)}(z))ds$, the pointwise limit yields
\begin{equation*}
    \Phi_t (z) = z + \lim_{k\rightarrow\infty}\int_0^t u_s^{(n_k)} (\Phi_s^{(n_k)}(z)) ds.
\end{equation*}
Our goal is to show that $u_s^{(n_k)} (\Phi_s^{(n_k)}(z))\rightarrow u_s (\Phi_s(z))$ as $k\rightarrow \infty$. Let $0<\eps\ll 1$, and $\chi\in C_c^{\infty}(\overline{\Omg_{\bt}})$ be a cut-off function such that $\chi=0$ in $B_{\eps}(z)$ and $\chi = 1$ in $B_{2\eps}(z)^c$. Then, the dominated convergence theorem implies that
\begin{equation*}
    \int_{\Omg_{\bt}} \chi(w) (K_{\bt}(z,\Phi_t(w)) - K_{\bt}(z,\Phi_t^{(n_k - 1)}(w))) \omg_0(w)dA(w) \rightarrow 0
\end{equation*}
as $k\rightarrow\infty$. When we replace $\chi$ with $1-\chi$, the integral is uniformly bounded by
\begin{equation*}
    \int_{\Omg_{\bt}\cap B_{2\eps}(z)} |K_{\bt}(z,\Phi_t^{(n_k - 1)}(w))| |\omg_0(w)|dA(w) \lesssim \eps \|\omg_0\|_{L^{\infty}}.
\end{equation*}
Therefore, we obtain $\lim_{k\rightarrow\infty}u^{(n_k)} = u$. Since $u^{(n_k)}$ is log-Lipschitz by Lemma \ref{lem:u_diff}, $\lim_{k\rightarrow\infty}u_s^{(n_k)} (\Phi_s^{(n_k)}(z)) = u_s(\Phi_s(z))$ for all $z\in\Omg_{\bt}$. Lemma \ref{1.1} and the dominated convergence theorem imply
\begin{equation*}
    \Phi_t (z) = z + \lim_{k\rightarrow\infty}\int_0^t u_s^{(n_k)} (\Phi_s^{(n_k)}(z)) ds = z + \int_0^t u_s (\Phi_s(z)) ds,
\end{equation*}
which is equivalent to $\Phi_t = \phi_t$.
\end{proof}

\section{Well-posedness in weighted H\"older spaces}\label{Holder-wellposed}

In this section, we prove Theorem \ref{thm:GWPHolder}. The estimates of $\|\nb u\|_{L^{\infty}}$ and $[\nb u]_{\mathring{C}^{k,\alp}}$ play important roles in the proof of the H\"older well-posedness results. Before the proof, we introduce the estimates for each norm and how they are used.
\begin{itemize}
    \item In Subsection \ref{GWPHolder}, if $0<\alp\le 1$ and $\omg\in \mathring{C}^{0,\alp}(\overline{\Omg_{\bt}})$, then the norm $\|\nb u\|_{L^{\infty}}$ is bounded by
    \begin{equation*}
        \|\nb u\|_{L^{\infty}} \lesssim_{\alp,\bt} \|\omg\|_{L^{\infty}} \log \frac{c\|\omg\|_{\mathring{C}^{0,\alp}}}{\|\omg\|_{L^{\infty}}}.
    \end{equation*}
    This guarantees the Lipschitz property of the velocity and serves as a basic ingredient for all of our H\"older well-posedness results. We cannot expect an analogous estimate of $\|\nb u\|_{L^{\infty}}$ when we assume $\omg\in C^{0,\alp}(\overline{\Omg_{\bt}})$ without decay at infinity. 
    \item In Subsection \ref{Holder_estimate}, the seminorm $[\nb u]_{\mathring{C}^{k,\alp}}$ is bounded by $[\nb u]_{\mathring{C}^{k,\alp}} \lesssim \|\omg\|_{\mathring{C}^{k,\alp}}$ if $k\in\mathbb{N}_0$, $0<\alp<1$, and $\omg\in \mathring{C}^{k,\alp}(\overline{\Omg_{\bt}})$. This helps us to prove the $\mathring{C}^{k+1,\alp}$ well-posedness result.
\end{itemize}

Since $u$ is a complex function, we first specify the convention for $\nb u$. Recall that, for $z=x+iy$, the partial derivatives satisfy $\rd_x u = \rd_z u + \rd_{\bar{z}} u$ and $\rd_y u = i(\rd_z u - \rd_{\bar{z}} u)$.
\begin{definition}
    Let $u$ be the complex-valued velocity function, and $K_{\bt}$ be the Biot-Savart kernel in the domain $\Omg_{\bt}$. Then $\nb u$ and $P_{\bt}$ are defined by
    \begin{align*}
        \nb u := \begin{pmatrix}
            \rd_z u + \rd_{\bar{z}} u  \\ i(\rd_z u - \rd_{\bar{z}} u)
        \end{pmatrix},\quad
        P_{\bt} :=\begin{pmatrix}
            \rd_z K_{\bt} + \rd_{\bar{z}} K_{\bt} \\ i(\rd_z K_{\bt} - \rd_{\bar{z}} K_{\bt})
        \end{pmatrix} = 2\rd_{\bar{z}}^2 G_{\bt}(z,w) \begin{pmatrix}
            i \\ 1
        \end{pmatrix}.
    \end{align*}
    We define their absolute values $|\nb u|$ and $|P_{\bt}|$ as the maximum of the absolute values of their entries. We call $\nb u$ and $P_{\bt}$ the \textit{gradient} of $u$ and $K_{\bt}$, respectively. The higher order derivatives $\nb^k u$ can be defined similarly using $\rd_x u = \rd_z u + \rd_{\bar{z}} u$ and $\rd_y u = i(\rd_z u - \rd_{\bar{z}} u)$.
\end{definition}
Lemma \ref{lem:gradu_pv} and $u = 2i\rd_{\bar{z}}\psi$ imply that
\begin{align}\label{eq:gradu_pv}
    \nb u(z) = p.v.\int_{\Omg_{\bt}} P_{\bt}(z,w) \omg(w) dA(w) +\frac{\omg(z)}{2}\begin{pmatrix}
            i \\ -1
        \end{pmatrix}.
\end{align}
Note that $P_{\bt}$ is a constant multiple of $\rd_{\bar{z}}^2 G_{\bt}$. Throughout this paper, we use $P_{\bt}$ and $\rd_{\bar{z}}^2 G_{\bt}$ interchangeably, suppressing the constant matrix factor.

\subsection{Estimate for the Lipschitz norm of the velocity}\label{GWPHolder}
We begin by proving an estimate of $\|\nb u\|_{L^{\infty}}$ motivated by \cite{MB}, \cite{EJB}. This estimate of $\|\nb u\|_{L^{\infty}}$ helps us to prove all of our main well-posedness results in Sections \ref{Holder-wellposed} and \ref{Holder-wellposed2}.

\begin{lemma}\label{lem:gradu_infty}
    Suppose that $0<\bt<\pi/2$. Let $0<\alp\le 1$, $\omg\in \mathring{C}^{0,\alp}(\overline{\Omg_{\bt}})$, and $u$ satisfies the Biot-Savart law \eqref{eq:Biot-Savart}. Then
    \begin{equation}\label{eq:gradu_infty}
        \|\nb u\|_{L^{\infty}} \lesssim_{\alp,\bt} \|\omg\|_{L^{\infty}}\log\frac{c\|\omg\|_{\mathring{C}^{0,\alp}}}{\|\omg\|_{L^{\infty}}}
    \end{equation}
    for some constant $c>3$.
\end{lemma}
\begin{proof}
    By Lemma \ref{lem:IntGradKer} and \eqref{eq:gradu_pv}, we only need to bound
    \begin{equation*}
        \left|\int_{\Omg_{\bt}} \rd_{\bar{z}}^2 G_{\bt}(z,w)\{\omg(z)-\omg(w)\} dA(w)\right|
    \end{equation*}
    instead of $|\nb u(z)|$.
    Let $0<\dlt<1/2$, and consider three regions: $A:=\{w\in\Omg_{\bt}: |z-w| \le \dlt|z|\}$, $B:=\{w\in\Omg_{\bt}:\dlt|z|\le |z-w| \le 2|z|\}$, and $C:=\{w\in\Omg_{\bt}: |z-w| \ge 2|z|\}$. Then, in the region $A$,
    \begin{equation*}
        \left|\int_A \rd_{\bar{z}}^2 G_{\bt} (z,w) \{\omg(z)-\omg(w)\}dA(w)\right| \lesssim_p \int_A \frac{1}{|z-w|^2} \frac{|z-w|^{\alp}}{|z|^{\alp}} \|\omg\|_{\mathring{C}^{0,\alp}} dA(w) \lesssim \frac{\|\omg\|_{\mathring{C}^{0,\alp}}}{\alp} \dlt^{\alp}.
    \end{equation*}
    In the regions $B$ and $C$,
    \begin{align*}
        & \left|\int_B \rd_{\bar{z}}^2 G_{\bt} (z,w) \{\omg(z)-\omg(w)\}dA(w)\right| \lesssim_p \int_B \frac{1}{|z-w|^2} \|\omg\|_{L^{\infty}} dA(w) \lesssim \|\omg\|_{L^{\infty}} \log\left(\frac{2}{\dlt}\right), \\ & \left|\int_C \rd_{\bar{z}}^2 G_{\bt} (z,w) \{\omg(z)-\omg(w)\} dA(w)\right| \lesssim \int_C \frac{|z|^{p-2}}{|w|^p} \|\omg\|_{L^{\infty}} dA(w) \lesssim \frac{\|\omg\|_{L^{\infty}}}{p-2}.
    \end{align*}
    Therefore, if we choose
    \begin{equation*}
        \dlt = \min\left(\frac{1}{2}, \left(\frac{\|\omg\|_{L^{\infty}}}{\|\omg\|_{\mathring{C}^{0,\alp}}}\right)^{\frac{1}{\alp}}\right),
    \end{equation*}
    then we obtain \eqref{eq:gradu_infty}.
\end{proof}

With the help of Lemma \ref{lem:gradu_infty}, we first prove the $\mathring{C}^{0,\alp}$ well-posedness in $\Omg_{\bt}$.

\begin{theorem}[Theorem \ref{thm:GWPHolder} in the case $k=0$]\label{thm:ring_C_0,a}
    Suppose that $0<\bt<\pi/2$. Let $0<\alp\le 1$, and $\omg_0\in \mathring{C}^{0,\alp}(\overline{\Omg_{\bt}})$. Then, for any $T>0$, the unique weak solution in Theorem \ref{thm:weak-wellposed} satisfies
    \begin{align*}
        \omg \in L^{\infty}([0,T); \mathring{C}^{0,\alp}(\overline{\Omg_{\bt}}))
    \end{align*}
    with initial data $\omg_0$. The norm $\|\omg(t,\cdot)\|_{\mathring{C}^{0,\alp}}$ is bounded by
        \begin{equation*}
            \frac{c\|\omg(t,\cdot)\|_{\mathring{C}^{0,\alp}}}{\|\omg_0\|_{L^{\infty}}} \le  \left(\frac{c\|\omg_0\|_{\mathring{C}^{0,\alp}}}{\|\omg_0\|_{L^{\infty}}}\right)^{  \exp\left(C \|\omg_0\|_{L^{\infty}} t\right)}
        \end{equation*}
    in $\Omg_{\bt}$, where $c>3$ and $C=C(\alp,\bt)$.
\end{theorem}

\begin{proof}
    Let $(\omg_t, u_t, \Phi_t)$ be the unique solution triple defined in the proof of Theorem \ref{thm:weak-wellposed}. To use the estimate \eqref{eq:gradu_infty} of $\|\nb u\|_{L^{\infty}}$, we first note that the flow map $\Phi$ satisfies
    \begin{equation*}
        \left|\frac{d}{dt}(\Phi_t(z) - \Phi_t(z'))\right| = \left|u_t(\Phi_t(z)) - u_t(\Phi_t(z'))\right| \le \|\nb u_t\|_{L^{\infty}} |\Phi_t(z) - \Phi_t(z')|.
    \end{equation*}
    Then, Gr\"onwall's inequality implies that
    \begin{align}\label{eq:Phi_Lip}
        \left|\log\frac{|\Phi_t(z) - \Phi_t(z')|}{|z - z'|}\right|  \le \int_0^t \|\nb u_s\|_{L^{\infty}} ds.
    \end{align}
    Let $z$, $z'\in\Omg_{\bt}$, $t>0$, and define $w:= \Phi_t^{-1}(z)$, $w':= \Phi_t^{-1}(z')$. By Lemma \ref{1.1} and \eqref{eq:Phi_Lip},
    \begin{align}\label{eq:Phi_Lip_2}
        \left|\log\frac{|w - w'|}{|z-z'|}\right| \le \int_0^t \|\nb u_s\|_{L^{\infty}} ds, \quad \left|\log\frac{|w|}{|z|}\right| \le \int_0^t \|\nb u_s\|_{L^{\infty}} ds.
    \end{align}
    From these inequalities and $||z|^{\alp}-|z'|^{\alp}|\le |z-z'|^{\alp}$,
    \begin{equation}\label{eq:omg-Holder}
        \begin{split}
            & \left||z|^{\alp} \omg_t(z)-|z'|^{\alp}\omg_t(z')\right| \\ &\quad \le \left||z|^{\alp}-|z'|^{\alp}\right| \left|\omg_0(w')\right| + \frac{|z|^{\alp}}{|w|^{\alp}} \left||w|^{\alp} \omg_0(w) - |w'|^{\alp} \omg_0(w')\right| + \frac{|z|^{\alp}}{|w|^{\alp}}\left||w|^{\alp} - |w'|^{\alp}\right| \left| \omg_0(w) \right| \\ &\quad \le |z-z'|^{\alp}\left(\|\omg_0\|_{\mathring{C}^{0,\alp}}\exp\left(\alp\int_0^t \|\nb u_s\|_{L^{\infty}} ds \right)\right) \\ &\quad \le |z-z'|^{\alp}\left(\|\omg_0\|_{\mathring{C}^{0,\alp}}\exp\left( C\alp\int_0^t \|\omg_0\|_{L^{\infty}} \log\frac{c\|\omg_s\|_{\mathring{C}^{0,\alp}}}{\|\omg_0\|_{L^{\infty}}} ds \right)\right),
        \end{split}
    \end{equation}
    where $c$ is the constant in Lemma \ref{lem:gradu_infty}, and $C= C(\alp,\bt)$ depends only on $\alp$ and $\bt$. The estimate \eqref{eq:gradu_infty} of $\|\nb u\|_{L^{\infty}}$ is used in the last inequality. Since
    \begin{align*}
        \log\frac{c\|\omg_t\|_{\mathring{C}^{0,\alp}}}{\|\omg_0\|_{L^{\infty}}} \le \log\frac{c\|\omg_0\|_{\mathring{C}^{0,\alp}}}{\|\omg_0\|_{L^{\infty}}} + \int_0^t C\alp\|\omg_0\|_{L^{\infty}} \log\frac{c\|\omg_s\|_{\mathring{C}^{0,\alp}}}{\|\omg_0\|_{L^{\infty}}} ds,
    \end{align*}
    Gr\"onwall's lemma yields that
    \begin{align*}
        \log\frac{c\|\omg_t\|_{\mathring{C}^{0,\alp}}}{\|\omg_0\|_{L^{\infty}}} \le \log\frac{c\|\omg_0\|_{\mathring{C}^{0,\alp}}}{\|\omg_0\|_{L^{\infty}}} e^{C\alp\|\omg_0\|_{L^{\infty}}t}.
    \end{align*}
    This completes the proof of existence. Uniqueness follows from the uniqueness part of Theorem \ref{thm:weak-wellposed}.
\end{proof}

\subsection{Schauder-type estimates in weighted H\"older spaces}\label{Holder_estimate}

We now proceed to the cases $k\ge 1$ of Theorem \ref{thm:GWPHolder}. It is necessary to obtain the Schauder-type estimate of $\|\nb u\|_{\mathring{C}^{k,\alp}}$. We first estimate $\|\nb u\|_{\mathring{C}^k}$.

Henceforth, we assume that $u$ is the velocity that satisfies the Biot-Savart law \eqref{eq:Biot-Savart}. We will often suppress the variables on which the implicit constant in $\lesssim$ depends.

\begin{lemma}\label{lem:vel_Lip}
     Suppose that $0<\bt<\pi/2$. Let $k\in\mathbb{N}$, $0<\alp\le 1$, and $\omg\in \mathring{C}^{k,\alp}(\overline{\Omg_{\bt}})$. Then
     \begin{equation*}
        \||\cdot|^k \nb^{k+1} u\|_{L^{\infty}} \lesssim_{k,\alp} \|\omg\|_{\mathring{C}^{k,\alp}}.
     \end{equation*}
\end{lemma}

\begin{proof}
    \noindent\textit{Step 1.} Let $z\in\Omg_{\bt}$ with $|z|=1$. We divide $\Omg_{\bt}$ into two regions:  $A:=\{w\in\Omg_{\bt}: |z-w| \le 2|z|\}$, $B:=\{w\in\Omg_{\bt}: |z-w|> 2|z|\}$. In the region $A$, since $|z|$ and $|w|$ are comparable,
    \begin{equation*}
        \left||z|^k\int_A \rd_{\bar{z}}^{k+2} G_{\bt}(z,w) \{\omg(w) - T_{k,z}\omg(w)\} dA(w)\right| \lesssim \int_A \frac{1}{|z-w|^{k+2}}\frac{|z-w|^{k+\alp}\|\omg\|_{\mathring{C}^{k,\alp}}}{|z|^{\alp}}dA(w) \lesssim \|\omg\|_{\mathring{C}^{k,\alp}}
    \end{equation*}
    by Lemma \ref{lem:Taylor}. In the region $B$, we obtain
    \begin{align*}
        \left||z|^k \int_B \rd_{\bar{z}}^{k+2} G_{\bt} (z,w) \omg(w) dA(w)\right| \lesssim |z|^k\int_B \frac{|z|^{p-(k+2)} \|\omg\|_{L^{\infty}}}{|w|^p}dA(w)  \lesssim \|\omg\|_{L^{\infty}}.
    \end{align*}
    Then, Lemma \ref{lem:gradu_pv2} implies 
    \begin{equation*}
        \left| |z|^k \rd_{\bar{z}}^{k+2} \psi(z)\right| \lesssim [\omg]_{\mathring{C}^{k,\alp}} + \|\omg\|_{L^{\infty}} + |R_k[\omg](z)| \lesssim \|\omg\|_{\mathring{C}^{k,\alp}}.
    \end{equation*}
    Using $4\rd_z\rd_{\bar{z}}\psi = \omg$, the estimates of the other partial derivatives of $\psi$ are similarly obtained.

    \medskip
    \noindent\textit{Step 2.} We now assume that $|z|\neq 1$. Note that $u(|z|\cdot)/|z|$ is the velocity corresponding to the vorticity $\omg(|z|\cdot)$. Let $\zt := z/|z|$. Since
    \begin{align*}
        |z|^k p.v.\int_{\Omg_{\bt}} \rd_{\bar{z}}^{k+2} G_{\bt}(z,w) \omg(w) dA(w)   =  p.v.\int_{\Omg_{\bt}} \rd_{\bar{z}}^{k+2} G_{\bt}\left(\zt,\xi\right) \omg(|z|\xi) dA(\xi), \qquad \xi :=w/|z|,
    \end{align*}
    the velocity $u$ satisfies $\nb^{k+1} [u(|z|\cdot)/|z|] (\zt)=|z|^k \nb^{k+1} u(z)$ by \eqref{eq:gradu_pv}. Therefore, we can use Step 1 with $|\zt|=1$ to obtain
    \begin{align*}
        ||z|^k \nb^{k+1} u(z)| = |\nb^{k+1} [u(|z|\cdot)/|z|](\zt)| \lesssim \| \omg(|z|\cdot)\|_{\mathring{C}^{k,\alp}}  = \| \omg\|_{\mathring{C}^{k,\alp}}.
    \end{align*}
    The last equality comes from the scale-invariance of $\mathring{C}^{k,\alp}$.
\end{proof}

\begin{remark}
    Applying the proof of Lemma \ref{lem:vel_Lip} in the case $k=0$, we obtain $\|\nb u\|_{L^{\infty}} \lesssim \|\omg\|_{\mathring{C}^{0,\alp}}$.   
\end{remark}

\begin{lemma}\label{lem:Taylor}
    Suppose that $0<\bt<\pi/2$. Let $k\in\mathbb{N}_0$, $0\le \alp\le 1$, and $f\in \mathring{C}^{k,\alp}(\overline{\Omg_{\bt}})$. 
    If $|z-w|\le 2|z|$, then
    \begin{equation*}
        |f(w) - T_{k,z} f(w)| \lesssim_{k,\bt,\alp} \frac{|z-w|^{k+\alp}}{|z|^{k+\alp}}\|f\|_{\mathring{C}^{k,\alp}}.
    \end{equation*}
\end{lemma}

\begin{proof}
    Define a function $g$ on $\Omg_{\bt}$ by
    \begin{equation*}
        g(w) := |w|^{k+\alp} \nb^k f(w).
    \end{equation*}
    Let $Z\subset \Omg_{\bt}$ be defined by $Z := \{y\in\Omg_{\bt}: |z-y| \le |z|/2\}$, and $y\in Z$. Then
    \begin{equation*}
        \nb^k f(y) - \nb^k f(z) = |z|^{-(k+\alp)} (g(y) - g(z)) +g(y) (|y|^{-(k+\alp)} - |z|^{-(k+\alp)}).
    \end{equation*}
    Since $f\in\mathring{C}^{k,\alp}$ and $|y|<2|z|$, two inequalities $|g(y) - g(z)| \lesssim |y-z|^{\alp} \|f\|_{\mathring{C}^{k,\alp}}$ and
    \begin{equation*}
        \left| |y|^{-(k+\alp)} -|z|^{-(k+\alp)}\right| \lesssim |z|^{-(k+\alp)-1}|y-z| \lesssim |z|^{-(k+\alp)-\alp} |y-z|^{\alp}
    \end{equation*}
    hold by the complex mean value inequality. Therefore,
    \begin{equation*}
        |\nb^k f(y) - \nb^k f(z)| \lesssim \frac{|y-z|^{\alp}}{|z|^{k+\alp}}[f]_{\mathring{C}^{k,\alp}}.
    \end{equation*}
    Since $z+t(w-z)\in Z$ for all $0<t<1$, Taylor's theorem yields that
    \begin{equation*}
        |f(w) - T_{k,z}f(w)|\lesssim |z-w|^k \int_0^1 (1-t)^{k-1} |\nb^k f(z+t(w-z)) - \nb^k f(z)| dt \lesssim \frac{|z-w|^{k+\alp}}{|z|^{k+\alp}}[f]_{\mathring{C}^{k,\alp}}
    \end{equation*}
    for $k\ge 1$. The case $k=0$ follows from the definition of $\mathring{C}^{0,\alp}$.
    
    If $|z|/2 < |z-w| < 2|z|$, then $|f(w)|\lesssim \|f\|_{L^{\infty}}$ and $|(w-z)^m \rd_z^m f(z)| \lesssim |w-z|^m \|f\|_{\mathring{C}^m} /|z|^m\lesssim \|f\|_{\mathring{C}^m}$ for all $0\le m\le k$. Therefore, we obtain $|f(w) - T_{k,z} f(w)| \lesssim \|f\|_{\mathring{C}^{k,\alp}}$.
\end{proof}

\begin{lemma}\label{lem:gradu_bound}
     Suppose that $0<\bt<\pi/2$. Let $k\in\mathbb{N}_0$, $0<\alp< 1$, and $\omg\in \mathring{C}^{k,\alp}(\overline{\Omg_{\bt}})$. Then
     \begin{equation*}
        [\nb u]_{\mathring{C}^{k,\alp}} \lesssim_{k,\alp,\bt} \|\omg\|_{\mathring{C}^{k,\alp}}.
     \end{equation*}
\end{lemma}

\begin{proof}
    Let $z, z'\in\Omg_{\bt}$. We first consider the case $|z|=1$. Without loss of generality, assume that $|z'|\le |z|$.  If $|z-z'|>\bt/10$, then
    \begin{equation}\label{eq:z_diff_large}
        \frac{||z|^{k+\alp}\nb^{k+1} u(z) - |z'|^{k+\alp} \nb^{k+1} u(z')|}{|z-z'|^{\alp}} \lesssim \|\nb u\|_{\mathring{C}^k}.
    \end{equation}
    We now suppose that $|z-z'|<\bt/10$. Let $K := \{w\in\Omg_{\bt}: |z-w|\le 2|z-z'|\}$. Then
    \begin{equation*}
        \left| \int_K \rd_{\bar{z}}^{k+2} G_{\bt}(z,w) \left(\omg(w) - T_{k,z}\omg(w)\right) dA(w) \right| \lesssim \int_K \frac{\|\omg\|_{\mathring{C}^{k,\alp}} dA(w)}{|z-w|^{2-\alp}} \lesssim \|\omg\|_{\mathring{C}^{k,\alp}}|z-z'|^{\alp}.
    \end{equation*}
    We note that the condition $\alp>0$ is necessary.
    
    We now obtain the desired estimate in the region $\Omg_{\bt}\setminus K$. Without loss of generality, assume $0<\arg(z)\le \bt/2$. We divide $\Omg_{\bt}\setminus K$ into two regions:  $B:=\{w\in\Omg_{\bt}: 2|z-z'|\le |z-w| \le 2\}$, $C:=\{w\in\Omg_{\bt}: |z-w|> 2\}$. In the region $B$, let $I_1$ and $I_2$ be defined by
    \begin{equation*}
        \begin{split}
            &\int_B \left( \rd_{\bar{z}}^{k+2} G_{\bt} (z,w) (\omg (w) - T_{k,z}\omg(w)) - \rd_{\bar{z}}^{k+2} G_{\bt} (z',w) (\omg(w) - T_{k,z'}\omg(w))\right) dA(w) \\ &\ = \int_B \rd_{\bar{z}}^{k+2} G_{\bt} (z,w) (T_{k,z}\omg(w)-T_{k,z'}\omg(w)) dA(w)\\ &\qquad + \int_B (\rd_{\bar{z}}^{k+2} G_{\bt} (z,w) - \rd_{\bar{z}}^{k+2} G_{\bt} (z',w)) (\omg(w)-T_{k,z}\omg(w)) dA(w) =: I_1+I_2.
        \end{split}
    \end{equation*}

    First, we obtain an estimate of $|I_1|$. The Taylor expansion of $T_{k,z}\omg(w) - T_{k,z'}\omg(w)$ is
    \begin{equation}\label{eq:Taylor_diff}
        T_{k,z}\omg(w) - T_{k,z'}\omg(w) = \sum_{0\le m+n\le k} \frac{\rd_z^m \rd_{\bar{z}}^n \omg(z) - T_{k-(m+n),z'}(\rd_z^m \rd_{\bar{z}}^n \omg)(z)}{m!n!} (w-z)^m (\overline{w-z})^n,
    \end{equation} 
    and the Taylor coefficient satisfies
    \begin{equation*}
        |\rd_z^m \rd_{\bar{z}}^n \omg(z) - T_{k-(m+n),z'}(\rd_z^m \rd_{\bar{z}}^n \omg)(z)| \lesssim \|\omg\|_{\mathring{C}^{k,\alp}} |z-z'|^{k+\alp - (m+n)}.
    \end{equation*}
    Therefore,
    \begin{equation*}
        \begin{split}
            \int_B \left|\rd_{\bar{z}}^{k+2} G_{\bt} (z,w)\right| \left|T_{k-1,z}\omg(w)-T_{k-1,z'}\omg(w)\right| dA(w) & \lesssim \|\omg\|_{\mathring{C}^{k,\alp}} \sum_{m=0}^{k-1}\int_B \frac{|z-z'|^{(k+\alp) - m}}{|z-w|^{(k+2)-m}} dA(w) \\ & \lesssim \|\omg\|_{\mathring{C}^{k,\alp}} |z-z'|^{\alp}.
        \end{split}
    \end{equation*}
    Let $P_{k,z}\omg(w) := T_{k,z}\omg(w) - T_{k-1,z}\omg(w)$. By Lemma \ref{lem:HalfGradKer},
    \begin{equation}\label{eq:bound_I_1}
        \begin{split}
            \left|\int_B \rd_{\bar{z}}^{k+2} G_{\bt} (z,w) \left(P_{k,z}\omg(w) - P_{k,z'}\omg(w) \right) dA(w)\right|  \lesssim \|\omg\|_{\mathring{C}^{k,\alp}} |z-z'|^{\alp}.
        \end{split}
    \end{equation}
    By these two inequalities, we obtain $|I_1| \lesssim \|\omg\|_{\mathring{C}^{k,\alp}} |z-z'|^{\alp}$. We now obtain an estimate of $|I_2|$. By the complex mean value inequality,
    \begin{equation*}
        \begin{split}
            |I_2| & = \left|\int_B (\rd_{\bar{z}}^{k+2} G_{\bt} (z,w) - \rd_{\bar{z}}^{k+2} G_{\bt} (z',w)) (\omg(w)-T_{k,z}\omg(w)) dA(w)\right| \\ & \lesssim |z-z'|[\omg]_{\mathring{C}^{k,\alp}} \int_B \frac{1}{|z-w|^{k+3}} \frac{|z-w|^{k+\alp}}{|z|^{\alp}} dA(w) \lesssim \|\omg\|_{\mathring{C}^{k,\alp}}|z-z'|^{\alp}.
        \end{split}
    \end{equation*}
    We note that the condition $\alp<1$ is required in the above integral. In the region $C$,
    \begin{equation}\label{eq:bound_I_2}
        \left|\int_C (\rd_{\bar{z}}^{k+2} G_{\bt} (z,w) - \rd_{\bar{z}}^{k+2} G_{\bt} (z',w)) \omg(w) dA(w)\right| \lesssim |z-z'| \int_C \frac{\|\omg\|_{L^{\infty}}}{|w|^p} dA(w) \lesssim \|\omg\|_{\mathring{C}^{k,\alp}}|z-z'|^{\alp}.
    \end{equation}

    Therefore, by Lemma \ref{lem:gradu_pv2}, we obtain $|\rd_{\bar{z}}^{k+2}\psi(z) - \rd_{\bar{z}}^{k+2}\psi(z')| \lesssim \|\omg\|_{\mathring{C}^{k,\alp}}|z-z'|^{\alp}$ for all $z$, $z'\in\Omg_{\bt}$ with $|z'|\le |z|=1$ and $|z-z'|<\bt/10$. The case $|z|\neq 1$ follows from the proof of Step 2 in Lemma \ref{lem:vel_Lip}.
\end{proof}

\begin{remark}
    If $k\in\mathbb{N}$, then we cannot expect that $[\nb u]_{\mathring{C}^{k-1,1}} \lesssim \|\omg\|_{\mathring{C}^{k-1,1}}$ holds for all $\omg\in \mathring{C}^{k-1,1}(\overline{\Omg_{\bt}})$. However, Lemma \ref{lem:vel_Lip} implies that if $\alp\in (0,1]$ and $\omg\in \mathring{C}^{k,\alp}(\overline{\Omg_{\bt}})$, then we obtain
    \begin{equation*}
        [\nb u]_{\mathring{C}^{k-1,1}} \lesssim \| \nb u\|_{\mathring{C}^k} \lesssim \|\omg\|_{\mathring{C}^{k,\alp}}.
    \end{equation*}
\end{remark}

\begin{theorem}[Theorem \ref{thm:GWPHolder} in the case $k\ge 1$]\label{thm:ring_C_1,a}
    Suppose that $0<\bt<\pi/2$. Let $k\in\mathbb{N}$, $0<\alp\le 1$, and $\omg_0\in \mathring{C}^{k,\alp}(\overline{\Omg_{\bt}})$. Then, for any $T>0$, the unique weak solution in Theorem \ref{thm:weak-wellposed} satisfies
    \begin{align*}
        \omg \in L^{\infty}([0,T); \mathring{C}^{k,\alp}(\overline{\Omg_{\bt}}))
    \end{align*}
    with initial data $\omg_0$. The norm $\|\omg(t,\cdot)\|_{\mathring{C}^{k,\alp}}$ is bounded by
    \begin{equation*}
        \|\omg_t\|_{\mathring{C}^{k,\alp}} \le \|\omg_0\|_{\mathring{C}^{k,\alp}} \exp\left( C_1 \exp(C_2 \|\omg_0\|_{L^{\infty}}t) \right),
    \end{equation*}
    where $C_2 = C_2 (\alp, \bt)$, and
    \begin{equation*}
       C_1 = \left\{
            \begin{aligned}
                & C_1(\alp, \bt, \|\omg_0\|_{\mathring{C}^{k-1,\alp}}),\qquad\qquad\quad\ \alp\neq 1,\\
                & C_1(\bt, \|\omg_0\|_{\mathring{C}^{k-1,1}}, \|\omg_0\|_{\mathring{C}^{k,1/2}}), \qquad \alp=1.
            \end{aligned}
        \right.
    \end{equation*}
\end{theorem}
\begin{proof}
    We omit the proof of the estimate for the cases $k>1$ because the proof idea is the same as the case $k=1$.

    \noindent\textit{Step 1.} We first prove the desired a priori estimate of $\|\omg_t\|_{\mathring{C}^{1,\alp}}$. If $\nb \omg_t$ is well-defined, then the vorticity form of the two-dimensional Euler equations \eqref{eq:Euler_vorticity} implies
    \begin{equation}\label{eq:grad_omg}
        \frac{d}{dt}\nabla \omg_t(\Phi_t(z)) = -((\nabla u_t)(\nabla\omg_t))(\Phi_t(z)).
    \end{equation}
    Therefore,
    \begin{equation*}
        \frac{d}{dt}\||\cdot|\nabla\omg_t(\Phi_t(\cdot))\|_{L^{\infty}} \le \|\nabla u_t(\Phi_t(\cdot))\|_{L^{\infty}}\||\cdot|\nabla \omg_t(\Phi_t(\cdot))\|_{L^{\infty}}
    \end{equation*}
    if $\||\cdot|\nb\omg_t\|_{L^{\infty}}$ is finite. Since $\omg_0\in \mathring{C}^{0,\alp}$, by $\|\nb u\|_{L^{\infty}} \lesssim \|\omg\|_{\mathring{C}^{0,\alp}}$, 
    \begin{equation*}
        \||\cdot| \nb \omg_t\|_{L^{\infty}} \le \||\cdot| \nb \omg_0\|_{L^{\infty}} \exp\left(C_1 \exp(C_2 \|\omg_0\|_{L^{\infty}}t)\right),
    \end{equation*}
    where $C_1 = C_1(\alp, \bt, \|\omg_0\|_{\mathring{C}^{0,\alp}})$ and $C_2 =  C_2(\alp, \bt)$. Now we obtain an estimate of $[\omg_t]_{\mathring{C}^{1,\alp}}$. By \eqref{eq:grad_omg},
    \begin{align*}
        &\frac{d}{dt}\frac{\left||z|^{1+\alp}\nb \omg_t(\Phi_t(z)) - |z'|^{1+\alp}\nb\omg_t(\Phi_t(z'))\right|}{\left|\Phi_t(z) - \Phi_t(z')\right|^{\alp}} \\ &\ \le \left[|\cdot| (\nb u_t)(\nb\omg_t)\right]_{\mathring{C}^{0,\alp}} + \alp \|\nb u_t\|_{L^{\infty}} \frac{\left||z|^{1+\alp}\nb \omg_t(\Phi_t(z)) - |z'|^{1+\alp}\nb\omg_t(\Phi_t(z'))\right|}{\left|\Phi_t(z) - \Phi_t(z')\right|^{\alp}}.
    \end{align*}
    Since $\left[|\cdot| (\nb u_t)(\nb\omg_t)\right]_{\mathring{C}^{0,\alp}} \le \|\nb u_t\|_{L^{\infty}} \left[ \omg_t\right]_{\mathring{C}^{1,\alp}} + \left[\nb u_t\right]_{\mathring{C}^{0,\alp}}  \||\cdot| \nb \omg_t\|_{L^{\infty}}$,
    \begin{align*}
        \frac{d}{dt}[ \omg_t]_{\mathring{C}^{1,\alp}} \le \left[\nb u_t\right]_{\mathring{C}^{0,\alp}} \||\cdot|\nb \omg_t\|_{L^{\infty}} + (1+\alp) \|\nb u_t\|_{L^{\infty}} [\omg_t]_{\mathring{C}^{1,\alp}}.
    \end{align*}

    If $\alp\neq 1$, Lemma \ref{lem:gradu_bound} implies that $\left[\nb u_t\right]_{\mathring{C}^{0,\alp}} \lesssim \|\omg_t\|_{\mathring{C}^{0,\alp}} \lesssim \exp(C\exp(Ct))$. Therefore, Gr\"onwall's inequality implies
    \begin{equation*}
        \begin{split}
            & [\omg_t]_{\mathring{C}^{1,\alp}}\\ &\  \le [\omg_0]_{\mathring{C}^{1,\alp}} \exp \left( \int_0^t (1+\alp) \|\nb u_s\|_{L^{\infty}}ds \right) + \int_0^t [\nb u_s]_{\mathring{C}^{0,\alp}} \||\cdot| \nb \omg_s\|_{L^{\infty}} \exp\left( \int_s^t  (1+\alp) \|\nb u_r\|_{L^{\infty}} dr \right) ds \\ &\ \le \|\omg_0\|_{\mathring{C}^{1,\alp}} \exp\left( C_1 \exp(C_2 \|\omg_0\|_{L^{\infty}}t) \right),
        \end{split}
    \end{equation*}
    where $C_1 = C_1(\alp, \bt, \|\omg_0\|_{\mathring{C}^{0,\alp}})$ and $C_2 = C_2 (\alp, \bt)$ are constants. If $\alp=1$, then we use $\left[\nb u_t\right]_{\mathring{C}^{0,1}} \lesssim \|\omg_t\|_{\mathring{C}^{1,1/2}}$ and the case $\alp=1/2$ of Theorem \ref{thm:ring_C_1,a}. This completes the proof of the a priori estimate of $\|\omg_t\|_{\mathring{C}^{1,\alp}}$.

    \medskip
    \noindent\textit{Step 2.} Now we prove the existence part. Theorem \ref{thm:ring_C_0,a} implies that the solution $\omg$ exists and $\omg_t\in \mathring{C}^{0,\alp}$. From the proof of the a priori estimate, it is enough to prove that $\nb \omg_t$ exists and $\||\cdot|\nb \omg_t\|_{L^{\infty}}$ is finite.

    We first prove that $\nb \omg_t$ exists. Let $t>0$, $z\in\Omg_{\bt}$, $h\in\mathbb{R}\setminus\{0\}$, and $D^h_i$ be the difference quotient
    \begin{equation}\label{eq:diff_quotient}
        D^h_i f(z) := \frac{f(z+h e_i) - f(z)}{h},
    \end{equation}
    where $f$ is a complex function, $i\in\{1,2\}$, $e_1:=1$, and $e_2:=i$. Since $\rd_t \Phi_t = u\circ \Phi_t$, Taylor's theorem yields the formula
    \begin{equation*}
            D^h_i \Phi_t(z)= e_i + \int_0^t \int_0^1 \nb u_s\left(\Phi_s(z) + h\tht D^h_i \Phi_s(z)  \right) \cdot D_i^h \Phi_s(z) d\tht ds,
    \end{equation*}
    where $\nb f \cdot g := g\rd_z f + \bar{g}\rd_{\bar{z}}f$. By \eqref{eq:Phi_Lip_2}, there exists $\dlt>0$ such that $\{\Phi_s(z): 0<s<t\}\cap B_{\dlt}=\emptyset$. Since $\nb u\in L^{\infty}([0,t]; C^{0,\alp}(\Omg_{\bt}\cap B_{\dlt}^c))$, we obtain
    \begin{equation*}
        \int_0^t |\nb u_s\left(\Phi_s(z) + h\tht D^h_i \Phi_s(z)  \right) - \nb u_s\left(\Phi_s(z)\right)|ds \lesssim t\tht^{\alp} \|\nb u\|_{L^{\infty} C^{0,\alp}} \sup_{0\le s\le t}|\Phi_s(z+h e_i) - \Phi_s(z)|^{\alp} \rightarrow 0
    \end{equation*}
    as $h\rightarrow 0$. We note that the term $|\Phi_s(z+h e_i) - \Phi_s(z)|$ is bounded by \eqref{eq:Phi_Lip}.
    
    Let $A_i(t,z)$ be the unique solution of the ODE
    \begin{equation}\label{eq:ODE2}
        \rd_t A_i(t,z) = \nb u_t\left(\Phi_t(z)\right)\cdot A_i(t,z), \quad A_i(0,z) = e_i.
    \end{equation}
    Then, the difference of $D^h_i \Phi_t(z)$ and $A_i(t,z)$ satisfies the following estimate
    \begin{equation*}
        \begin{split}
            |D^h_i \Phi_t(z) - A_i(t,z)| & \le \int_0^t \sup_{0\le \tht\le 1}\left|\nb u_s\left(\Phi_s(z) + h\tht D^h_i \Phi_s(z)  \right)\right| |D^h_i \Phi_s(z) - A_i(s,z)| ds\\ &\qquad +\int_0^t \sup_{0\le \tht\le 1}\left|\nb u_s\left(\Phi_s(z) + h\tht D^h_i \Phi_s(z)  \right) - \nb u_s\left(\Phi_s(z)\right)\right| |A_i(s,z)| ds.
        \end{split}
    \end{equation*}
    Since $\sup_{0<s<t}|A_i(s,z)|$ is finite, we obtain $|D^h_i \Phi_t(z) - A_i(t,z)|\rightarrow 0$ as $h\rightarrow 0$ by Gr\"onwall's inequality. Therefore, $\nb \Phi_t$ exists and $\rd_{e_i} \Phi_t(z) = A_i(t,z)$ for each $z$. Since $\nb u_t\in C^{0,\alp}$ locally, the flow map $\Phi$ satisfies $\Phi_t\in C^{1,\alp}$ locally. Moreover, since $\nb\cdot u=0$, Liouville's formula and the inverse function theorem imply that $\Phi_t^{-1}\in C^{1,\alp}$ locally in $\Omg_{\bt}$.

    The same difference quotient argument on $D^h_i \omg_t$ yields that $\rd_{e_i} \omg_t(z)$ exists, and
    \begin{equation*}
        \rd_{e_i}\omg_t(z) = \nb \omg_0(\Phi_t^{-1}(z))\cdot \rd_{e_i}\Phi_t^{-1}(z).
    \end{equation*}
    By \eqref{eq:Phi_Lip_2} and \eqref{eq:ODE2}, there exists a constant $C_t$ such that
    \begin{equation*}
        \||\cdot| \nb \omg_t\|_{L^{\infty}(\Omg_{\bt})} \lesssim C_t \||\cdot| \nb \omg_0\|_{L^{\infty}(\Omg_{\bt})}.
    \end{equation*}
    This completes the proof of existence.
\end{proof}

\section{Applications to H\"older regularity}\label{Holder-wellposed2}

In this section, we prove Theorems \ref{thm:LWPHolder2} and \ref{thm:ill-posed}. We first summarize the H\"older estimates of $\nb u$ and how they are used for our well-posedness results (Theorem \ref{thm:LWPHolder2}).
\begin{itemize}
    \item In Subsection \ref{LWPHolder}, the seminorm $[\nb u]_{C^{k,\alp}}$ is bounded by $[\nb u]_{C^{k,\alp}}\lesssim \|\omg\|_{C^{k,\alp}}$ if $0<\alp<\min\{\pi/\bt -k - 2, 1\}$ and $\omg\in C^{0,\alp}(\overline{\Omg_{\bt}})$. We show the $C^{k+1,\alp}$ well-posedness with this estimate under $k+\alp<\pi/\bt - 2$ assumption.
    \item If $k + \alp = \pi/\bt - 2$, then the corner causes the estimate $[\nb u]_{C^{k,\alp}}\lesssim \|\omg\|_{C^{k,\alp}}$ to fail. Instead, in Subsection \ref{borderline}, the seminorm $[\nb u]_{C^{k,\alp}}$ is bounded by
    \begin{equation*}
        [\nb u]_{C^{k,\alp}}\lesssim_{k,\bt} \|\omg\|_{C^{k,\alp}} + \|\omg\|_{C^{k}}^{k+3-\pi/\bt} \|\omg\|_{C^{k,1}}^{\pi/\bt-k-2}.
    \end{equation*}
    Although this inequality is weaker than the previous one, it is sufficient for the proof of the $C^{k+1,\alp}$ well-posedness when $k+ \alp = \pi/\bt - 2$.
\end{itemize}

We next introduce how the singularity of $u$ at the origin occurs in our instantaneous blow-up results (Theorem \ref{thm:ill-posed}). The velocity can be decomposed into the most singular term at the origin and the smooth remainder. The behavior of the singularity differs depending on whether $\pi/\bt$ is an integer.
\begin{itemize}
    \item If $\pi/\bt$ is not an integer, the coefficient of the singular term is determined by the values of the vorticity throughout $\Omg_{\bt}$.
    \item If $\pi/\bt$ is an integer, the coefficient is determined by $\nb^{p-2} \omg(0)$, and values of $\omg$ on $\Omg_{\bt}\setminus\{0\}$ do not affect it.
\end{itemize}
In Subsection \ref{ill-posed2}, we construct initial data for which the singular term does not vanish, and the solution exhibits an instantaneous blow-up.

\subsection{H\"older regularity below the critical threshold}\label{LWPHolder}

In this subsection, we prove Theorem \ref{thm:LWPHolder2} under the condition $k+\alp<\pi/\bt - 1$. We note that the assumption $\mathring{C}^{0,\tld{\alp}}$ is needed to apply the estimate for $\|\nb u\|_{L^{\infty}}$.

\begin{theorem}[Theorem \ref{thm:LWPHolder2} in the case $k=0$]
    Suppose that $0<\bt<\pi/2$, $0 < \alp \le 1$, $0<\tld{\alp}<1$, and $\omg_0\in C^{0,\alp}\cap \mathring{C}^{0,\tld{\alp}}(\overline{\Omg_{\bt}})$. Then, for any $T>0$, there exists a unique solution
    \begin{align*}
        \omg \in L^{\infty}([0,T); C^{0,\alp}(\overline{\Omg_{\bt}}))
    \end{align*}
    to the Euler equations with initial data $\omg_0$. The norm $\|\omg(t,\cdot)\|_{C^{0,\alp}}$ is bounded by
    \begin{equation}\label{eq:Holder_cap}
        \frac{c\|\omg_t\|_{C^{0,\alp}}}{\|\omg_0\|_{L^{\infty}}} \le  \left(\frac{c\|\omg_0\|_{{C^{0,\alp}}\cap \mathring{C}^{0,\tld{\alp}}}}{\|\omg_0\|_{L^{\infty}}}\right)^{  \exp\left(C \|\omg_0\|_{L^{\infty}} t\right)}
    \end{equation}
    in $\Omg_{\bt}$, where $c>3$ and $C=C(\alp, \tld{\alp},\bt)$.
\end{theorem}
\begin{proof}
	As in the proof of Theorem \ref{thm:GWPHolder}, we obtain
	\begin{equation*}
		\begin{split}
			|\omg_t(z) - \omg_t(z')| & \le |z-z'|^{\alp} \|\omg_0\|_{C^{0,\alp}} \exp\left(\alp \int_0^t \|\nb u_s \|_{L^{\infty}} ds \right) \\ & \le |z-z'|^{\alp} \|\omg_0\|_{C^{0,\alp}} \exp\left(C\alp \int_0^t \|\omg_0\|_{L^{\infty}} \log \frac{c\|\omg_s\|_{\mathring{C}^{0,\tld{\alp}}}}{\|\omg_0\|_{L^{\infty}}}  ds \right),
		\end{split}
	\end{equation*}
	where $C=C(\tld{\alp}, \bt)$ is a constant. The last inequality follows from Lemma \ref{lem:gradu_infty}. Therefore,
	\begin{equation*}
		\|\omg_t\|_{C^{0,\alp}}\le \|\omg_0\|_{C^{0,\alp}} \exp\left(C\alp \int_0^t \|\omg_0\|_{L^{\infty}} \log \frac{c\|\omg_s\|_{\mathring{C}^{0,\tld{\alp}}}}{\|\omg_0\|_{L^{\infty}}}  ds \right).
	\end{equation*}
	Note that the integral in the right hand side is well-defined and finite by Theorem \ref{thm:GWPHolder}. This inequality and Theorem \ref{thm:ring_C_0,a} yield
	\begin{equation*}
		\frac{\|\omg_t\|_{C^{0,\alp} \cap \mathring{C}^{0,\tld{\alp}}}}{\|\omg_0\|_{L^{\infty}}} \le \frac{\|\omg_0\|_{C^{0,\alp} \cap \mathring{C}^{0,\tld{\alp}}}}{\|\omg_0\|_{L^{\infty}}} \exp\left(C\max\{\alp,\tld{\alp}\} \int_0^t \|\omg_0\|_{L^{\infty}} \log \frac{c \|\omg_s\|_{C^{0,\alp} \cap \mathring{C}^{0,\tld{\alp}}}}{\|\omg_0\|_{L^{\infty}}} ds \right).
	\end{equation*}
	Then, Gr\"onwall's lemma implies that
	\begin{equation*}
		\frac{c\|\omg_t\|_{C^{0,\alp} \cap \mathring{C}^{0,\tld{\alp}}}}{\|\omg_0\|_{L^{\infty}}} \le \exp\left( \log\frac{c\|\omg_0\|_{C^{0,\alp} \cap \mathring{C}^{0,\tld{\alp}}}}{\|\omg_0\|_{L^{\infty}}} \exp (C\max\{\alp,\tld{\alp}\}\|\omg_0\|_{L^{\infty}}t )\right),
	\end{equation*}
    and in particular,
    \begin{equation*}
        \frac{c\|\omg_t\|_{C^{0,\alp}}}{\|\omg_0\|_{L^{\infty}}} \le  \left(\frac{c\|\omg_0\|_{C^{0,\alp} \cap \mathring{C}^{0,\tld{\alp}}}}{\|\omg_0\|_{L^{\infty}}}\right)^{  \exp\left(C\max\{\alp,\tld{\alp}\} \|\omg_0\|_{L^{\infty}} t\right)}.
    \end{equation*}
    This completes the proof.
\end{proof}

We now begin to prove the case $k\ge 1$ of Theorem \ref{thm:LWPHolder2}. We first prove the $C^k$ analogue of Lemma \ref{lem:vel_Lip}.

\begin{lemma}\label{lem:gradu_higher1}
    Let $k\in\mathbb{N}$. Suppose that $0<\bt<\pi/(k+2)$, $0<\alp<1$, $\omg\in C^{k,\alp}(\overline{\Omg_{\bt}})$. Then,
    \begin{equation*}
        \|\nb^{k+1} u\|_{L^{\infty}} \lesssim_{k,\alp,\bt} \|\omg\|_{C^{k,\alp}}.
    \end{equation*}
\end{lemma}

\begin{proof}
    By Lemma \ref{lem:IntGradKer}, it is enough to obtain
    \begin{equation*}
        \left|\int_{\Omg_{\bt}} \rd_{\bar{z}}^{k+2} G_{\bt} (z,w) \{\omg(w)- T_{k-1,z} \omg(w)\} dA(w)\right|\lesssim \|\omg\|_{C^{k,\alp}}.
    \end{equation*}
    Let $R>0$, and divide $\Omg_{\bt}$ into two regions:  $A:=\{w\in\Omg_{\bt}: |z-w| \le R\}$, $B:=\{w\in\Omg_{\bt}: |z-w|> R\}$. In the region $A$,
    \begin{equation*}
        \left|\int_A \rd_{\bar{z}}^{k+2} G_{\bt}(z,w) \{\omg(w) - T_{k,z}\omg(w)\} dA(w)\right| \lesssim \int_A \frac{|z-w|^{k+\alp}[\omg]_{C^{k,\alp}}}{|z-w|^{k+2}}dA(w) \lesssim \frac{\|\omg\|_{C^{k,\alp}}}{\alp} R^{\alp}.
    \end{equation*}
    In the region $B$, since $p>k+2$, we can use a simple bound
    \begin{align*}
        \left|\int_B \rd_{\bar{z}}^{k+2} G_{\bt} (z,w) \{\omg(w)-T_{k-1,z}\omg(w)\} dA(w)\right| \lesssim \int_B \frac{|z-w|^{k-1}\|\omg\|_{C^{k-1}}}{|z-w|^{k+3}}  dA(w) \lesssim \frac{\|\omg\|_{C^{k-1}}}{R}.
    \end{align*}
    Choose $R$ to balance the preceding bounds as
    \begin{equation*}
        R := \left(\frac{\|\omg\|_{C^{k-1}}}{\|\omg\|_{C^{k,\alp}}}\right)^{1/(1+\alp)}.
    \end{equation*}
    The remainder term has a uniform bound
    \begin{equation*}
        \left|\int_A \rd_{\bar{z}}^{k+2} G_{\bt}(z,w) \{T_{k,z}\omg(w) - T_{k-1,z}\omg(w)\}dA(w) \right| \le C_p \|\omg\|_{C^k}
    \end{equation*}
    by Lemma \ref{lem:HalfGradKer}. In conclusion,
    \begin{equation*}
        \left|\int_{\Omg_{\bt}} \rd_{\bar{z}}^{k+2} G_{\bt} (z,w) \{\omg(w)- T_{k-1,z} \omg(w)\} dA(w)\right| \lesssim \frac{\|\omg\|_{C^{k,\alp}}}{\alp} R^{\alp} + \frac{\|\omg\|_{C^{k-1}}}{R} + \|\omg\|_{C^k} \lesssim \|\omg\|_{C^{k,\alp}}.
    \end{equation*}
\end{proof}

We note that the proof of Lemma \ref{lem:gradu_higher1} requires $k\ge 1$. We next prove the $C^{k,\alp}$ analogue of Lemma \ref{lem:gradu_bound}.

\begin{lemma}\label{lem:gradu_bound2}
    Let $k\in\mathbb{N}_0$. Suppose that $0<\bt<\pi/(k+2)$, $0<\alp < \min\{1,\pi/\bt - k-2\}$, $0<\tld{\alp}<1$, $\omg\in C^{k,\alp}\cap \mathring{C}^{0,\tld{\alp}}(\overline{\Omg_{\bt}})$. Then, $\nb u\in C^{k,\alp}$ and
    \begin{equation*}
        [\nb u]_{C^{k,\alp}} \lesssim_{k,\alp,\bt} \|\omg\|_{C^{k,\alp}}.
    \end{equation*}
\end{lemma}

\begin{remark}
    Unlike Lemma \ref{lem:gradu_bound}, we need an additional assumption $k+\alp<\pi/\bt - 2$. As we see in the proof, this is due to the corner singularity of the kernel.
\end{remark}

\begin{proof}
    Since $\omg\in \mathring{C}^{0,\tld{\alp}}$, Lemma \ref{lem:gradu_infty} yields the boundedness of $\|\nb u\|_{L^{\infty}}$. We follow the proof of Step 1 in Lemma \ref{lem:gradu_bound}, but we cannot assume $|z|=1$.
    
    We only note the difference from the proof of Lemma \ref{lem:gradu_bound}. Let $z'\in\Omg_{\bt}$ satisfy $|z-z'| <\bt|z|/10$ and $|z'|\le |z|$. By Lemma \ref{lem:IntGradKer}, it is enough to obtain
    \begin{equation*}
        \left|\int_{\Omg_{\bt}} \left(\rd_{\bar{z}}^{k+2} G_{\bt} (z,w) \{\omg(w)-T_{k,z}\omg(w)\} -\rd_{\bar{z}}^{k+2} G_{\bt} (z',w) \{\omg(w)-T_{k,z'}\omg(w)\}\right) dA(w)\right| \lesssim \|\omg\|_{C^{k,\alp}}|z-z'|^{\alp}.
    \end{equation*}
    Let $C:=\{w\in\Omg_{\bt}: |z-w|> 2|z|\}$. In the region $C$, instead of the estimate \eqref{eq:bound_I_2}, we use
    \begin{equation}\label{eq:bound_I_3}
        \begin{split}
            &\left|\int_C (\rd_{\bar{z}}^{k+2} G_{\bt} (z,w) - \rd_{\bar{z}}^{k+2} G_{\bt} (z',w)) (\omg(w)- T_{k,z}\omg(w)) dA(w)\right| \\ &\ \lesssim |z-z'| \int_C \frac{|z|^{p-(k+3)}}{|w|^p} |z-w|^{k+\alp}\|\omg\|_{C^{k,\alp}} dA(w) \lesssim \|\omg\|_{C^{k,\alp}}|z-z'|^{\alp}.
        \end{split}
    \end{equation}
    We note that the condition $k+\alp<\pi/\bt-2$ is required in \eqref{eq:bound_I_3}. The other parts of the proof are the same as those in the proof of Lemma \ref{lem:gradu_bound}.
\end{proof}

With the help of Lemma \ref{lem:gradu_bound2}, we can prove the $C^{k,\alp}$ well-posedness result under the condition $k+\alp \neq \pi/\bt - 1$.

\begin{theorem}[Theorem \ref{thm:LWPHolder2} in the case $k+\alp \neq \pi/\bt - 1$]\label{thm:LWPHolder2_1}
    Suppose that $0<\bt<\pi/2$, $k\in\mathbb{N}$, $0 < \alp \le 1$ with
    \begin{equation}
         k+\alp < \pi/\bt - 1,
    \end{equation}
    $0<\tld{\alp}<1$, and $\omg_0\in C^{k,\alp}\cap \mathring{C}^{0,\tld{\alp}}(\overline{\Omg_{\bt}})$. Then, for any $T>0$, there exists a unique solution
    \begin{align*}
        \omg \in L^{\infty}([0,T); C^{k,\alp}(\overline{\Omg_{\bt}}))
    \end{align*}
    to the Euler equations with initial data $\omg_0$. The norm $\|\omg(t,\cdot)\|_{C^{k,\alp}}$ is bounded by
    \begin{equation*}
        \|\omg_t\|_{C^{k,\alp}} \le \|\omg_0\|_{C^{k,\alp}} \exp\left( C_1 \exp(C_2 \|\omg_0\|_{L^{\infty}}t) \right),
    \end{equation*}
    where $C_2  = C_2(\alp, \tld{\alp}, \bt)$, and
    \begin{equation*}
       C_1 = \left\{
            \begin{aligned}
                & C_1(\alp, \tld{\alp}, \bt, \|\omg_0\|_{C^{k-1,\alp}}, \|\omg_0\|_{\mathring{C}^{0,\tld{\alp}}}),\qquad\qquad\quad\ \alp\neq 1,\\
                & C_1(\tld{\alp}, \bt, \|\omg_0\|_{C^{k-1,1}}, \|\omg_0\|_{C^{k,1/2}}, \|\omg_0\|_{\mathring{C}^{0,\tld{\alp}}}), \qquad \alp=1.
            \end{aligned}
        \right.
    \end{equation*}
\end{theorem}

\begin{proof}
    One can refer to Theorem \ref{thm:ring_C_1,a} for the proof of existence. We only prove the case $k=1$ of the a priori estimate. Suppose that $\nb \omg_t$ is well-defined. Since
    \begin{equation*}
        \frac{d}{dt}\|\nabla\omg_t(\Phi_t(\cdot))\|_{L^{\infty}} \le \|\nabla u_t(\Phi_t(\cdot))\|_{L^{\infty}}\|\nabla \omg_t(\Phi_t(\cdot))\|_{L^{\infty}}
    \end{equation*}
    and $\omg_0\in \mathring{C}^{0,\tld{\alp}}$,
    \begin{equation*}
        \|\nb \omg_t\|_{L^{\infty}} \le \|\nb \omg_0\|_{L^{\infty}} \exp(C_1\exp(C_2\|\omg_0\|_{L^{\infty}}t)),
    \end{equation*}
    where $C_1 = C_1(\tld{\alp}, \bt, \|\omg_0\|_{\mathring{C}^{0,\tld{\alp}}})$ and $C_2 =  C_2(\tld{\alp}, \bt)$.

    Since $\left[(\nb u_t)(\nb\omg_t)\right]_{C^{0,\alp}} \le \|\nb u_t\|_{L^{\infty}} \left[\nb \omg_t\right]_{C^{0,\alp}} + \left[\nb u_t\right]_{C^{0,\alp}}  \|\nb \omg_t\|_{L^{\infty}}$,
    \begin{align*}
        \frac{d}{dt}\|\nb \omg_t\|_{C^{0,\alp}} \le \left[\nb u_t\right]_{C^{0,\alp}} \|\nb \omg_t\|_{L^{\infty}} + (1+\alp) \|\nb u_t\|_{L^{\infty}} \|\nb\omg_t\|_{C^{0,\alp}}.
    \end{align*}
    If $\alp\neq 1$, \eqref{eq:Holder_cap} and Lemma \ref{lem:gradu_bound2} yield $\left[\nb u_t\right]_{C^{0,\alp}} \lesssim \|\omg_t\|_{C^{0,\alp}} \lesssim \exp(C\exp(Ct))$. Therefore, Gr\"onwall's inequality implies
    \begin{equation*}
        \begin{split}
            & [\nb \omg_t]_{C^{0,\alp}}\\ &\  \le [\nb \omg_0]_{C^{0,\alp}} \exp \left( \int_0^t (1+\alp) \|\nb u_s\|_{L^{\infty}}ds \right) + \int_0^t [\nb u_s]_{C^{0,\alp}} \|\nb \omg_s\|_{L^{\infty}} \exp\left( \int_s^t  (1+\alp) \|\nb u_r\|_{L^{\infty}} dr \right) ds \\ &\ \le \|\omg_0\|_{C^{1,\alp}} \exp\left( C_1 \exp(C_2 \|\omg_0\|_{L^{\infty}}t) \right),
        \end{split}
    \end{equation*}
    where $C_1 = C_1(\alp, \bt, \|\omg_0\|_{C^{0,\alp}\cap \mathring{C}^{0,\tld{\alp}}})$ and $C_2 = C_2 (\alp, \tld{\alp}, \bt)$ are constants. If $\alp=1$, then we use $[\nb u]_{C^{0,1}}\lesssim \|\omg_t\|_{C^{1,1/2}}$ and the case $\alp=1/2$ of Lemma \ref{lem:gradu_higher1}.
\end{proof}

\subsection{The well-posedness threshold}\label{borderline}

In this subsection, we assume that $\pi/\bt$ is not an integer, and prove the case $k+\alp = \pi/\bt - 1$ of Theorem \ref{thm:LWPHolder2} to complete the proof. If $k+\alp = \pi/\bt - 1$, then we cannot expect
\begin{equation*}
    [\nb u]_{C^{k-1,\alp}} \lesssim \|\omg\|_{C^{k-1,\alp}}
\end{equation*}
because the integral in \eqref{eq:bound_I_3} is no longer well-defined.

\begin{example}
    Suppose that $\bt = 2\pi/7$, $k=2$, and $\alp=1/2$. Let $\psi$ be the stream function defined by
    \begin{equation*}
        \psi(re^{i\tht}) :=  r^{7/2} \log (r) \sin (7\tht/2) \eta( re^{i\tht}),
    \end{equation*}
    where $\eta\in C_c^{\infty}(\overline{\Omg_{\bt}})$ is a cut-off function such that $\eta \equiv 1$ if $|z|<1$, and $\eta\equiv 0$ if $|z|>2$. Then,
    \begin{equation*}
        \omg (re^{i\tht}) = \Dlt \psi(re^{i\tht}) = 7 r^{3/2} \sin (7\tht/2)
    \end{equation*}
    in $\Omg_{\bt}\cap B_1$. Therefore $\omg\in C^{1,1/2}(\overline{\Omg_{\bt}})$ with compact support. However, $\psi\notin C^{3,1/2}$ and $\nb u$ also satisfies $\nb u = \nb \nb^{\perp} \psi \notin C^{1,1/2}$.
\end{example}

Instead, we prove an alternative inequality under an additional assumption $\omg\in C^{0,1}$. We first derive the most singular term of the derivatives of $G_{\bt}(\cdot,w)$ at $0$.

\begin{lemma}\label{lem:Green_decompose}
    Let $1\le k < \pi/\bt $ be an integer, and $w\in\Omg_{\bt}$. For $|z|<|w|$, the Green's function $G_{\bt}$ satisfies
    \begin{equation*}
        \rd_{\bar{z}}^k G_{\bt} (z,w) = \sum_{n=1}^{\infty} c_{k,n}(w) \bar{z}^{np-k},
    \end{equation*}
    where the constant $c_{k,n}(w)$ is defined by
    \begin{equation*}
        c_{k,n} (w) := \frac{np(np-1)\cdots (np-k+1)}{4\pi n} \left(w^{-np} - \bar{w}^{-np}\right).
    \end{equation*}
\end{lemma}

\begin{proof}
    Since $-\log(1-\zt) = \sum_{n=1}^{\infty} \zt^n/n$ for all $|\zt|<1$,  the power series representation of $G_{\bt}(\cdot,w)$ is
    \begin{equation*}
        \begin{split}
            G_{\bt}(z,w) & = -\frac{1}{2\pi}\sum_{n=1}^{\infty} \frac{1}{n}
            \left(\Re\left(\frac{z^{np}}{w^{np}}\right) - \Re\left(\frac{z^{np}}{\bar{w}^{np}}\right)\right) \\ & = \frac{1}{4\pi}\sum_{n=1}^{\infty} \frac{w^{-np} - \bar{w}^{-np}}{n} (\bar{z}^{np} - z^{np})
        \end{split}
    \end{equation*}
    if $|z|<|w|$. This series converges uniformly on every compact subset of $\{z\in\Omg_{\bt}: |z|<|w|\}$. Therefore,
    \begin{equation*}
        \rd_{\bar{z}}^k G_{\bt}(z,w) = \frac{1}{4\pi}\sum_{n=1}^{\infty} \rd_{\bar{z}}^k \left(\frac{w^{-np} - \bar{w}^{-np}}{n} (\bar{z}^{np} - z^{np})\right) = \sum_{n=1}^{\infty} c_{k,n}(w) \bar{z}^{np-k}.
    \end{equation*}
\end{proof}

In this subsection, we define $c_{0,n}(w)$ for $n\in\mathbb{N}$, and $\mathcal{R}_{\bt}(z,w)$ by
\begin{equation*}
    c_{0,n}(w) := \frac{w^{-np}-\bar{w}^{-np}}{4\pi n},\quad \mathcal{R}_{\bt}(z,w) := G_{\bt} (z,w) - c_{0,1}(w) (\bar{z}^{p} - z^{p}).
\end{equation*}

\begin{lemma}\label{lem:borderline}
    Let $k\in\mathbb{N}_0$. Suppose that $\pi/(k+3) < \bt < \pi/(k+2)$, 
    \begin{equation*}
        \alp := \pi/\bt - k - 2,
    \end{equation*}
    $0<\tld{\alp}<1$, and $\omg\in C^{k,1}\cap \mathring{C}^{0,\tld{\alp}}(\overline{\Omg_{\bt}})$. Then, $\nb u\in C^{k,\alp}$ and
    \begin{equation*}
        [\nb u]_{C^{k,\alp}} \lesssim_{k, \bt} \|\omg\|_{C^{k,\alp}} + \|\omg\|_{C^k}^{k+3-\pi/\bt} \|\omg\|_{C^{k,1}}^{\pi/\bt-k-2}.
    \end{equation*}
\end{lemma}

\begin{proof}
    We cannot follow the inequality \eqref{eq:bound_I_3} in the proof of Lemma \ref{lem:gradu_bound2}, because the integral
    \begin{equation*}
        \int_C \frac{|z|^{p-2}}{|w|^p} |z-w|^{k+\alp-1} dA(w)
    \end{equation*}
    requires $k+\alp < p - 1$ for integrability.

    Let $z\in\Omg_{\bt}$, and $|z-z'|<\bt|z|/10$. We define the set $C:= \{w\in\Omg_{\bt}: |z-w|>2|z|\}$ as in Lemma \ref{lem:gradu_bound2}, and denote $w=: \rho e^{i\phi}$ in polar coordinates with $0<\phi<\bt$. If $z$ satisfies $|z| \ge R := \|\omg\|_{C^k} / \|\omg\|_{C^{k,1}}$, then
    \begin{equation*}
        \begin{split}
            &\left|\int_C c_{k+2,1}(w)(\bar{z}^{p-k-2} - \bar{z}'^{p-k-2}) (\omg(w)- T_{k,z}\omg(w)) dA(w)\right| \\ &\ \lesssim \frac{p(p-1)\cdots (p-k-1)}{8\pi} |\bar{z}^{\alp} - \bar{z}'^{\alp}| \int_{3|z|/2}^{\infty} \int_0^{\bt} \frac{\sin(p\phi)}{\rho^p} |z-\rho e^{i\phi}|^k \|\omg\|_{C^k} \rho d\phi d\rho,\qquad w=:\rho e^{i\phi}, \\ &\ \lesssim \frac{\|\omg\|_{C^k}}{|z|^{p-k-2}}|z-z'|^{\alp} \lesssim \|\omg\|_{C^k}^{k+3-p} \|\omg\|_{C^{k,1}}^{p-k-2}|z-z'|^{\alp}.
        \end{split}
    \end{equation*}
    The inequality $|z^{\alp} - z'^{\alp}| \lesssim_{\alp} |z-z'|^{\alp}$ follows from Lemma \ref{lem:z_s_diff}. 
    
    We now suppose that $|z|\le R$ and follow the above inequality. In this case, we need to split the range of $\rho$ into $\rho<3R/2$ and $\rho>3R/2$. Then, the assumption $\omg\in C^{k,1}$ yields
    \begin{equation*}
        \begin{split}
            &\int_{3|z|/2}^{3R/2}\int_0^{\bt} \frac{\sin(p\phi)}{\rho^p} |z-\rho e^{i\phi}|^{k+1} \|\omg\|_{C^{k,1}} \rho d\phi d\rho + \int_{3R/2}^{\infty}\int_0^{\bt} \frac{\sin(p\phi)}{\rho^p} |z-\rho e^{i\phi}|^k \|\omg\|_{C^k} \rho d\phi d\rho  \\ &\ \lesssim  \|\omg\|_{C^{k,1}} (R^{k+3-p} - |z|^{k+3-p}) + \|\omg\|_{C^k} R^{k+2-p} \lesssim \|\omg\|_{C^k}^{k+3-p} \|\omg\|_{C^{k,1}}^{p-k-2}.
        \end{split}
    \end{equation*}

    We now estimate the remainder. In the region $C$, $|\rd_{\bar{z}}^{k+2} \mathcal{R}_{\bt}(z,w)| \lesssim |z|^{2p-k-2}/|w|^{2p}$. Therefore,
    \begin{equation*}
        \begin{split}
            & \left|\int_C (\rd_{\bar{z}}^{k+2} \mathcal{R}_{\bt}(z,w) - \rd_{\bar{z}}^{k+2} \mathcal{R}_{\bt}(z',w)) (\omg(w) - T_{k,z}\omg(w)) dA(w)\right| \\ &\ \lesssim |z-z'|\int_C \frac{|z|^{2p-k-3}}{|w|^{2p}}|z-w|^{k+\alp}\|\omg\|_{C^{k,\alp}} dA(w) \lesssim \|\omg\|_{C^{k,\alp}}|z-z'|^{\alp}.
        \end{split}
    \end{equation*}
    These three inequalities imply that
    \begin{equation}\label{eq:C}
        \left|\int_C (\rd_{\bar{z}}^{k+2} G_{\bt}(z,w) - \rd_{\bar{z}}^{k+2} G_{\bt}(z',w)) (\omg(w) - T_{k,z}\omg(w)) dA(w)\right| \lesssim \left(\|\omg\|_{C^{k,\alp}} + \|\omg\|_{C^k}^{k+3-p} \|\omg\|_{C^{k,1}}^{p-k-2}\right)|z-z'|^{\alp}.
    \end{equation}
    We replace \eqref{eq:bound_I_3} with the inequality above to finish the proof. One can repeat the other parts of the proof of Lemma \ref{lem:gradu_bound2} directly.
\end{proof}

We note the proof of $|z^{\alp} - z'^{\alp}| \lesssim_{\alp} |z-z'|^{\alp}$.

\begin{lemma}\label{lem:z_s_diff}
	Let $z_1$, $z_2$ be two nonzero complex numbers. Suppose that $\alp\in (0,1)$, and $0<\arg(z_1)$, $\arg(z_2) <\pi/(2+\alp)$. Then
    \begin{equation*}
		|z_1^{\alp} - z_2^{\alp}| \le \left(1 + 2^{1-\alp}\sup_{0<\dlt < \pi/(2+\alp)}  \left(\frac{\sin(\alp\dlt/2)}{\sin^{\alp}(\dlt/2)}\right)\right) |z_1-z_2|^{\alp}.
	\end{equation*}
\end{lemma}

\begin{proof}
	Let $z_1:=r_1e^{i\tht_1}$ and $z_2= r_2 e^{i\tht_2}$. Suppose that $r_2 \le r_1$. Then,
    \begin{equation*}
		|z_1^{\alp} - z_2^{\alp}|  \le |r_1^{\alp} - r_2^{\alp}| + r_2^{\alp}|e^{i\alp\tht_1} - e^{i\alp \tht_2}| \le |z_1 - z_2|^{\alp}+ 2 |z_2|^{\alp} \sin \frac{\alp|\tht_1 - \tht_2|}{2}.
    \end{equation*}
    Note that $r_2\le r_1$ implies
    \begin{equation*}
	    |z_1-z_2| \ge 2|z_2|\sin \frac{|\tht_1 - \tht_2|}{2}.
    \end{equation*}
    Therefore,
    \begin{equation*}
		|z_1^{\alp} - z_2^{\alp}| \le |z_1 - z_2|^{\alp} + 2\sin\frac{\alp|\tht_1-\tht_2|}{2} \left(2\sin\frac{|\tht_1-\tht_2|}{2}\right)^{-\alp} |z_1-z_2|^{\alp}.
    \end{equation*}
\end{proof}

By Lemma \ref{lem:borderline} and Theorem \ref{thm:LWPHolder2_1}, one can prove the case $k+\alp = \pi/\bt - 1$ so that the proof of Theorem \ref{thm:LWPHolder2} is done.

\begin{theorem}[Theorem \ref{thm:LWPHolder2} in the case $k+\alp = \pi/\bt - 1$]
    Suppose that $k\in\mathbb{N}$ and $\pi/(k+2)<\bt<\pi/(k+1)$. Let $0<\tld{\alp}<1$,
    \begin{equation*}
        \alp := \pi/\bt - k -1,
    \end{equation*}
    and $\omg_0\in C^{k,\alp}\cap \mathring{C}^{0,\tld{\alp}}(\overline{\Omg_{\bt}})$. Then, for any $T>0$, there exists a unique solution
    \begin{align*}
        \omg \in L^{\infty}([0,T); C^{k,\alp}(\overline{\Omg_{\bt}}))
    \end{align*}
    to the Euler equations with initial data $\omg_0$. The norm $\|\omg(t,\cdot)\|_{C^{k,\alp}}$ is bounded by
    \begin{equation*}
        \|\omg_t\|_{C^{k,\alp}} \le \|\omg_0\|_{C^{k,\alp}} \exp\left( C_1 \exp(C_2 \|\omg_0\|_{L^{\infty}}t) \right),
    \end{equation*}
    where $C_1 = C_1(\tld{\alp}, \bt, \|\omg_0\|_{C^{k-1,\alp}}, \|\omg_0\|_{C^{k-1,1}}, \|\omg_0\|_{\mathring{C}^{0,\tld{\alp}}})$, $C_2  = C_2(\tld{\alp}, \bt)$.
\end{theorem}

\subsection{Instantaneous loss of higher-order regularity}\label{ill-posed2}

In this subsection, we prove Theorem \ref{thm:ill-posed}. We use the notation from the preceding subsection. For all $n$, $k\in\mathbb{N}$,
\begin{equation*}
    c_{k,n} (w) := \frac{np(np-1)\cdots (np-k+1)}{4\pi n} \left(w^{-np} - \bar{w}^{-np}\right),\quad c_{0,n}(w) := \frac{w^{-np}-\bar{w}^{-np}}{4\pi n}, 
\end{equation*}
and $\mathcal{R}_{\bt}(z,w) := G_{\bt} (z,w) - c_{0,1}(w) (\bar{z}^{p} - z^{p})$.

First, suppose that $\pi/\bt$ is not an integer. Using Lemma \ref{lem:Green_decompose}, the velocity $u$ can be decomposed similarly.

\begin{lemma}\label{lem:psi_singular_term2}
    Let $k\in\mathbb{N}$. Suppose that $\pi/(k+2) < \bt < \pi/(k+1)$, and $\omg\in C^k(\overline{\Omg_{\bt}})$. Then, the velocity $u$ can be decomposed into
    \begin{equation*}
        u(z) = P_k(z) + C_k \bar{z}^{p-1} + R_k(z)
    \end{equation*}
    for all $z\in\Omg_{\bt}$, where the $k$-th order polynomial $P_k$ and the real constant $C_k$ satisfy
    \begin{equation*}
        P_k(z) := \int_{\Omg_{\bt}} K_{\bt}(z,w) T_{k-1} \omg(w) dA(w), \quad C_k := \frac{p}{\pi} \int_{\Omg_{\bt}} \frac{\sin (p \arg w)}{|w|^p} \left(\omg(w) - T_{k-1}\omg(w)\right) dA(w).
    \end{equation*}
    The remainder $R_k$ is defined by
    \begin{equation*}
        R_k(z) :=\frac{ip \bar{z}^{2p-1}}{2\pi} \int_{\Omg_{\bt}}  \left(\frac{1}{\bar{w}^p (\bar{z}^p - \bar{w}^p)} - \frac{1}{w^p (\bar{z}^p - w^p)}\right)\left(\omg(w) - T_{k-1}\omg(w)\right) dA(w)
    \end{equation*}
    and satisfies $|\rd_{\bar{z}}^m R_k(z)| \lesssim |z|^{k+1-m}$ for all $0\le m\le k$.
\end{lemma}

We note that $T_{k-1}\omg(w) := T_{k-1,0}\omg(w)$ and $T_{k-1,z}\omg(w)$ are different.

\begin{proof}
    \textit{Step 1.} We first prove the decomposition $u(z) = P_k(z) + C_k \bar{z}^{p-1} + R_k(z)$. Lemma \ref{lem:Gbeta_nm} implies that $P_k$ is a $k$-th order polynomial with coefficients determined by the derivatives of $\omg$ at $0$. The definition of $\mathcal{R}_{\bt}$ and $c_{1,1}$ implies that
    \begin{equation*}
        K_{\bt}(z,w) = 2i \rd_{\bar{z}} \mathcal{R}_{\bt}(z,w) + 2i c_{1,1}(w)\bar{z}^{p-1}
    \end{equation*}
    and
    \begin{equation*}
        \int_{\Omg_{\bt}} 2i c_{1,1}(w)\bar{z}^{p-1} \left(\omg(w) - T_{k-1}\omg(w)\right) dA(w) = C_k\bar{z}^{p-1}.
    \end{equation*}
    Therefore,
    \begin{equation*}
        2i\rd_{\bar{z}} \mathcal{R}_{\bt}(z,w) = K_{\bt}(z,w) - 2i c_{1,1}(w) \bar{z}^{p-1} = \frac{ip\bar{z}^{2p-1}}{2\pi} \left(\frac{1}{\bar{w}^p (\bar{z}^p - \bar{w}^p)} - \frac{1}{w^p (\bar{z}^p - w^p)}\right),
    \end{equation*}
    so that $u(z) - P_k(z)-C_k\bar{z}^{p-1}$ satisfies
    \begin{equation*}
        u(z) - P_k(z)-C_k\bar{z}^{p-1} = \int_{\Omg_{\bt}} 2i\rd_{\bar{z}} \mathcal{R}_{\bt}(z,w) \left(\omg(w) - T_{k-1}\omg(w)\right) dA(w) = R_k(z)
    \end{equation*}
    by the definition of $R_k$.

    \medskip
    \noindent\textit{Step 2.} Let $m\in [0,k]$ be an integer, and $z\in\Omg_{\bt}$ with $z=:re^{i\tht}$ and $0<\tht<\bt/2$. We now prove that $|\rd_{\bar{z}}^m R_k(z)| \lesssim |z|^{k+1-m}$. Let $\tld{\omg}(w) := \omg(w) - T_{k-1}\omg(w)$, and
    \begin{equation*}
        Q_{\bt}(z,w) := \bar{z}^{2p-1} \left(\frac{1}{\bar{w}^p (\bar{z}^p - \bar{w}^p)} - \frac{1}{w^p (\bar{z}^p - w^p)}\right).
    \end{equation*}
    Then,
    \begin{equation*}
        R_k(z) = c\int_{\Omg_{\bt}} Q_{\bt}(z,w) \tld{\omg}(w) dA(w)
    \end{equation*}
    for some constant $c$. Let $\zt := z/r$. We denote $\xi := w/r$ and $\tld{\omg}_r(\xi):= r^{-k} \tld{\omg}(r\xi)$. Then,
    \begin{equation*}
        R_k(z) = c\int_{\Omg_{\bt}} r^{-1} Q_{\bt}(\zt,\xi)r^k \tld{\omg}_r(\xi) r^2 dA(\xi),
    \end{equation*} 
    and $\rd_{\bar{z}}^m Q_{\bt}(z,w) = r^{-m-1} \rd_{\bar{z}}^m Q_{\bt} (\zt,\xi)$. Since $Q_{\bt}- \rd_{\bar{z}}G_{\pi}$ is bounded in $\{\xi\in\Omg_{\bt}: |\xi|\in(0,2), \arg(w)\in(0,\bt/2)\}$, we use the $Q_{\bt}$ analogue of Lemma \ref{lem:HalfGradKer} and \ref{lem:gradu_pv2} for $|\zt|=1$ so that 
    \begin{equation*}
        \left|p.v.\int_A \rd_{\bar{z}}^{m} Q_{\bt}(\zt,\xi) \tld{\omg}_r(\xi) dA(\xi)\right| \lesssim \|\tld{\omg}_r\|_{C^m(\bar{A})} \lesssim \|\omg\|_{C^k(\bar{A})},
    \end{equation*}
    where $A:= \{\xi\in\Omg_{\bt}: |\zt-\xi|\le \bt/100\}$. We obtain the last inequality uniformly in $r$ by the definition of $\tld{\omg}$ and $\tld{\omg}_r$. Indeed, for all $n\le m$, 
    \begin{equation*}
        |\nb^n \tld{\omg}_r (\xi)| = r^{-k+n} |\left(\nb^n \omg - T_{k-1-n}\nb^n \omg\right)(r\xi)| \lesssim |\xi|^{k-n} \|\omg\|_{C^k}, \qquad T_{-1} := 0.
    \end{equation*}
   
    In the region $\{\xi\in\Omg_{\bt}: |\zt-\xi|\ge \bt/100\}$, we obtain the same bound by $|\tld{\omg}_r(\xi)| \lesssim |\xi|^k \|\omg\|_{C^k}$,
    \begin{equation*}
        |\rd_{\bar{z}}^{m} Q_{\bt}(\zt,\xi)| \lesssim \left\{
        \begin{aligned}
        & |\xi|^{-p},\qquad |\xi|<1/2,\\
        & |\xi|^{-2p},\qquad |\xi|>2,
        \end{aligned}
    \right.
    \end{equation*}
    and the boundedness of $\rd_{\bar{z}}^{m+1} Q_{\bt}$ in $\{\xi\in\Omg_{\bt}: 1/2<|\xi|<2\}$. In conclusion,
    \begin{equation*}
        |\rd_{\bar{z}}^m R_k(z)| \lesssim r^{k+1-m}\left|p.v.\int_{\Omg_{\bt}}\rd_{\bar{z}}^{m} Q_{\bt}(\zt,\xi) \tld{\omg}_r(\xi) dA(\xi)\right| \lesssim r^{k+1-m}\|\omg\|_{C^k},
    \end{equation*}
    where $r=|z|$. This completes the proof.
\end{proof}

We now prove Theorem \ref{thm:ill-posed} when $\pi/\bt$ is not an integer.

\begin{proof}[Proof of Theorem \ref{thm:ill-posed} in the case $\pi/\bt\notin \mathbb{N}$]
    \textit{Step 1.} Let $f\in C^k(\overline{\Omg_{\bt}})$. Then
    \begin{equation*}
        C_k[f] := \frac{p}{\pi}\int_{\Omg_{\bt}} \frac{\sin(p\arg w)}{|w|^p}\left(f(w) - T_{k-1}f(w)\right) dA(w)
    \end{equation*} 
    is well-defined by Lemma \ref{lem:psi_singular_term2}. Let $a := e^{i\bt/2}$, and $\chi\in C_c^{\infty}(\overline{\Omg_{\bt}})$ be a real-valued cut-off function satisfying
    \begin{equation*}
        \chi(z) := \left\{
            \begin{aligned}
                &1,\qquad |z|<1,\\
                &0,\qquad |z|>2.
            \end{aligned}
        \right.
    \end{equation*}
    Let $\eta\in C_c^{\infty}(\overline{\Omg_{\bt}})$ be a nonnegative cut-off function such that $\emptyset\subsetneq \supp(\eta)\subset B_1(4a)$. Our goal is to show that there exists $\tld{T}>0$ and $0<\dlt\ll 1$ such that
    \begin{equation*}
        \omg_0(z) := \chi(z) a\cdot z + A\eta(z)
    \end{equation*}
    satisfies $\omg(t,\cdot)\notin C^{k,\alp}(\overline{\Omg_{\bt}}\cap B_{\dlt})$ for $0<t<\tld{T}$, where $A>0$ is a real constant such that $C_k[\omg_0]\neq 0$. Here, $a\cdot z$ is defined by
    \begin{equation}
        a\cdot z := \frac{a\bar{z} + \bar{a} z}{2},
    \end{equation}
    which is the usual inner product in $\mathbb{R}^2$ sense. Note that $\omg_0(0)=0$, $\rd_z \omg_0(0)=\bar{a}/2$, and $\rd_{\bar{z}} \omg_0(0) = a/2$. Since $\omg(t,\cdot)\in C^{k,p-k-1}$ for all $t>0$, the formula $\rd_t \omg + u\cdot \nb \omg=0$ implies that $\rd_t \omg$, $\cdots$, $\rd_t \nb^{k-1}\omg$ are continuous. Since the integral in $C_k[\omg_t]$ is well-defined and it is uniformly bounded in finite time interval, we obtain $\lim_{\eps\rightarrow 0} C_k[\omg_{t+\eps}] = C_k[\omg_t]$ by the dominated convergence theorem. Since $C_k[\omg_0]\neq 0$, there exists $T>0$ such that $C_k[\omg(t,\cdot)]\neq 0$ and
    \begin{equation*}
        B(t) := \int_0^t C_k[\omg(s,\cdot)] ds \neq 0 
    \end{equation*}
    for all $0<t<T$.

    \medskip
    \noindent\textit{Step 2.} From now on, we restrict the range of $z\in\Omg_{\bt}$ by $|z|<1/2$. We denote the decomposition of $u$ in Lemma \ref{lem:psi_singular_term2} by
    \begin{equation*}
        u(t,z) = P_t(z) + C_k[\omg(t,\cdot)] \bar{z}^{p-1} + R_t(z).
    \end{equation*}
    The time variable is added because the vorticity depends on time. Our goal in Step 2 is to obtain a similar decomposition of the flow map $\Phi_t$.

    Let $L_t$ be the flow map associated with $P_t$, defined by
    \begin{equation}\label{eq:L_t}
        \rd_t L_t(z) = P_t(L_t(z)),\quad L_0(z)=z.
    \end{equation}
    We make a bootstrap assumption $|\Phi_t(z)-L_t(z)|\le M|z|^{p-1}$. Since the polynomial $P_t$ does not have the constant term, the function $L$ satisfies $|L_t(z)|\lesssim |z|$. Therefore, the value of $|\Phi_t(z)|$ is bounded by $|\Phi_t(z)|\le M|z|^{p-1} + c(T)|z|$. For sufficiently small $|z|$ that the range depends on $T$, we obtain $|\Phi_t(z)|\le c(T)'|z|$.

    Since $\rd_t \Phi_t = u_t\circ \Phi_t$,
    \begin{equation*}
        \begin{split}
            |\Phi_t(z) - L_t(z)| & \le \left|\int_0^t P_s(\Phi_s(z)) - P_s(L_s(z)) + C_k[\omg_s] \overline{\Phi_s(z)}^{p-1} + R_s(\Phi_s(z)) ds\right| \\& \le \int_0^t C'|\Phi_s(z) - L_s(z)| ds + |B(T)| (c(T)'|z|)^{p-1}.
        \end{split}
    \end{equation*}
    By Gr\"onwall's inequality, we obtain $|\Phi_t(z) - L_t(z)| \le |B(T)| c(T)'' |z|^{p-1}$. The bootstrap argument yields that there exists $T'<T$ such that
    \begin{equation*}
        |\Phi_t(z)-L_t(z)|\le M|z|^{p-1} 
    \end{equation*}
    if $t<T'$ and $|z|$ is sufficiently small.

    We now decompose $\Phi_t(z)-L_t(z)$ into two functions
    \begin{equation*}
        \Phi_t(z)-L_t(z) =: H_t(z) + R_t'(z),
    \end{equation*}
    where $H_t(z)$ satisfies $H_t(rz) = r^{p-1}H_t(z)$, and $\lim_{|z|\rightarrow 0} R_t'(z)/|z|^{p-1} = 0$ for all $r>0$ and $t$. Then, $H_t$ satisfies 
    \begin{equation}\label{eq:H_t}
        \rd_t H_t(z) = P_{t,1}(H_t(z)) + C_k[\omg_t] \overline{L_{t,1}(z)}^{p-1}
    \end{equation}
    from the homogeneous part of $\rd_t(\Phi_t(z) - L_t(z))$, where the terms $P_{t,1}$ and $L_{t,1}$ are the linear parts of $P_t$ and $L_t$, respectively. Then,
    \begin{equation}\label{eq:R_t}
        \begin{split}
            \rd_t R_t'(z) & = \rd_t (\Phi_t - L_t - H_t)(z) \\ & = P_t(\Phi_t(z)) - P_t(L_t(z)) + C_k[\omg_t] \overline{\Phi_t(z)}^{p-1} + R_t(\Phi_t(z)) - P_{t,1}(H_t(z)) - C_k[\omg_t] \overline{L_{t,1}(z)}^{p-1} \\ & = P_{t,1}(R_t'(z)) + P_{t,\ge 2}(\Phi_t(z)) - P_{t,\ge 2}(L_t(z)) + C_k[\omg_t]\left(\overline{\Phi_t(z)}^{p-1} - \overline{L_{t,1}(z)}^{p-1}\right) + R_t(\Phi_t(z)),
        \end{split}
    \end{equation}
    where $P_{t,\ge 2}(z) := P_t(z)-P_{t,1}(z)$.

    Let $|R_t'(z)| = |\Phi_t(z)-L_t(z) - H_t(z)|\le M'|z|^{p-\alp}$ be the bootstrap assumption. We note that it is enough to prove the case $0<\alp<1$. Since two inequalities
    \begin{equation*}
        |P_{t,1}(R_t'(z))| \lesssim |R_t'(z)|\lesssim |z|^{p-\alp},\quad |\overline{\Phi_t(z)}^{p-1} - \overline{L_{t,1}(z)}^{p-1}| \lesssim |z|^{p-1} (|z|+O(|z|^2)) \lesssim |z|^p
    \end{equation*}
    hold by Taylor's theorem and $|R_t(\Phi_t(z))|\lesssim |z|^{k+1}\lesssim |z|^{p-\alp}$, if we choose sufficiently small time interval and the range of $|z|$ near $0$, then this assumption closes the bootstrap
    \begin{equation*}
        |\Phi_t(z)-L_t(z) - H_t(z)|\le M'|z|^{p-\alp}.
    \end{equation*}
    In conclusion, we obtain our desired decomposition $\Phi_t(z) = L_t(z) + H_t(z) + R_t'(z)$.
    
    \medskip
    \noindent\textit{Step 3.} We now consider the decomposition of the derivatives of $\Phi_t$. We only prove the case $m=1$, and the other cases $m\ge 2$ can be proved similarly.

    We make a bootstrap assumption $|\rd_{\bar{z}}\left(\Phi_t(z) - L_t(z) \right)| \lesssim M''|z|^{p-2}$. Then,
    \begin{equation*}
        \begin{split}
            |\rd_{\bar{z}}\left(\Phi_t(z) - L_t(z) \right)| \le \int_0^t |\nb u_s(\Phi_s(z)) \rd_{\bar{z}}\Phi_s(z) - \nb P_s (L_s(z)) \rd_{\bar{z}}L_s(z) | ds.
        \end{split}
    \end{equation*}
    Since $|\nb u_s(\Phi_s(z)) - \nb P_s(L_s(z))| \lesssim |z|^{p-2}$ by Step 2 and Lemma \ref{lem:psi_singular_term2}, Gr\"onwall's inequality yields that $ |\rd_{\bar{z}}\left(\Phi_t(z) - L_t(z) \right)|\lesssim \tld{c}_t|z|^{p-2}$. Since the flow map is sufficiently smooth by $u\in C^{k-1}$, the coefficient $\tld{c}_t\rightarrow 0 $ as $t\rightarrow 0$. Therefore, the bootstrap argument concludes the proof of $|\rd_{\bar{z}}\left(\Phi_t(z) - L_t(z) \right)| \lesssim M''|z|^{p-2}$.

    We now differentiate \eqref{eq:R_t} to obtain a bound $|\rd_{\bar{z}} R_t'(z)|\lesssim |z|^{p-\alp-1}$. Using Taylor's expansion, one can similarly prove that both
    \begin{equation*}
        \rd_{\bar{z}}(P_{t,\ge 2}(\Phi_t(z)) - P_{t,\ge 2}(L_t(z))), \quad \rd_{\bar{z}}\left(C_k[\omg_t](\overline{\Phi_t(z)}^{p-1} - \overline{L_{t,1}(z)}^{p-1})\right)
    \end{equation*}
    are bounded by $|z|^{p-1}$. The remaining terms satisfy
    \begin{equation*}
        |\rd_{\bar{z}} P_{t,1}(R_t'(z))| \lesssim |\rd_{\bar{z}} R_t'(z)|\lesssim |z|^{p-\alp-1},\quad |\rd_{\bar{z}}(R_t(\Phi_t(z)))|\lesssim |z|^k \lesssim |z|^{p-\alp-1}
    \end{equation*}
    by Lemma \ref{lem:psi_singular_term2}. Then, the time continuity and bootstrap argument imply $|\rd_{\bar{z}} R_t'(z)|\lesssim |z|^{p-\alp-1}$.
    
    We briefly review the properties of the decomposition $\Phi_t = L_t + H_t + R_t'(z)$. Suppose that $t>0$ and $|z|$ are sufficiently small.
    \begin{itemize}
        \item The flow map $L_t$ is associated with the polynomial $P_t$. In particular, the flow $L_t$ is smooth.
        \item The difference $|\Phi_t(z) - L_t(z)|$ satisfies $|\Phi_t(z) - L_t(z)| \lesssim |z|^{p-1}$ and $|\nb \left(\Phi_t - L_t\right)(z)| \lesssim |z|^{p-2}$.
        \item The term $ H_t (z)$ is homogeneous of degree $p-1$. Thus, the derivative $\nb H_t$ is homogeneous of degree $p-2$.
        \item The remainder $R_t'$ satisfies $|R_t'(z) |\lesssim |z|^{p-\alp}$ and $|\nb R_t'(z) |\lesssim |z|^{p-\alp-1}$.
    \end{itemize}
    
    \medskip
    \noindent\textit{Step 4.} The vorticity $\omg_t$ satisfies $\omg_t(z) = \omg_0(\Phi_t^{-1}(z)) = a\cdot \Phi_t^{-1}(z)$ for all $z\in\Omg_{\bt}$ with $|z|<1/2$. We now obtain an appropriate decomposition of $\Phi_t^{-1}$ from the formula $\Phi_t(z) = L_t(z) + H_t(z) + R_t'(z)$.

    Since $\rd_t L_t(z) = P_t(L_t(z))$ and $P_t$ is a polynomial, the flow map $L_t$ has a smooth inverse. In particular, the inverse $L_t^{-1}$ is Lipschitz. Therefore,
    \begin{equation*}
        |\Phi_t^{-1}(z) - L_t^{-1}(z)| \lesssim |L_t\left(\Phi_t^{-1}(z)\right) - z| \lesssim |H_t(\Phi_t^{-1}(z))| + |R_t'(\Phi_t^{-1}(z))| \lesssim |z|^{p-1},
    \end{equation*}
    where $|z|\ll 1$. This inequality comes from
    \begin{equation*}
        L_t(\Phi_t^{-1}(z)) - L_t(L_t^{-1}(z)) = - H_t(\Phi_t^{-1}(z)) -R_t'(\Phi_t^{-1}(z)).
    \end{equation*}
    Motivated by Taylor's theorem, let $N_L$, $N_H$ be defined by
    \begin{equation*}
        \begin{split}
            L_t(\Phi_t^{-1}(z)) - L_t(L_t^{-1}(z))  = \nb L_t(L_t^{-1}(z))\cdot (\Phi_t^{-1}(z) - L_t^{-1}(z)) + N_L,\qquad H_t(\Phi_t^{-1}(z))  = H_t(L_t^{-1}(z)) + N_H,
        \end{split}
    \end{equation*}
    where we denote $\nb f(z)\cdot h := \rd_z f(z) h + \rd_{\bar{z}} f(z) \bar{h} $ for $z$, $h\in\mathbb{C}$. If we define $\tld{R}_t'$ by 
    \begin{equation*}
        \tld{R}_t'(z) := \Phi_t^{-1}(z) - L_t^{-1}(z) + \nb L_t^{-1}(z)\cdot H_t(L_t^{-1}(z)),
    \end{equation*}
    then we obtain
    \begin{equation*}
        \nb L_t(L_t^{-1}(z))\cdot \tld{R}_t'(z) + N_L = - N_H - R_t'(\Phi_t^{-1}(z)).
    \end{equation*}
    Therefore, $\tld{R}_t'(z) = -\nb L_t^{-1}(z)\cdot \left(N_L + N_H + R_t'(\Phi_t^{-1}(z))\right)$.

    Using Taylor's theorem and the results of Step 3, we obtain
    \begin{equation*}
        |N_L| \lesssim |\Phi_t^{-1}(z) - L_t^{-1}(z)|^2 \lesssim |z|^{2(p-1)} \lesssim |z|^p,\qquad |N_H| \lesssim |z|^{p-1} (|z|+O(|z|^2)) \lesssim |z|^p.
    \end{equation*}
    The estimates of the derivatives $|\nb N_L|$, $|\nb N_H| \lesssim |z|^{p-1}$ also hold. Since $|R_t(\Phi_t^{-1}(z))|\lesssim |z|^{p-\alp}$, we conclude that $|\tld{R}_t'(z)| \lesssim |z|^{p-\alp}$ and $|\rd_{\bar{z}}\tld{R}_t'(z)| \lesssim |z|^{p-\alp-1}$ in Step 3.

    However, the term $\nb L_t^{-1}(z)\cdot H_t(L_t^{-1}(z))$ is not homogeneous of degree $p-1$. Unlike the case $\pi/\bt\notin \mathbb{N}$, we need to decompose it into homogeneous part and the remainder. Let $M_t z$ be the linear homogeneous part of $L_t^{-1}(z)$, and $\tld{H}_t$ be defined by $\tld{H}_t(z):= -M_t H_t(M_t z)$. Then, $\tld{H}_t(rz) = r^{p-1}\tld{H}_t(z)$ and
    \begin{equation*}
        \begin{split}
            |\nb L_t^{-1}(z)\cdot H_t(L_t^{-1}(z)) + \tld{H}_t(z)| & \le |\nb L_t^{-1}(z) - M_t| |H_t(L_t^{-1}(z))| + |M_t| |H_t(L_t^{-1}(z)) - H_t(M_t z)| \\ & \lesssim |z| |z|^{p-1} + 1 (|z|^{p-2}|z|^2) \lesssim |z|^p \lesssim |z|^{p-\alp}.
        \end{split}
    \end{equation*} 
    The derivative estimate follows similarly. 
    
    Let $\tld{R}_t''(z) := \Phi_t^{-1}(z) - L_t^{-1}(z) + M_t H_t(M_t z)$. Then, our desired decomposition is
    \begin{equation*}
        \Phi_t^{-1}(z) = L_t^{-1}(z) - M_t H_t (M_t z) + \tld{R}_t''(z).
    \end{equation*}
    We summarize the properties that are shown in this proof. Suppose that $t>0$ and $|z|$ are sufficiently small.
    \begin{itemize}
        \item The flow map $L_t$ is associated with the polynomial $P_t$. In particular, the flow $L_t$ has its inverse, and both $L_t$ and $L_t^{-1}$ are smooth in a neighborhood of $0$.
        \item The difference $|\Phi_t^{-1}(z) - L_t^{-1}(z)|$ satisfies $|\Phi_t^{-1}(z) - L_t^{-1}(z)| \lesssim |z|^{p-1}$ and $|\nb \left(\Phi_t^{-1} - L_t^{-1}\right)(z)| \lesssim |z|^{p-2}$.
        \item The term $ M_t H_t (M_t z)$ is homogeneous of degree $p-1$, and its derivative is homogeneous of degree $p-2$.
        \item The remainder $\tld{R}_t''$ satisfies $\tld{R}_t''(z) = \tld{R}_t'(z) - (\nb L_t^{-1}(z) H_t(L_t^{-1}(z)) + \tld{H}_t(z))$. Two bounds $|\tld{R}_t''(z) |\lesssim |z|^{p-\alp}$ and $|\nb \tld{R}_t''(z) |\lesssim |z|^{p-\alp-1}$ hold.
    \end{itemize}
    Higher derivatives can be estimated similarly, and we skip their proofs.

    \medskip
    \noindent\textit{Step 5.} We now prove that $M_t H_t (M_t z) \neq 0$ for some $z\in\Omg_{\bt}$ and complete the proof. To the contrary, assume that $M_t H_t (M_t z) = 0$ for all $z\in\Omg_{\bt}$. Since $M_t z$ is the linear part of $L_t^{-1}(z)$ and $L_t$, $L_t^{-1}$ are smooth, the operator $M_t$ is also invertible. Therefore, we obtain $H_t \equiv 0$. The function $H_t$ satisfies
    \begin{equation*}
        \rd_t H_t(z) = P_{t,1}(H_t(z)) + C_k[\omg_t] \overline{L_{t,1}(z)}^{p-1} = \rd_z \omg_t(0) H_t(z) + \rd_{\bar{z}} \omg_t(0) \overline{H_t(z)} + C_k[\omg_t] \overline{L_{t,1}(z)}^{p-1},
    \end{equation*}
    and $C_k[\omg_t]>0$ for all $t<T$. Since $\nb \omg_t(0)$ is bounded near $t=0$, the flow $L_t$ satisfies $C_k[\omg_t] \overline{L_{t,1}(z)}^{p-1} \rightarrow C_k[\omg_0] \bar{z}^{p-1}$ as $t\rightarrow 0$ uniformly. Therefore, there exists $\dlt\in (0,T)$ such that $H_t(z)\neq 0$ for $0<t<\dlt$.

    We now fix $t<\dlt$. Since $a\cdot \bar{z}^{p-1} \neq 0$ for almost every $z\in\Omg_{\bt}$, there exists $e\in\Omg_{\bt}$ such that $|e|=1$ and $a\cdot M_t H_t (M_t e) \neq 0$. Moreover, $z\mapsto a\cdot M_t H_t (M_t z)$ is a homogeneous function of degree $p-1$, and this implies that
    \begin{equation*}
        |\rd_e^k (a\cdot M_t H_t (M_t (\cdot)))(e)| \neq 0.
    \end{equation*}
    Therefore, for all $\eps>0$,
    \begin{equation*}
        \rd_e^k (a\cdot M_t H_t (M_t(\cdot)))(2\eps e) - \rd_e^k (a\cdot M_t H_t (M_t(\cdot)))(\eps e) = (2^{p-1-k} - 1)\eps^{p-1-k} \rd_e^k (a\cdot M_t H_t(M_t (\cdot)))(e) \neq 0. 
    \end{equation*}
    Since $\alp > p-1-k$,
    \begin{equation*}
        \lim_{\eps\rightarrow 0}\frac{|\rd_e^k (a\cdot M_t H_t (M_t(\cdot)))(2\eps e) - \rd_e^k(a\cdot M_t H_t (M_t(\cdot)))(\eps e)|}{\eps^{\alp}} =\infty. 
    \end{equation*}
    We now check for the other terms of the decomposition $\Phi_t^{-1}(z) = L_t^{-1}(z) - M_t H_t(M_t z) + \tld{R}_t''(z)$. The function $L_t^{-1}$ is in $C^{\infty}$ and the remainder $\tld{R}_t''$ satisfies
    \begin{equation*}
        \frac{|\rd_e^k \tld{R}_t''(2\eps e) - \rd_e^k \tld{R}_t''(\eps e)|}{\eps^{\alp}} \lesssim \frac{\eps}{\eps^{\alp}}\rightarrow 0
    \end{equation*}
    as $\eps\rightarrow 0$. In conclusion,
    \begin{equation*}
        \lim_{\eps\rightarrow 0}\frac{|\rd_e^k \omg_t(2\eps e) - \rd_e^k \omg_t(\eps e)|}{\eps^{\alp}} =\infty, 
    \end{equation*}
    and this completes the proof.
\end{proof}

We now prove Theorem \ref{thm:ill-posed} when $\pi/\bt$ is an integer. To this end, we need an analogue of Lemma \ref{lem:psi_singular_term2}.

\begin{lemma}\label{lem:psi_singular_term3}
    Suppose that $p = \pi/\bt \ge 3$ is an integer, and $\omg\in C^{p-2,\alp}(\overline{\Omg_{\bt}})$ for some $0<\alp\le 1$. Let $|z|<1/4$. Then, the velocity $u$ can be decomposed into
    \begin{equation*}
        u(z) = P_{p-2}(z) + \tld{C}_{p-2} \bar{z}^{p-1}\log\frac{1}{2|z|} + R_{p-2}(z),
    \end{equation*}
    where the polynomial $P_{p-2}$ of degree $p-2$ and the constant $\tld{C}_{p-2}$ are defined by
    \begin{equation*}
        P_{p-2}(z) = \int_{\Omg_{\bt}} K_{\bt}(z,w) T_{p-3}\omg(w) dA(w),\quad \tld{C}_{p-2} = \frac{p}{\pi} \int_0^{\bt} \sin(p\phi)\left(T_{p-2}\omg(e^{i\phi}) - T_{p-3}\omg(e^{i\phi})\right) d\phi.
    \end{equation*}
    The remainder $R_{p-2}$ satisfies $|\rd_{\bar{z}}^m R_{p-2} (z)|\lesssim |z|^{p-1-m}\|\omg\|_{C^{p-2,\alp}}$ for all $0\le m\le p-2$.
\end{lemma}

\begin{proof}
    \noindent\textit{Step 1.} We first obtain the term $\tld{C}_{p-2} \bar{z}^{p-1}\log\frac{1}{2|z|}$. Let $w =: \rho e^{i\phi}$ with $0\le\phi \le\bt$. As in the proof of Lemmas \ref{lem:borderline} and \ref{lem:psi_singular_term2}, we use the decomposition
    \begin{equation}\label{eq:K_beta_decompose}
        K_{\bt}(z,w) = \frac{p}{\pi} \bar{z}^{p-1} \frac{\sin(p\phi)}{\rho^p} + Q_{\bt}(z,w),\quad |Q_{\bt}(z,w)|\lesssim \frac{|z|^{2p-1}}{\rho^{2p}}.
    \end{equation}
    Then, the principal contribution of $K_{\bt}$ in $\{w\in\Omg_{\bt}: 2|z|<|w|<1\}$ is
    \begin{equation*}
        \begin{split}
            & \int_{\{w\in\Omg_{\bt}: 2|z|<|w|<1\}} \left(K_{\bt}(z,w) - Q_{\bt}(z,w)\right) \{T_{p-2}\omg(w) - T_{p-3}\omg(w)\}dA(w) \\ &\ = \frac{p}{\pi}\bar{z}^{p-1} \int_0^{\bt} \sin(p\phi)\{T_{p-2}\omg(e^{i\phi}) - T_{p-3}\omg(e^{i\phi}) \} d\phi \int_{2|z|}^1 \frac{d\rho}{\rho} = \tld{C}_{p-2}\bar{z}^{p-1} \log\frac{1}{2|z|}.
        \end{split}
    \end{equation*}
    The constant $\tld{C}_{p-2}$ depends on $p$ and $\nb^{p-2}\omg(0)$.

    \medskip
    \noindent\textit{Step 2.} By the definition of $R_{p-2}$, it satisfies
    \begin{equation*}
        \begin{split}
            R_{p-2}(z) & = \int_{\Omg_{\bt}\setminus\{2|z|<|w|<1\}} K_{\bt}(z,w)(\omg(w) - T_{p-3}\omg(w)) dA(w) \\ &\quad + \int_{\Omg_{\bt}\cap\{2|z|<|w|<1\}} \left(K_{\bt}(z,w) (\omg(w) - T_{p-2}\omg(w)) + Q_{\bt}(z,w) (T_{p-2}\omg(w) - T_{p-3}\omg(w))\right) dA(w).
        \end{split}
    \end{equation*}
    We follow the proof of Step 2 of Lemma \ref{lem:psi_singular_term2}: For example, if we define $r:=|z|$, $\tld{\omg}(w) := \omg(w) - T_{p-3}\omg(w)$, and $\tld{\omg}_r (w) := r^{-(p-2)}\tld{\omg}(rw)$, then we obtain
    \begin{equation*}
        \left|p.v.\int_{\Omg_{\bt}\cap\{|w|<2|z|\}} \rd_{\bar{z}}^{m+1} G_{\bt}(z,w)(\omg(w) - T_{p-3}\omg(w)) dA(w)\right| \lesssim |z|^{p-1-m} \|\tld{\omg}_r\|_{C^{p-2}} \lesssim |z|^{p-1-m} \|\omg\|_{C^{p-2}}
    \end{equation*}
    for each $m\le p-2$. The other terms do not need Cauchy principal value. In the region $\{w\in\Omg_{\bt}: |w|>1\}$ and $\{w\in\Omg_{\bt}: 2|z|<|w|<1\}$,
    \begin{equation*}
        \begin{split}
            & \left| \int_{\Omg_{\bt}\cap\{|w|>1\}} \rd_{\bar{z}}^{m+1} G_{\bt}(z,w)(\omg(w) - T_{p-3}\omg(w)) dA(w) \right| \\ &\  \lesssim \int_1^{\infty} \frac{|z|^{p-(m+1)}}{\rho^p} \rho^{p-3} \|\omg\|_{C^{p-3}} \rho d\rho \lesssim |z|^{p-1-m} \|\omg\|_{C^{p-3}},\qquad \rho := |w|,
        \end{split}
    \end{equation*}
    and
    \begin{equation*}
        \begin{split}
            & \left| \int_{\Omg_{\bt}\cap\{2|z|<|w|<1\}} \rd_{\bar{z}}^{m+1} G_{\bt}(z,w)(\omg(w) - T_{p-2}\omg(w)) dA(w) \right| \\ &\ \lesssim \int_{2|z|}^1 \frac{|z|^{p-(m+1)}}{\rho^p} \rho^{p-2+\alp} \|\omg\|_{C^{p-2,\alp}} \rho d\rho \lesssim |z|^{p-1-m} \|\omg\|_{C^{p-2,\alp}}.
        \end{split}
    \end{equation*}
    The term including $Q_{\bt}$ satisfies
    \begin{equation*}
        \begin{split}
            & \left| \int_{\Omg_{\bt}\cap\{2|z|<|w|<1\}}  \rd_{\bar{z}}^{m} Q_{\bt}(z,w) (T_{p-2}\omg(w) - T_{p-3}\omg(w))dA(w) \right| \\ &\ \lesssim \int_{2|z|}^1 \frac{|z|^{2p-(m+1)}}{\rho^{2p}} \rho^{p-2} \|\omg\|_{C^{p-2}} \rho d\rho \lesssim |z|^{p-1-m} \|\omg\|_{C^{p-2}}.
        \end{split}
    \end{equation*}
    Therefore, we conclude that $|\rd_{\bar{z}}^m R_{p-2} (z)|\lesssim |z|^{p-1-m}\|\omg\|_{C^{p-2,\alp}}$.
\end{proof}

We note the proof idea of Theorem \ref{thm:ill-posed} when $\pi/\bt$ is an integer. The proof is similar to that for the case $\pi/\bt\notin \mathbb{N}$.

\begin{proof}[Proof idea of Theorem \ref{thm:ill-posed} in the case $\pi/\bt\in \mathbb{N}$]
    We note where we need to modify the proof of the case $\pi/\bt\notin \mathbb{N}$. 
    
    \medskip
    \noindent\textit{Step 1.}
    For $f\in C^{k,\alp}$, let $C_p[f]$ be defined by
    \begin{equation*}
        C_p[f] := \frac{p}{\pi} \int_0^{\bt} \sin(p\phi)\left(T_{p-2} f(e^{i\phi}) - T_{p-3} f(e^{i\phi})\right) d\phi,
    \end{equation*}
    which comes from Lemma \ref{lem:psi_singular_term3}. Let $a:= e^{i\bt/2}$, $\chi$ be defined as in the proof of the case $\pi/\bt\notin \mathbb{N}$. Our goal is to show that there exists $\tld{T}>0$ and $0<\dlt\ll 1$ such that
    \begin{equation*}
        \omg_0(z) := \chi(z) a\cdot (z + z^{p-2}) 
    \end{equation*}
    satisfies $\omg(t,\cdot)\notin C^{p-2,1}(\overline{\Omg_{\bt}}\cap B_{\dlt})$ for $0<t<\tld{T}$. Since $\omg_t\in C^{p-2,1/2}$ for all $t$, we obtain $\rd_t \omg$, $\cdots$, $\rd_t \nb^{p-2} \omg$ are continuous so that $C_p[\omg_t]$ is continuous in time. Since
    \begin{equation*}
        \Re\left(\frac{p\bar{a}}{\pi} \int_0^{\bt} \sin(p\phi) (e^{i\phi})^{p-2} d\phi \right)\neq 0,
    \end{equation*}
    the initial vorticity $\omg_0$ satisfies $C_p[\omg_0]\neq 0$. Therefore, the continuity of $C_p[\omg_t]$ yields that there exists $T>0$ such that $C_p[\omg_t]\neq 0$, and
    \begin{equation*}
        B(t) := \int_0^t C_p[\omg(s,\cdot)] ds \neq 0
    \end{equation*}
    for all $0<t<T$.
    
    \medskip
    \noindent\textit{Step 2.} We now follow the proof of the case $\pi/\bt \notin \mathbb{N}$. Then, a similar decomposition of $\Phi_t$
    \begin{equation*}
        \Phi_t(z) = L_t(z) + H_t(z) + R_t'(z)
    \end{equation*}
    follows from the bootstrap argument, where $|H_t(z)| = C_t |z|^{p-1}\log(1/|z|)$ for some constant $C_t\ge 0$, and $|R_t'(z)|\lesssim |z|^{p-1}$. More precisely, the terms $H_t$ and $R_t'$ satisfy
    \begin{equation*}
        \begin{split}
            \rd_t H_t(z) & = P_{t,1}(H_t(z)) + C_p[\omg_t] \overline{L_{t,1}(z)}^{p-1}\log \frac{1}{|z|}, \\ \rd_t R_t'(z) & = P_{t,1} (R_t'(z)) + P_{t,\ge 2}(\Phi_t(z)) - P_{t,\ge 2}(L_t(z)) \\ & \qquad + C_p[\omg_t] \left(\overline{\Phi_t(z)}^{p-1}\log\frac{1}{|\Phi_t(z)|} - \overline{L_{t,1}(z)}^{p-1}\log\frac{1}{|z|}\right) + R_t(\Phi_t(z)),
        \end{split}
    \end{equation*}
    which are analogous to \eqref{eq:L_t} and \eqref{eq:H_t}, respectively. Using bootstrap argument, the terms $H_t$ and $R_t'$ satisfy 
    \begin{equation*}
        |H_t(z)| = C_t |z|^{p-1}\log\frac{1}{|z|},\quad |R_t'(z)| \lesssim |z|^{p-1},\quad |\nb R_t'(z)| \lesssim |z|^{p-2}.
    \end{equation*}
    
    \medskip
    \noindent\textit{Step 3.}
    Unlike the case $\pi/\bt \notin \mathbb{N}$, although $|H_t(z)| = C_t |z|^{p-1}\log\frac{1}{|z|}$, the derivative $\nb H_t(z)$ does not satisfy $|\nb H_t(z)| = C_t |z|^{p-2}\log\frac{1}{|z|}$. Instead, the term $\nb H_t(z)$ is decomposed into $\nb H_t(z) =: H_t'(z) + H_t''(z)$, where $|H_t'(z)| = C_t' |z|^{p-2}\log\frac{1}{|z|}$ and $|H_t''(z)| = C_t'' |z|^{p-2}$.
    We note that the most singular term in $\nb \Phi_t^{-1}$ is $H_t'$. Therefore, our desired decomposition of $\nb \Phi_t$ is 
    \begin{equation*}
        \nb \Phi_t(z) = \nb L_t(z) + H_t'(z) + \left( H_t''(z) + \nb R_t'(z)\right),
    \end{equation*}
    where $H_t''+R_t'$ is the smooth remainder such that $|H_t''(z) + \nb R_t'(z)| \lesssim |z|^{p-2}$.

    \medskip
    \noindent\textit{Step 4.} The decomposition of $\Phi_t^{-1}$ is also similarly obtained by
    \begin{equation*}
        \Phi_t^{-1}(z) = L_t^{-1}(z) - M_t H_t' (M_t z) + \tld{R}_t''(z),
    \end{equation*}
    where $|M_t H_t' (M_t z)|$ is proportional to $|z|^{p-1}\log(1/|z|)$, and $|\tld{R}_t''(z)|\lesssim |z|^{p-1}$. As in Step 3, the derivative $\nb (M_t H_t' (M_t (\cdot)))(z)$ can be decomposed into two terms; one has size $|z|^{p-2}\log(1/|z|)$ and the other is bounded by $|z|^{p-2}$.
    
    \medskip
    \noindent\textit{Step 5.}
    Fix $t\ll 1$. One can follow Step 5 of the proof of the case $\pi/\bt \notin \mathbb{N}$. Since $C_p[\omg_0]\neq 0$, there exists $e\in\Omg_{\bt}$ such that $|e|<1/2$ and $\tld{H}_t^{(p-2)}(e)\neq 0$, where $\tld{H}_t^{(p-2)} (z)$ is the most singular term with scale $|\tld{H}_t^{(p-2)} (z)|=\tld{C}_t |z|\log(1/|z|)$ in the decomposition of $\rd_e^{p-2} (a\cdot M_t H_t'(M_t (\cdot)))(z)$. Then,
    \begin{equation}\label{eq:p-2_singular}
        \frac{|\tld{H}_t^{(p-2)}(2\eps e) - \tld{H}_t^{(p-2)}(\eps e)| }{\eps} \gtrsim \frac{\eps \log (1/\eps)}{\eps}
    \end{equation}
    for sufficiently small $\eps>0$. The other terms in the decomposition of $\rd_e^{p-2} \Phi_t^{-1}(z)$ are Lipschitz in a neighborhood of the origin. Since $p-2\in\mathbb{N}$ and  $\omg_t(z) = a\cdot (\Phi_t^{-1}(z) + (\Phi_t^{-1}(z))^{p-2})$ for $|z|<1/2$, the inequality \eqref{eq:p-2_singular} is the only contribution to make
    \begin{equation*}
        \lim_{\eps\rightarrow 0}\frac{|\rd_e^{p-2} \omg_t(2\eps e) - \rd_e^{p-2}  \omg_t(\eps e)| }{\eps} = \infty.
    \end{equation*}
    Therefore, $\omg_t\notin C^{p-2,1}$ for all sufficiently small positive times.
\end{proof}

\section{Long-time behavior of scale-invariant data}\label{Long-time}

If the initial data $\omg_0$ defined in $\Omg_{\bt}$ is scale-invariant, then the scale-invariance property is preserved for all time. In other words, $u$ satisfies $u(rz) = ru(z)$ for all $r>0$ and $z\in\Omg_{\bt}$. Since both $u$ and $\omg$ are determined by their values on $\Omg_{\bt}\cap \{|z|=1\}$, an analogous 1D system for the Euler equations can be deduced. For $0<\tht<\bt$, let $g(t,\tht):=\omg(t,e^{i\tht})$, and $G(t,\theta)$ be defined by
\begin{align*}
    G(t,\tht) := \frac{1}{2}u(t,e^{i\tht})\cdot ie^{i\tht}.
\end{align*}
They are introduced in $\mathbb{R}^2$ by \cite{EJS}, and it is well-known that $g$ and $G$ satisfy
\begin{align*}
    \partial_t g + 2G\partial_{\tht} g = 0, \quad g=4G+\partial_{\tht\tht}G.
\end{align*}

We first discuss some properties induced by $g=4G+\partial_{\tht\tht}G$. Since the domain is $\Omg_{\bt}$ and the particles can move only radially along the boundaries, $G$ has zero boundary conditions. Since $g=4G+\partial_{\tht\tht}G$ with $G(t,0)=G(t,\bt)=0$, one can deduce the formula of the kernel
\begin{align*}
    K(\tht,\tht') = \begin{cases}
        -\frac{1}{2\sin(2\bt)} \sin(2\tht)\sin(2\bt-2\tht'),\qquad \tht\le\tht', \\ -\frac{1}{2\sin(2\bt)} \sin(2\bt-2\tht)\sin(2\tht'),\qquad \tht\ge\tht'.
    \end{cases}
\end{align*}
This $K$ is defined by the fundamental solution $4K(\tht,\tht')+\rd_{\tht\tht}K(\tht,\tht')=\dlt(\tht-\tht')$ with boundary conditions $K(0,\tht')=K(\bt,\tht')=0$.

From now on, we focus on the case $g\ge 0$. Since $K(\tht,\tht')<0$ for all $\tht$, $\tht'\in (0,\bt)$, if $g \ge 0$, then it is clear that $G\le 0$, and $\partial_{\tht\tht}G = g-4G\ge 0$. We next discuss properties induced by $\partial_t g + 2G\partial_{\tht} g = 0$. Since this equation is similar to a transport equation, it is natural to introduce the flow map generated by $2G$. Let $\chi$ be the flow map defined by
\begin{equation}\label{eq:def_chi}
    \frac{d}{dt}\chi(t,\tht) = 2G(t,\chi(t,\tht)), \quad \chi(0,\tht)=\tht,\qquad t>0,\ \tht\in (0,\bt).
\end{equation}
Then, $g$ satisfies $g(t,\chi(t,\tht)) = g(0,\tht)$ for all $t>0$ and $0<\tht<\bt$. We begin the proof of Theorem \ref{thm:long} with a long-time behavior of $\chi$.

\begin{lemma}\label{lem:long1}
    Suppose that $g_0$ is nonnegative in $(0,\bt)$. Let $\chi(t,\tht)$ be the flow map by \eqref{eq:def_chi}. If $\chi$ satisfies $\lim_{t\rightarrow{}\infty}\chi(t,\tht_0)>0$ for some $\tht_0\in (0,\bt)$, then $\lim_{t\rightarrow{\infty}} G(t,\cdot)=0$ uniformly.
\end{lemma}
\begin{proof}
    Since $G(t,0)=G(t,\bt)=0$ and $\partial_{\tht\tht}G>0$, we obtain
    \begin{equation}\label{eq:convex}
        \frac{G(t,\chi(t,\tht_0))-0}{\chi(t,\tht_0)-0} \le \partial_{\tht}G(t,\chi(t,\tht_0)) \le \frac{0-G(t,\chi(t,\tht_0))}{\bt-\chi(t,\tht_0)}.
    \end{equation}
    We note that $\chi(\cdot,\tht)$ is a monotonically decreasing function for each $\tht\in(0,\bt)$. Since the length of $[0,\bt]$ is finite, the monotonicity implies $\lim_{t\rightarrow{}\infty} G(t,\chi(t,\tht_0))=0$. Therefore, \eqref{eq:convex} yields $\lim_{t\rightarrow{}\infty}\partial_{\tht}G(t,\chi(t,\tht_0))=0$, and these two limits with the convexity of $G(t,\cdot)$ complete the proof.
\end{proof}
\begin{lemma}\label{lem:long2}
    Suppose that $g_0$ is nonnegative and bounded in $(0,\bt)$. Let $\chi(t,\tht)$ be the flow map by \eqref{eq:def_chi}. If $\lim_{t\rightarrow{}\infty}\chi(t,\tht_0)>0$ for some $\tht_0\in (0,\bt)$, then $\lim_{t\rightarrow{\infty}} \|g(t,\cdot)\|_{L^1(0,\bt)}=0$.
\end{lemma}
\begin{proof}
    Lemma \ref{lem:long1} implies that $\lim_{t\rightarrow{\infty}} G(t,\cdot)=0$ uniformly. Let $\epsilon>0$. Then there exists $T>0$ such that $G(t,\cdot)>-\epsilon$ for all $t>T$. Since $G(t,\cdot)$ is convex, the zero boundary values yield $\|\rd_{\tht} G(t,\cdot)\|_{L^1} \le 2\eps$.
    
    Since $g(t,\chi(t,\tht)) = g_0(\tht)$, the function $g$ is uniformly bounded in $t$. The uniform boundedness of $g$ and the uniform convergence of $G$ to 0 imply that $\partial_{\tht\tht}G(t,\cdot) = g-4G$ is also uniformly bounded for large $t$. Using this with $\|\rd_{\tht} G(t,\cdot)\|_{L^1} \le 2\eps$, we obtain that
    \begin{equation*}
        \lim_{t\rightarrow\infty} |\rd_{\tht} G(t,\bt) - \rd_{\tht} G(t,0)| = 0.
    \end{equation*}
    
    Therefore, the $L^1$ norm of $\rd_{\tht\tht} G(t,\cdot)$ satisfies
    \begin{equation*}
        \|\rd_{\tht\tht} G(t,\cdot)\|_{L^1} = \int_0^{\bt} \rd_{\tht\tht} G(t,\tht)d\tht = \rd_{\tht} G(t,\bt) - \rd_{\tht} G(t,0) \rightarrow 0
    \end{equation*}
    as $t\rightarrow \infty$. Since $\lim_{t\rightarrow{\infty}} G(t,\cdot)=0$ uniformly, we conclude that $\lim_{t\rightarrow\infty} \|g(t,\cdot)\|_{L^1} =0$.
\end{proof}

Now we prove Theorem \ref{thm:long}, which is equivalent to Theorem \ref{thm:Long-Scale}.

\begin{theorem}\label{thm:long}
    Suppose that $g_0$ is left continuous at $\bt$, nonnegative and bounded. Let $\lim_{\tht\rightarrow{\bt^-}}g_0(\tht)=:g_0(\bt)\ge 0$. Then,
    \begin{align*}
        \lim_{t\rightarrow{\infty}} \|g(t,\cdot)-g_0(\bt)\|_{L^1([0,\bt])}=0.
    \end{align*}
\end{theorem}
\begin{proof}
    Let $\epsilon>0$. There exists $\dlt>0$ such that $\bt-\dlt <\tht<\bt$ implies $|g_0(\tht)-g_0(\bt)|<\epsilon$.
    
    We first suppose that $\lim_{t\rightarrow{\infty}}\chi(t,\tht)=0$ for all $\tht\in(0,\bt)$. Then, there exists $T=T(\epsilon)$ such that if $t>T$, then $\chi(t,\bt-\dlt)<\epsilon$. Therefore, we obtain
    \begin{align*}
        |g(t,\tht) - g_0(\bt)| < \epsilon,\qquad t>T,\ \tht\in (\eps,\bt).
    \end{align*}
    That is, $g(t,\cdot)$ converges in measure to $g_0(\bt)$. By uniform integrability, $g(t,\cdot)$ converges to $g_0(\bt)$ in $L^1$.
    
    We now suppose that $\lim_{t\rightarrow{}\infty}\chi(t,\tht_0)>0$ for some $\tht_0\in(0,\bt)$. By Lemma \ref{lem:long2}, $\lim_{t\rightarrow{\infty}} \|g(t,\cdot)\|_{L^1(0,\bt)}=0$. Therefore, it is enough to show that $g_0(\bt)=0$. To the contrary, assume $g_0(\bt)>0$. Since $G<0$, the transport property gives
    \begin{align*}
        |g(t,\tht) - g_0(\bt)| < \epsilon,\qquad t>0,\ \tht\in (\bt-\dlt,\bt).
    \end{align*}
    If we take $\epsilon\in (0,g_0(\bt))$, then $\|g(t,\cdot)\|_{L^1(0,\bt)} \ge \dlt (g_0(\bt) - \epsilon)$ for all $t>0$. This yields
    \begin{align*}
        \lim_{t\rightarrow\infty}\|g(t,\cdot)\|_{L^1(0,\bt)} \ge \dlt (g_0(\bt) - \epsilon) > 0,
    \end{align*}
    which contradicts $\lim_{t\rightarrow\infty}\|g(t,\cdot)\|_{L^1(0,\bt)} = 0$.
\end{proof}

\appendix

\section{Appendix}\label{Appendix}
This section introduces some classical calculus-based results. We first note the proof idea of some basic properties of $G_{\bt}$ stated in Section \ref{Preliminary}. We next prove the boundedness and log-Lipschitz property of $u$ under the assumption $\omg\in L^{\infty}(\Omg_{\bt})$.

\subsection{Basic properties of Green's function}

We first derive the formula of $G_{\bt}$, which holds for all $0<\bt<2\pi$.

\begin{lemma}\label{thm:G_beta}
    Suppose that $0<\bt<2\pi$. The Dirichlet Green's function $G_{\bt}$ on $\Omg_{\bt}$ satisfies
    \begin{align*}
        G_{\bt}(z,w) = \frac{1}{2\pi}\log \frac{\left|z^{\pi/\bt} - w^{\pi/\bt}\right|}{\left|z^{\pi/\bt} - \bar{w}^{\pi/\bt}\right|}
    \end{align*}
    for all $z$, $w\in\Omg_{\bt}$ ($z\neq w$). Here, the logarithm has the principal branch.
\end{lemma}

The existence of a solution of \eqref{eq:Green} and the formula of $G_{\bt}$ are classical results, and we note the proof of Lemma \ref{thm:G_beta}. The decomposition of the formula of $G_{\bt}$ is used in Section \ref{Holder-wellposed}. 

\begin{proof}
    The Green's function $G_{\bt}$ vanishes on $\partial\Omg_{\bt}$. Using separation of variables, let us denote
    \begin{equation*}
        \begin{split}
            G_{\bt}(z,w) = \sum_{m=1}^{\infty} g_m(r,r_0)\sin \left(m\pi\frac{\tht}{\beta}\right) \sin \left(m\pi\frac{\tht_0}{\beta}\right) = \sum_{m=1}^{\infty} \frac{g_{m}(r,r_0)}{2}\left( \cos\left( m\pi\frac{\tht-\tht_0}{\beta} \right) - \cos\left(m\pi\frac{\tht+\tht_0}{\beta} \right)\right),
        \end{split}
    \end{equation*}
    where $z=re^{i\tht}$, $w=r_0e^{i\tht_0}$, and $g_m$ $(m\in\mathbb{N})$ will soon be determined. Since $-2\bt < \tht-\tht_0 < 2\bt$,
    \begin{equation*}
        1+\sum_{m=1}^{\infty} 2\cos\left( m\pi\frac{\tht-\tht_0}{\beta} \right) = \sum_{m\in\mathbb{Z}} e^{\frac{im\pi}{\beta} (\tht - \tht_0)} = 2\pi \cdot\dlt_0\left( \frac{\pi}{\beta}(\tht-\tht_0)\right) = 2\beta \cdot \dlt_0(\tht-\tht_0)
    \end{equation*}
    in the distribution sense. If we find $g_m$ such that
    \begin{equation*}
        \frac{1}{r}\frac{\partial}{\partial r}\left( r\frac{\partial g_m}{\partial r} \right) - \frac{1}{r^2}\left(\frac{m^2\pi^2}{\beta^2} g_m\right) = \frac{2}{r\beta}\dlt_0(r-r_0)
    \end{equation*}
    for each $m\in\mathbb{N}$, then $\Dlt_z G_{\bt}(z,w) = \dlt_0(r-r_0)\dlt_0(\tht-\tht_0)/r = \dlt_0(z-w)$. The homogeneous solutions of the above equation are $r^{\pm m\pi/\beta}$. Since $g_m$ is continuous, $\sum_{m=1}^{\infty} g_m(0,r_0) = 0$, $\lim_{r\rightarrow\infty}\sum_{m=1}^{\infty} g_m(r,r_0) = 0$, and the jump condition
    \begin{equation*}
        \left[ r\frac{\rd g_m}{\rd r} \right]_{r_0-\epsilon}^{r_0+\epsilon} = \frac{2}{\beta}
    \end{equation*}
    should be satisfied, the function $g_m$ satisfies
    \begin{equation*}
        \begin{split}
            g_m(r,r_0) = \begin{cases} -\frac{1}{m\pi}\left(\frac{r}{r_0}\right)^\frac{m\pi}{\beta}, &\qquad 0 < r \le r_0,\\
            -\frac{1}{m\pi}\left(\frac{r_0}{r}\right)^\frac{m\pi}{\beta}, &\qquad r \ge r_0.
            \end{cases} 
        \end{split}
    \end{equation*}
    Therefore, if $r<r_0$, then
    \begin{equation*}
        \begin{split}
            G_{\bt}(z,w) & = \sum_{m=1}^{\infty} -\frac{1}{2m\pi} \left(\frac{r}{r_0}\right)^{\frac{m\pi}{\beta}} \left(\cos\left(m\pi\frac{\tht-\tht_0}{\beta}\right) - \cos\left(m\pi\frac{\tht+\tht_0}{\beta}\right)\right)\\ & = \Re\ \frac{1}{2\pi}\left(\log\left(1 - \left(\frac{r}{r_0}\right)^{\frac{\pi}{\beta}}e^{i\pi\frac{\tht-\tht_0}{\beta}}\right) - \log\left(1 - \left(\frac{r}{r_0}\right)^{\frac{\pi}{\beta}}e^{i\pi\frac{\tht+\tht_0}{\beta}}\right)\right) = \frac{1}{2\pi} \log\frac{|z^{\frac{\pi}{\beta}} - w^{\frac{\pi}{\beta}}|}{|z^{\frac{\pi}{\beta}} - \overline{w}^{\frac{\pi}{\beta}}|}.
        \end{split}
    \end{equation*}
    Since the equation $-\log(1-z) = \sum_{m=1}^{\infty} z^m/m$ holds for all $|z|\le 1$ with $z\neq 1$, the same formula is obtained for all $z$, $w\in \Omg_{\bt}$ with $z\neq w$.
\end{proof}

\begin{lemma}\label{lem:ker_bound}
    Suppose that $0<\bt<\pi$, and $z$, $w\in\Omg_{\bt}$. Let $k\in\mathbb{N}$.
    \begin{itemize}
        \item If $|w|\le 2|z|$, then
            \begin{equation}\label{eq:z-w_bound}
                \left|\rd_{\bar{z}}^k G_{\bt}(z,w)\right| \le \frac{C_1}{|z-w|^k}
            \end{equation}
            holds for some constant $C_1 =C_1(k,\bt)$. If $\pi/\bt > k+2$, then \eqref{eq:z-w_bound} holds for all $z$, $w\in\Omg_{\bt}$.
        \item If $|w|\ge 2|z|$, the inequality
            \begin{equation}\label{eq:z,w_bound}
                \left|\rd_{\bar{z}}^k G_{\bt}(z,w)\right| \le C_2\frac{|z|^{\pi/\bt - k}}{|w|^{\pi/\bt}}
            \end{equation}
            holds for some constant $C_2=C_2(k,\bt)$.
    \end{itemize}
\end{lemma}

\begin{remark}
    Since $x\mapsto \log|1-x|$ is Lipschitz on $|x|<1/2$, the inequality \eqref{eq:z,w_bound} also holds when $k=0$.
\end{remark}
Lemma \ref{lem:ker_bound} follows from the formula of $\rd_{\bar{z}}^k G_{\bt}$. We omit the proof, and just note the idea of the proof. The inequality \eqref{eq:z-w_bound} is proved by Lemma \ref{lem:exp_p}, and the second inequality \eqref{eq:z,w_bound} is a direct consequence of \eqref{eq:z-w_bound} when $k \ge \pi/\bt $. If $k<\pi/\bt$, then calculating derivatives of \eqref{eq:G_deri_deri} implies \eqref{eq:z,w_bound}.

\begin{lemma}\label{lem:exp_p}
    Let $p>1$, and $z$, $w\in\Omg_{\bt}$. Then,
    \begin{align*}
        |z^p-w^p|\gtrsim_p (\max\{|z|,|w|\})^{p-1} |z-w|.
    \end{align*}
\end{lemma}
\begin{proof}
    Without loss of generality, suppose that $|w|\le |z|$. If $|w| \le |z|/2$, then
    \begin{align*}
        \frac{|z^p-w^p|}{|z-w|} \ge \frac{|z|^p/2}{2|z|} = \frac{1}{4}|z|^{p-1}.
    \end{align*}
    Now we assume $|z|/2 \le |w| \le |z|$. The distance between the origin and the convex hull of $\{w^{p-1}: w\in\Omg_{\bt},\ |z|/2\le |w| \le |z|\}$ is $|z|^{p-1}\sin(\pi/2p)/2^{p-1}$. Therefore,
    \begin{equation*}
        \left|\frac{z^p-w^p}{z-w}\right| =\left|\int_0^1 p\left(w + t(z-w)\right)^{p-1}dt\right| \ge \frac{p}{2^{p-1}}|z|^{p-1}\sin\left(\frac{\pi}{2p}\right). 
    \end{equation*}
    In summary, $|z^p-w^p|\ge C_p |z|^{p-1} |z-w|$, where $0<C_p<1/4$ and $\lim_{p\rightarrow\infty} C_p=0$.
\end{proof}

\subsection{Velocity estimates for bounded vorticity}

Lemmas \ref{1.1} and \ref{lem:u_diff} deal with the properties of $u$ under the assumption $\omg\in L^{\infty}(\Omg_{\bt})$. Recall that in $\Omg_{\bt}$ with $0<\bt<\pi/2$, the velocity $u$ is defined by
\begin{equation*}
    u(z) :=  \int_{\Omg_{\bt}} \frac{ip}{2\pi}\frac{\bar{z}^{p-1} (\bar{w}^p - w^p)}{(\bar{z}^p-\bar{w}^p)(\bar{z}^p-w^p)} \omg(w) dA(w).
\end{equation*}
\begin{lemma} \label{1.1}
    Suppose that $0<\bt<\pi/2$, $\omg\in L^{\infty}(\Omg_{\bt})$, and $z\in\Omg_{\bt}$. Then $u$ satisfies $\abs{u(z)} \lesssim_{\bt} \nrm{\omg}_{L^{\infty}} \abs{z}$.
\end{lemma}

\begin{proof}
    Let $z\in\Omg_{\bt}$, and
    \begin{equation*} A := \{w\in\Omg_{\bt}: |w|< 2|z|\}, \quad B := \{w\in\Omg_{\bt}: |w|> 2|z|\}.
    \end{equation*}
    Then, since $p>2$,
    \begin{align*}
        \left|\int_{B}\frac{\bar{z}^{p-1} (\bar{w}^p - w^p)}{(\bar{z}^p-\bar{w}^p)(\bar{z}^p-w^p)} \omg(w) dA(w)\right|  & \le C\|\omg\|_{L^{\infty}}|z|^{p-1}  \int_{2|z|}^{\infty} \frac{r^p}{r^{2p}} rdr \le \frac{C}{p-2}\|\omg\|_{L^{\infty}}|z|.
    \end{align*}
    In the region $A$,
    \begin{align*}
        \left|\int_{A}\frac{\bar{z}^{p-1} (\bar{w}^p - w^p)}{(\bar{z}^p-\bar{w}^p)(\bar{z}^p-w^p)} \omg(w) dA(w)\right| & \le \|\omg\|_{L^{\infty}}|z|^{p-1} \int_{A} \left( \frac{1}{|\bar{z}^p-\bar{w}^p|} + \frac{1}{|\bar{z}^p-w^p|} \right) dA(w) \\ & \le 2 \|\omg\|_{L^{\infty}}|z|^{p-1} \int_{A} \frac{1}{|\bar{z}^p-\bar{w}^p|}dA(w) \lesssim_p \|\omg\|_{L^{\infty}}|z|.
    \end{align*}
    The last inequality comes from Lemma \ref{lem:exp_p}. This completes the proof.
\end{proof}

\begin{lemma} \label{lem:u_diff}
    Suppose that $0<\bt<\pi/2$, $\omg\in L^{\infty}(\Omg_{\bt})$ and $z_1$, $z_2\in\Omg_{\bt}$ $(z_1\neq z_2)$. Then $u$ satisfies
    \begin{equation*}
        \abs{u(z_1)-u(z_2)} \lesssim_{\bt} \nrm{\omg}_{L^{\infty}} \abs{z_1-z_2}\log\left(\frac{6\max\{\abs{z_1},\abs{z_2}\}}{\abs{z_1-z_2}}\right).
    \end{equation*}
\end{lemma}

\begin{proof}
    Without loss of generality, let $|z_1|\ge |z_2|$. Divide $\Omg_{\bt}$ into three subsets: $A  := \{ w\in\Omg_{\bt} : \abs{z_1-w}< 2\abs{z_1-z_2} \}$, $C  := \{ w\in\Omg_{\bt} : \abs{w}> 4\abs{z_1}\}$, $B  := \Omg_{\bt}\setminus(A\cup C)$.
    
    In the region $C$, $\abs{w} > 4\abs{z_1}$ and $\abs{w} > 4\abs{z_2}$ for all $w\in C$. Lemma \ref{lem:ker_bound} and the complex mean value inequality yield
    \begin{align*}
    \left|\int_{C} \{K_{\bt}(z_1,w) - K_{\bt}(z_2,w)\} \omg(w) dA(w)\right| & \lesssim \|\omg\|_{L^{\infty}} |z_1-z_2|\int_{C} \frac{\abs{z_1}^{p-2}}{\abs{w}^p} dA(w) \\ & \le \|\omg\|_{L^{\infty}}\abs{z_1-z_2}|z_1|^{p-2}\int_{4|z_1|}^{\infty}\frac{1}{r^p} r dr \lesssim \nrm{\omg}_{L^{\infty}}\abs{z_1-z_2}.
    \end{align*}
    In the region $A$,
    \begin{align*}
        \left|\int_A \{K_{\bt}(z_1,w) - K_{\bt}(z_2,w)\} \omg(w) dA(w)\right| & \lesssim \|\omg\|_{L^{\infty}} \int_A \left(\frac{|z_1|^{p-1}}{|z_1^p-w^p|} + \frac{|z_2|^{p-1}}{|z_2^p-w^p|}\right)dA(w) \\ & \lesssim \|\omg\|_{L^{\infty}} \int_{\{w\in\Omg_{\bt}: |z_1-w|<3|z_1-z_2|\}} \frac{dA(w)}{|z_1-w|} \lesssim \|\omg\|_{L^{\infty}}|z_1-z_2|.
    \end{align*}
    The second inequality follows from Lemma \ref{lem:exp_p}. For the last inequality, note that $A \subset (\{w:|z_1-z_2|/2 < |z_1-w|< 2|z_1-z_2|\}\cup \{w: |z_1-z_2|/2 < |z_2-w| < 3|z_1-z_2|\})$. Finally, in the region $B$, we obtain
    \begin{align*}
        \quad \left|\int_B \{K_{\bt}(z_1,w) - K_{\bt}(z_2,w)\} \omg(w) dA(w)\right| & \lesssim \|\omg\|_{L^{\infty}}  \int_B \frac{|z_1-z_2|}{|z_1-w||z_2-w|} dA(w) \\ & \lesssim \|\omg\|_{L^{\infty}}  \int_B \frac{|z_1-z_2|}{|z_1-w|^2}dA(w) \lesssim \|\omg\|_{L^{\infty}} |z_1-z_2| \log\frac{2|z_1|}{|z_1-z_2|}.
    \end{align*}
\end{proof}

\bibliographystyle{amsplain}
\bibliography{well-posedness_threshold}

@article{AN2,
  author    = {Agrawal, S. and Nahmod, A. R.},
  title     = {Uniqueness of the {2D} {E}uler equation on a corner domain with non-constant vorticity around the corner},
  journal   = {Nonlinearity},
  volume    = {35},
  number    = {6},
  pages     = {2767--2808},
  year      = {2022}
}

@article{AN1,
  author    = {Agrawal, S. and Nahmod, A. R.},
  title     = {Uniqueness of the {2D} {E}uler equation on rough domains},
  journal   = {J. Eur. Math. Soc.},
  year      = {2025},
  note      = {Published online first}
}

@article{BPT,
  author    = {Bardos, C. and Di Plinio, F. and Temam, R.},
  title     = {The {E}uler equations in planar nonsmooth convex domains},
  journal   = {J. Mathematical Analysis and Applications},
  volume    = {407},
  number    = {1},
  pages     = {69--89},
  year      = {2014}
}

@article{BKM,
  author    = {Beale, J. T. and Kato, T. and Majda, A.},
  title     = {Remarks on the breakdown of smooth solutions for the {$3$}-{D} {E}uler equations},
  journal   = {Comm. Math. Phys.},
  volume    = {94},
  number    = {1},
  pages     = {61--66},
  year      = {1984}
}

@article{BMP,
  author    = {Benedetto, D. and Marchioro, C. and Pulvirenti, M.},
  title     = {On the {E}uler flow in {${\bf R}^2$}},
  journal   = {Arch. Rational Mech. Anal.},
  volume    = {123},
  number    = {4},
  pages     = {377--386},
  year      = {1993}
}

@book{BK1,
  author    = {Borsuk, M. and Kondratiev, V.},
  title     = {Elliptic boundary value problems of second order in piecewise smooth domains},
  series    = {Vol. 69},
  publisher = {Elsevier},
  year      = {2006}
}

@article{BL2,
  author    = {Bourgain, J. and Li, D.},
  title     = {Strong illposedness of the incompressible {E}uler equation in integer {$C^m$} spaces},
  journal   = {Geom. Funct. Anal.},
  volume    = {25},
  number    = {1},
  pages     = {1--86},
  year      = {2015}
}

@article{CK,
  author    = {Cozzi, E. and Kelliher, J. P.},
  title     = {Well-posedness of the {2d} {E}uler equations when velocity grows at infinity},
  journal   = {arXiv preprint arXiv:1709.07422},
  year      = {2017}
}

@article{co2,
  author    = {Di Plinio, F. and Temam, R.},
  title     = {Grisvard's shift theorem near {$L^\infty$} and {Y}udovich theory on polygonal domains},
  journal   = {SIAM J. Math. Anal.},
  volume    = {47},
  number    = {1},
  pages     = {159--178},
  year      = {2015}
}

@article{EJB,
  title     ={Finite-time singularity formation for strong solutions to the Boussinesq system},
  author    ={Elgindi, T. M. and Jeong, I.-J.},
  journal   ={Ann. PDE},
  volume    ={6},
  number    ={1},
  pages     ={5},
  year      ={2020},
  publisher ={Springer}
}

@article{EJS,
  author    = {Elgindi, T. M. and Jeong, I.-J.},
  title     = {Symmetries and critical phenomena in fluids},
  journal   = {Comm. Pure. Applied Math.},
  volume    = {73},
  number    = {2},
  pages     = {257--316},
  year      = {2020}
}

@article{EMS,
  author    = {Elgindi, T. M. and Murray, R. M. and Said, A. R.},
  title     = {On the long-time behavior of scale-invariant solutions to the {2d} {E}uler equation and applications},
  journal   = {Les Annales de l'ENS},
  volume    = {35},
  pages     = {943--970},
  year      = {2025}
}

@book{GT,
  title     ={Elliptic partial differential equations of second order},
  author    ={Gilbarg, D. and Trudinger, N. S.},
  series    ={Classics in Mathematics},
  volume    ={224},
  year      ={2001},
  publisher ={Springer}
}

@book{G1,
  author    = {Grisvard, P.},
  title     = {Elliptic problems in nonsmooth domains},
  series    = {Monographs and Studies in Mathematics},
  volume    = {24},
  publisher = {Pitman (Advanced Publishing Program)},
  address   = {Boston, MA},
  year      = {1985}
}

@article{HZ2,
  author    = {Han, Z. and Zlato{\v{s}}, A.},
  title     = {Euler Equations on General Planar Domains},
  journal   = {Ann. PDE},
  volume    = {7},
  number    = {20},
  year      = {2021}
}

@article{HZ1,
  author    = {Han, Z. and Zlato{\v{s}}, A.},
  title     = {Uniqueness of positive vorticity solutions to the {2d} {E}uler equations on singular domains},
  journal   = {Arch. Ration. Mech. Anal.},
  volume    = {247},
  number    = {84},
  year      = {2023}
}

@article{IMY2,
  author    = {Itoh, T. and Miura, H. and Yoneda, T.},
  title     = {The growth of the vorticity gradient for the two-dimensional {E}uler flows on domains with corners},
  journal   = {arXiv preprint arXiv:1602.00815},
  year      = {2016}
}

@article{IMY1,
  author    = {Itoh, T. and Miura, H. and Yoneda, T.},
  title     = {Remark on {S}ingle {E}xponential {B}ound of the {V}orticity {G}radient for the {T}wo-{D}imensional {E}uler {F}low {A}round a {C}orner},
  journal   = {J. Math. Fluid Mech.},
  volume    = {18},
  number    = {3},
  pages     = {531--537},
  year      = {2016}
}

@article{KS,
  author    = {Kiselev, A. and {\v{S}}ver{\'a}k, V.},
  title     = {Small scale creation for solutions of the incompressible two-dimensional {E}uler equation},
  journal   = {Ann. of Math. (2)},
  volume    = {180},
  number    = {3},
  pages     = {1205--1220},
  year      = {2014}
}

@book{KMR,
  author    = {Kozlov, V. and Maz'ya, V. G. and Rossmann, J.},
  title     = {Spectral problems associated with corner singularities of solutions to elliptic equations},
  series    = {No. 85},
  publisher = {American Mathematical Soc.},
  year      = {2001}
}

@article{co3,
  author    = {Lacave, C.},
  title     = {Uniqueness for two-dimensional incompressible ideal flow on singular domains},
  journal   = {SIAM J. Math. Anal.},
  volume    = {47},
  number    = {2},
  pages     = {1615--1664},
  year      = {2015}
}

@article{co1,
  author    = {Lacave, C. and Miot, E. and Wang, C.},
  title     = {Uniqueness for the two-dimensional {E}uler equations on domains with corners},
  journal   = {Indiana Univ. Math. J.},
  volume    = {63},
  number    = {6},
  pages     = {1725--1756},
  year      = {2014}
}

@book{MB,
  author    = {Majda, A. J. and Bertozzi, A. L.},
  title     = {Vorticity and incompressible flow},
  series    = {Cambridge Texts in Applied Mathematics},
  volume    = {27},
  publisher = {Cambridge University Press},
  address   = {Cambridge},
  year      = {2002}
}

@book{MP,
  author    = {Marchioro, C. and Pulvirenti, M.},
  title     = {Mathematical theory of incompressible nonviscous fluids},
  series    = {Applied Mathematical Sciences},
  volume    = {96},
  publisher = {Springer-Verlag},
  address   = {New York},
  year      = {1994}
}

@article{S1,
  author    = {Serfati, P.},
  title     = {Solutions {$C^\infty$} en temps, {$n$}-{$\log$} {L}ipschitz born\'ees en espace et \'equation d'{E}uler},
  journal   = {C. R. Acad. Sci. Paris S\'er. I Math.},
  volume    = {320},
  number    = {5},
  pages     = {555--558},
  year      = {1995}
}

@article{Y1,
  author    = {Yudovich, V. I.},
  title     = {Non-stationary flows of an ideal incompressible fluid},
  journal   = {Z. Vycisl. Mat. i Mat. Fiz.},
  volume    = {3},
  pages     = {1032--1066},
  year      = {1963}
}

\end{document}